\documentclass[11pt, centertags,oneside]{amsart}

\usepackage{amsmath,amstext,amsthm,amscd,typearea,hyperref,stmaryrd}
\usepackage{amssymb}
\usepackage{a4wide}
\usepackage[mathscr]{eucal}
\usepackage{mathrsfs}
\usepackage{typearea}
\usepackage{charter}
\usepackage{pdfsync}
\usepackage{MnSymbol}
\usepackage[a4paper,width=16.2cm,top=3cm,bottom=3cm]{geometry}

\numberwithin{equation}{section}

\usepackage{xcolor}
\usepackage{tikz}
\usepackage{microtype}

\newtheorem{theorem}{Theorem}[section]
\newtheorem{definition}[theorem]{Definition}
\newtheorem{proposition}[theorem]{Proposition}

\newtheorem{lemma}[theorem]{Lemma}
\newtheorem{remark}[theorem]{Remark}

\newcommand{\cali}[1]{\mathscr{#1}}

\newcommand{\supp}{{\rm supp}}

\newcommand{\dist}{{\rm dist}}

\newcommand{\loc}{{loc}}

\newcommand{\Cc}{\cali{C}}
\newcommand{\Dc}{\cali{D}}

\newcommand{\Fc}{\cali{F}}

\newcommand{\Hc}{\cali{H}}

\newcommand{\Lc}{\cali{L}}

\newcommand{\Rc}{\cali{R}}

\newcommand{\C}{\mathbb{C}}

\newcommand{\N}{\mathbb{N}}
\newcommand{\Z}{\mathbb{Z}}

\newcommand{\R}{\mathbb{R}}

\renewcommand{\P}{\mathbb{P}}

\newcommand{\Xf}{{\mathfrak{X}}}
\newcommand{\Xfh}{{\widehat{\mathfrak{X}}}}

\newcommand{\Kmc}{{\mathcal{K}}}
\newcommand{\Tmc}{{\mathcal{T}}}

\title[]{Intersection of Positive Closed Currents}

\author{Taeyong Ahn}
\address{(Ahn) Department of Mathematics Education, Inha University, 100 Inha-ro, Michuhol-gu, Incheon 22212, Republic of Korea}%
\email{t.ahn@inha.ac.kr}

\subjclass[2020]{32U05, 32U40, 32C30}%
\keywords{pluripotential theory, tangent current, intersection of positive closed currents, self-intersection of analytic subsets}%

\date{\today}

\begin{document}
\begin{abstract}
	We investigate the intersection of positive closed currents in a general setting, employing tangent currents alongside King's residue formula. Our main result establishes a natural condition for the intersection--namely, the Dinh-Sibony product--of positive closed currents on domains and derives an integral representation of this intersection. In parallel, we study the existence, $h$-dimension, and shadow of tangent currents, extending our approach to the study of the self-intersection of analytic subsets. We also present a local version of superpotentials and a regularization of positive closed currents, explore the connections with slicing theory, and examine classical examples. Our work extends to general complex manifolds, including compact K\"ahler manifolds.
\end{abstract}
\maketitle

\section{Introduction} 
Positive closed currents are key tools in complex geometry, with broad applications in complex dynamics and algebraic geometry (e.g., \cite{DemaillyH}, \cite{FSCPLDYN2}, \cite{DS09}). One of the fundamental problems in pluripotential theory is to establish a reasonable definition for the intersection of positive closed currents (e.g., Demailly’s question in \cite{DemaillyG}).\medskip

In this context, due to the application of (quasi-)plurisubharmonic functions, the bidegree $(1, 1)$ case has been intensively studied, notably by Bedford-Taylor, Demailly, Forn{\ae}ss-Sibony, Boucksom-Eyssidieux-Guedj-Zeriahi, Vu, and Huynh-Kaufmann-Vu, among others (\cite{BT}, \cite{Demailly}, \cite{FS}, \cite{BEGZ}, \cite{VU21-1}, \cite{VU21-2}, \cite{HKV}).\medskip

In the higher bidegree case, there had not been much known until Dinh-Sibony brought in remarkable ideas: the theory of superpotentials (\cite{DS06}, \cite{DS09}, \cite{DS10}) and the theory of tangent currents (\cite{DS18}). Dinh-Sibony initiated the idea of superpotentials in \cite{DS06} and used superpotentials to answer Demailly's question raised in \cite{DemaillyG} in the case of complex projective spaces in \cite{DS09}. From the perspective of intersection theory, superpotentials have been further studied by Dinh-Sibony, Vu, Dinh-Nguyen-Vu and Luo-Zhou (\cite{DS10}, \cite{VuMich}, \cite{DNV}, \cite{LZ}). However, the general case is far from being well understood. Then, the notions of the tangent current and the Dinh-Sibony product were introduced in \cite{DS18} and have since been investigated by Kaufmann-Vu, Vu, Hyunh-Kaufmann-Vu, Nguyen, and Nguyen-Truong (\cite{KV}, \cite{VU21-1}, \cite{VU21-2}, \cite{Vu}, \cite{HKV}, \cite{NguyenPositive}, \cite{Nguyen2501}, \cite{NT}). 
While this paper was being prepared, in the preprint \cite{Nguyen2501}, Nguyen introduced the effective sufficient condition for the intersection--the Dinh-Sibony product--of positive closed currents on compact K\"ahler manifolds.\medskip

The primary purpose of this paper is to establish a natural condition for the intersection--the Dinh-Sibony product (Definition \ref{defn:DSproduct})--of positive closed currents of arbitrary bidegree and to derive an integral representation of this intersection in a general setting including compact K\"ahler manifolds. Our approach avoids cohomological techniques and yields a local analytic formula that is amenable to explicit computations, potentially providing an analytic method for the local study of intersections of algebraic cycles (e.g., Remark \ref{rmk:int_for_ex}; see also Theorem \ref{thm:non-proper}). The following is our main result:\medskip

Let $D\subset\C^n$ be a bounded simply connected domain with smooth boundary and $\omega_{\rm euc}$ the Euclidean K\"ahler form on $\C^n$. Let $k\in\Z$ be such that $k\ge 2$. Let $\pi_i:D^k\to D$ be the canonical projection onto the $i$-th factor for $i=1,\ldots, k$. Let $u=\frac{1}{2}\log \sum_{i=1}^{k-1} |x_i-x_k|^2$ and $\Delta=\{(x, \ldots, x)\in D^k: x\in D\}$, where $(x_1, \ldots, x_k)\in D^k$ are the coordinates. For tangent currents, see Section \ref{sec:tangent_currents}. The product $\langle \cdot \rangle_\Kmc$ is defined as in Definition \ref{def:K-product}. We now present our main theorem.
\begin{theorem}\label{thm:main_intersection_1}
	Let $S_i$ for $i=1, \ldots, k$ be positive closed $(s_i, s_i)$-currents on $D$, respectively, where $1\le s:=s_1+\cdots+ s_k\le n$. Suppose that we have
	\begin{align*}
	u\in L_{\loc}^1\left(\left\langle \pi_1^*S_1\wedge\cdots\wedge \pi_{k-1}^*S_{k-1}\wedge\pi_k^*\left(S_k\wedge \omega_{\rm euc}^{kn-s-i+1}\right)\wedge (dd^cu)^{i-1}\right\rangle_\Kmc\right)
	\end{align*}
	in $D^k$ inductively from $i=1$ through $n$. Then, there exists a unique tangent current of $\pi_1^*S_1 \wedge \ldots\wedge \pi_k^*S_k$ along $\Delta$ with the minimal $h$-dimension. Its shadow is exactly the positive closed $(s, s)$-current $\left(S_1\wedge\ldots\wedge S_k\right)_K$ on $D$, where we have
	\begin{align*}
		\left\langle \left(S_1\wedge \cdots\wedge S_k\right)_K, \varphi\right\rangle:=\int_{D^k\setminus \Delta} \pi_1^*S_1 \wedge \cdots \wedge \pi_k^*S_k\wedge\left(u(dd^cu)^{(k-1)n-1}\right)\wedge dd^c\Phi
	\end{align*}
	for a smooth test $(n-s, n-s)$-form $\varphi$ on $D$ and for a smooth $(n-s, n-s)$-form $\Phi$ with compact support on $D^k$ such that $\Phi=\pi_k^*\varphi$ in a neighborhood of $\Delta$. In particular, the Dinh-Sibony product of $S_1, \ldots, S_k$ is well-defined and equal to $\left(S_1\wedge \cdots \wedge S_k\right)_K$.
\end{theorem}
For general complex manifolds, see Theorem \ref{thm:intersection_cpt} and Remark \ref{rmk:singularity_K}. The integrability condition in Theorem \ref{thm:main_intersection_1} is called Condition $(\mathrm{I})$ (Definition \ref{defn:K-integrable_many}). The inductively defined integrability condition in Condition $(\mathrm{I})$ might seem a bit strong at first, but actually, each step lowers the $h$-dimension of the tangent current by one, as shown in Proposition \ref{prop:h-dim}. In particular, Condition $(\mathrm{I})$ implies that the $h$-dimension of the tangent current is minimal. Roughly speaking, Condition $(\mathrm{I})$ is for the inductive definition of the product of $\pi_1^*S_1 \wedge \cdots \wedge \pi_k^*S_k$ with $(dd^c u)^{(k-1)n}$. For the sake of generality, some technicalities have been added as in Lemma \ref{lem:basic_conv_KT}. Note that according to King's residue formula in \cite{King}, we have $(dd^c u)^{(k-1)n}=[\Delta]$ in the sense of currents. In fact, Condition $(\mathrm{I})$ is satisfied in the classical cases discussed in Section \ref{sec:examples}. Under Condition $(\mathrm{I})$, the intersection, in the sense of slicing theory, is also well-defined as proved in Subsection \ref{subsec:slicing}.\medskip

We also extend our approach to the study of the self-intersection of analytic subsets including algebraic cycles. The shadow of tangent currents (Definition \ref{def:shadow}) may be regarded as the largest piece of the intersection of positive closed currents. Theorem \ref{thm:non-proper} shows that the shadow of the tangent current associated with the self-intersection of an analytic subset is the current defined by the analytic subset itself.
\begin{theorem}\label{thm:non-proper}
	Let $X$ be a compact K\"ahler manifold of dimension $n$. Let $A$ be an irreducible analytic subset of pure codimension $a$ in $X$. Let $\Delta$ be the diagonal submanifold of $X^2$. Let $\pi_i:X^2\to X$ denote the canonical projection onto the $i$-th factor for $i=1, 2$. Then, $\pi_1^*[A]\wedge\pi_2^*[A]$ has a unique tangent current along $\Delta$ and its $h$-dimension is $n-a$. Its shadow is exactly $[A]$.
\end{theorem}

To give a flavor of our approach, we consider Theorem \ref{thm:main_intersection_1} with $k=2$.  
We use King's residue formula in \cite{King} to explore the relationship between tangent currents and complex Monge-Amp\`ere type currents. 
When $S_1$ and $S_2$ are smooth, the residue formula says
\begin{align*}
	\left\langle S_1\wedge S_2, \varphi\right\rangle = \int_{{D^2}}\pi_1^*S_1\wedge\pi_2^*(S_2\wedge\varphi)\wedge[\Delta]= \int_{{D^2}}\pi_1^*S_1\wedge\pi_2^*(S_2\wedge\varphi)\wedge (dd^c u)^n,
\end{align*}
where $1\le s:=s_1+s_2\le n$ and $\varphi$ is a smooth test $(n-s, n-s)$-form on $D$. Using the notion of double currents in \cite{Federer}, $\pi_1^*S_1\wedge \pi_2^*S_2$ is well-defined for general positive closed currents $S_1$ and $S_2$. So, if the product of $\pi_1^*S_1\wedge \pi_2^*S_2$ and $(dd^c u)^n$ makes sense, one may expect the resulting product as above to define the intersection of positive closed currents.\medskip

The main difficulty in this approach is that $\Delta = \{u=-\infty\}$ is not compact in ${D^2}$ and is quite big in dimension. It is not clear how the comparison principle or Stokes' theorem can be applied in this setting. Hence, instead, we consider integrability conditions. There are two natural options: one associated with the classical complex Monge-Amp\`ere type product (Subsection \ref{subsec:classical} and Theorem \ref{thm:main_intersection_1}) and the other with the relative non-pluripolar product (Subsection \ref{subsec:general_CondII} and Theorem \ref{thm:non-proper}). Here, we continue to describe our approach with the classical case. Under Condition $(\mathrm{I})$, the product $(S_1\wedge S_2)_K$ of $S_1$ and $S_2$ in terms of $\pi_1^*S_1\wedge\pi_2^*S_2$ and $\left(dd^cu\right)^n$ is well-defined.\medskip

We relate $(S_1\wedge S_2)_K$ and tangent currents of $\pi_1^*S_1\wedge\pi_2^*S_2$ along $\Delta$ to show that the Dinh-Sibony product is well-defined. We briefly summarize the notion of tangent currents in our setting (Definition \ref{def:tangent_current}). We may assume that ${D^2}$ is a subset of $E$, where $E:=\Delta\times \C^n$ is the normal bundle of $\Delta$ in ${D^2}$. Let $\overline{E}:=\Delta\times\P^n$ be the projective compactification of $E$. Let $\pi_\Delta: \overline{E}\to \Delta$ and $\pi_F:\overline{E}\to \P^n$ denote the canonical projections onto $\Delta$ and $\P^n$, respectively. Let $A_\lambda:\overline{E}\to \overline{E}$ be the extension to $\overline{E}$ of the multiplication by $\lambda$ on fibers of $E$ for $\lambda\in\C^*$. A tangent current $(\pi_1^*S_1\wedge \pi_2^*S_2)_\infty$ of $\pi_1^*S_1\wedge\pi_2^*S_2$ along $\Delta$ is a limit current of the family $\left(\left(A_\lambda\right)_*\left(\mathbf{1}_{D^2}\pi_1^*S_1\wedge \pi_2^*S_2\right)\right)_{\lambda\in\C^*}$ of currents on $E$. Under Condition $(\mathrm{I})$, tangent currents exist as in Section \ref{sec:integrability}. Let $\left(\lambda_k\right)_{k\in\N}$ be a sequence in $\C^*$ such that $\lim_{k\to\infty}\lambda_k=\infty$ and that $(\pi_1^*S_1\wedge \pi_2^*S_2)_\infty=\lim_{k\to\infty}\left(A_{\lambda_k}\right)_*(\mathbf{1}_{D^2}\pi_1^*S_1\wedge \pi_2^*S_2)$ in the sense of currents on $E$.
\medskip

The point in connecting $(\pi_1^*S_1\wedge \pi_2^*S_2)_\infty$ to $(S_1\wedge S_2)_K$ is that we utilize the independence property involved in defining the shadow of tangent currents as in Remark \ref{rmk:indep_shadow}. We use the form $\Omega_n$ as in Lemma \ref{lem:useful_form}, which is a smooth positive closed form of maximal bidegree with compact support in $\C^n\subset\P^n$ cohomologous to a linear subspace. The current $(dd^cu)^n$ is related to $\Omega_n$ via change of coordinates.  
We approximate $(dd^c u)^n$ by 
smooth positive closed $(n, n)$-forms $\left(\Kmc_\theta^n\right)_{0<|\theta|\ll 1}$ defined by $\Kmc^n_{\theta}=dd^c\chi\left(u+\log\left|\theta\right|\right) \wedge(dd^cu)^{n-1}$, where $\theta\in\C^*$ and $\chi:\R\to \R_{\ge 0}$ is the smooth convex increasing function as in the definition of the form $\Omega_n$ in Lemma \ref{lem:useful_form}. On ${D^2}\subset \overline{E}$, the change of variables in the fiber direction in $E$ gives us
\begin{align*}
	&\left(A_{\lambda_k}\right)^*\left(\pi_F^*\Omega_n\right)
	=dd^c\chi\left(\left(\log\left|\lambda_k(x-y)\right|\right)+M\right) \wedge(dd^c\log|\lambda_k\left(x-y\right)|)^{n-1}\\
	\notag&\quad=dd^c\chi\left(\left(\log\left|x-y\right|\right)+\log\left|e^M\lambda_k\right|\right) \wedge(dd^c\log|x-y|)^{n-1}=\Kmc^n_{e^{-M}\lambda_k^{-1}}.
\end{align*}
Hence, for a smooth test form $\varphi$ on $D$, since $\supp\, \left(\pi_F^*\Omega_n\wedge\pi_\Delta^*\varphi\right)$ is compact in $E$, the shadow $(\pi_1^*S_1\wedge \pi_2^*S_2)_\infty^h$ of $(\pi_1^*S_1\wedge \pi_2^*S_2)_\infty$ is as follows: 
\begin{align*}
	\left\langle\left(\pi_1^*S_1\wedge \pi_2^*S_2\right)_\infty^h, \varphi\right\rangle&=\left\langle\left(\pi_\Delta\right)_*\left(\left(\pi_1^*S_1\wedge \pi_2^*S_2\right)_\infty\wedge \pi_F^*\Omega_n\right), \varphi\right\rangle\\
	&= \lim_{k\to\infty}\left\langle \left(A_{\lambda_k}\right)_*\left(\mathbf{1}_{D^2}\pi_1^*S_1\wedge \pi_2^*S_2\right)\wedge \pi_F^*\Omega_n, \pi_\Delta^*\varphi \right\rangle\\
	&= \lim_{k\to\infty}\int_{D^2}\pi_1^*S_1\wedge \pi_2^*S_2\wedge \Kmc_{e^{-M}\lambda_k^{-1}}^n\wedge\pi_2^*\varphi=\left\langle (S_1\wedge S_2)_K, \varphi \right\rangle.
\end{align*}
Here, we identify $\Delta$ with $D$, and the last convergence is from the integrability condition. We see that $(\pi_1^*S_1\wedge \pi_2^*S_2)_\infty^h$ coincides with $(S_1\wedge S_2)_K$. In particular, all the tangent currents have the same shadow. As the $h$-dimension of $(\pi_1^*S_1\wedge \pi_2^*S_2)_\infty$ is minimal, there exists a unique tangent current according to Proposition \ref{prop:uniqueness_from_shadow}. Hence, the Dinh-Sibony product of $S_1$ and $S_2$ is well-defined and equal to $(S_1\wedge S_2)_K$.
\medskip

The above framework works for general tangent currents. Indeed, we obtain Theorem \ref{thm:main_intersection_1} as a particular case of Theorem \ref{thm:uniqueness_tangent_current}, which is proved in Sections \ref{sec:tangent_currents} and \ref{sec:integrability}. More precisely, we relate tangent currents and complex Monge-Amp\`ere type currents as in Proposition \ref{prop:main_true_general}, and as applications, we present: a characterization of the $h$-dimension of tangent currents and a condition for the existence of tangent currents. In Section \ref{sec:integrability}, integrability conditions are introduced. In particular, together with the results in Section \ref{sec:tangent_currents}, from the integrability condition, Condition $(\mathrm{K}-\max)$, we deduce the existence and uniqueness of tangent current with the minimal $h$-dimension and describe the shadow of the tangent current via an integral representation, which proves Theorem \ref{thm:uniqueness_tangent_current}. In Section \ref{sec:intersection}, we apply the results to the intersection of positive closed currents and obtain Theorem \ref{thm:main_intersection_1}. Due to the invariant nature of tangent currents, our results extend to general complex manifolds as in Section \ref{sec:compact}. In addition, we present a local version of superpotentials. In Section \ref{sec:regularization}, regularizations of positive closed currents are discussed, and in particular, we show that the standard regularization by convolution works well with the Dinh-Sibony product under Condition $(\mathrm{I})$, which means that the Dinh-Sibony product, together with Condition $(\mathrm{I})$, generalizes King's work in \cite{King-1}. We also consider slicing theory. In Section \ref{sec:examples}, we examine classical examples. Finally, in Section \ref{sec:self-intersecting}, we investigate the self-intersection of analytic subsets, proving Theorem \ref{thm:non-proper}.
\medskip

For the intersection on compact K\"ahler manifolds, Nguyen used generalized Lelong numbers to study tangent currents and proved several Lelong-Jensen type formulas for estimates in \cite{Nguyen2501}; we use King's residue formula to define the product of a given current and a given submanifold, and then directly establish the relationship between tangent currents and the product current as described above. In essence, closedness and global nature are pursued in \cite{Nguyen2501} while positivity and local nature in this work. In this sense, for compact K\"ahler manifolds, the two works employ different approaches but arrive at essentially the same conclusion regarding the condition for the intersection--the Dinh-Sibony product--of positive closed currents (cf. Theorem \ref{thm:intersection_cpt} and \cite[Theorem 2.18]{Nguyen2501}) and the characterization of the $h$-dimension of tangent currents (cf. Theorem \ref{thm:h_dim_K} and \cite[Theorem 2.12]{Nguyen2501}). Further, we hope that our approach can provide a different perspective to the study of generalized Lelong numbers.
\medskip

\noindent {\bf Notations.} We use a smooth convex increasing function $\chi:\R\to \R_{\ge 0}$ such that $\chi(t)=0$ if $t\le -1$ and $\chi(t)=t$ if $t\ge 1$. For a set $A$ and $\epsilon>0$, we denote by $A_\epsilon$ the $\epsilon$-neighborhood of $A$ with respect to the given distance. We use $\|\cdot\|_\infty$ (resp., $\|\cdot\|_{C^\alpha}$, $\|\cdot\|_{L^p}$) to denote the uniform (resp., $C^\alpha$, $L^p$) norm of a given function. By each norm applied to a form, we mean the supremum of the norms of its coefficients. The norms $\|\cdot\|_{\infty, U}$,  $\|\cdot\|_{C^\alpha, U}$ and $\|\cdot\|_{L^p, U}$ refer to the corresponding norms over the set $U$, respectively. For a complex manifold $X$, we denote by $\Cc_k(X)$ the set of positive closed $(k, k)$-currents on $X$. For an analytic subset $V$, we write $[V]$ for the current of integration on its regular part.
For a positive current $S$, we denote by $\|S\|$ the mass of the current $S$ and for a Borel set $U$, by $\|S\|_U$ the mass of $S$ over $U$.
\medskip

\noindent
\textbf{Acknowledgments.} 
The research of the author was supported by the National Research Foundation of Korea (NRF) grant funded by the Korea government (MSIT) (No. RS-2023-00250685). The paper was partially prepared during the visit of the author at the University of Michigan, Ann Arbor. He would
like to express his gratitude to the organization for hospitality. The author would like to give thanks to Dan Burns and Mattias Jonsson for constructive discussions and for their support. The author would like to thank Tien-Cuong Dinh, Viet-Anh Nguyen, Duc-Viet Vu and Sungmin Yoo for their valuable comments. 

\section{Tangent Currents}\label{sec:tangent_currents}

The tangent current was introduced by Dinh-Sibony (\cite{DS18}) and further studied by Vu (\cite{Vu}) and Nguyen (\cite{NguyenPositive},\cite{Nguyen2501}). In this section, we relate tangent currents to complex Monge-Amp\`ere type currents as in Proposition \ref{prop:main_true_general} and use this relationship to study the $h$-dimension and existence of tangent currents on domains.\medskip

Let $X$ be a complex manifold of dimension $N$ and $V\subset X$ a complex submanifold of codimension $n$. Let $T\in\Cc_p(X)$. We may assume that $T$ has no mass on $V$ as the El Mir-Skoda theorem says that every positive closed current on $X$ can be decomposed into the sum of a positive closed current on $X$ with support in $V$ and another positive closed current on $X$ having no mass on $V$. Let $E$ be the normal bundle of $V$ in $X$ and $\overline E:=\P(E\oplus \C)$ the projective compactification of $E$. Let $\pi_V:\overline{E}\to V$ denote the canonical projection. The hypersurface at infinity $H_\infty:=\overline E\setminus E$ of $\overline E$ is isomorphic to $\P(E)$ as a fiber bundle over $V$. We also have another canonical projection $\pi_\infty:\overline{E}\setminus V\to H_\infty$.
\begin{definition}
	Let $U$ be an open subset of $X$ with $U\cap V\ne \emptyset$. A (local) holomorphic admissible map is a biholomorphism $\tau$ from $U$ to an open neighborhood of $U\cap V$ in $E$, which is the identity on $U\cap V$, and the restriction of whose differential $d\tau$ to ${U\cap V}$ is the identity. 
\end{definition}
The adjective ``local'' means that it may not be defined in a global neighborhood of the entire submanifold $V$ in $X$ but is defined in the neighborhood $U$ of $U\cap V$. The admissible maps in \cite{DS18} are global but may not be holomorphic. In this work, we use local holomorphic admissible maps. For simplicity, we will drop ``local'' and just say a ``holomorphic admissible map.'' 
\medskip

For $\lambda\in\C^*$, let $A_\lambda:E\to E$ be the multiplication by $\lambda$ on fibers of $E$. It can be extended to $\overline{E}$. The following definition of tangent current is the version in \cite{Vu}. See also \cite{DS18}.
\begin{definition}[Definition 2.1 in \cite{Vu}]\label{def:tangent_current}
	A tangent current $T_\infty$ of $T$ along $V$ is a positive closed current on $E$ such that there exists a sequence $(\lambda_k)_{k\in\N}\subset \C^*$ converging to $\infty$ and a collection of holomorphic admissible maps $\tau_i:U_i\to E$ for $i\in I$ satisfying the following two properties:
	\begin{align*}
		&(i) V\subset \bigcup_{i\in I} U_i\\
		&(ii) T_\infty :=\lim_{k\to\infty} \left(A_{\lambda_k}\right)_*(\tau_i)_*T \,\,\textrm{ on }\,\pi_V^{-1}(U_i\cap V)\setminus H_\infty\,\,\textrm{ for every }\,i\in I.
	\end{align*}
	A tangent current $T_\infty$ trivially extends to $\overline{E}$. We still denote it by $T_\infty$.
	For an open subset $V_0$ of $V$, the horizontal dimension (or the $h$-dimension for short) of $T_\infty$ over $V_0$ is the largest integer $h_T$ such that $T_\infty\wedge \pi_V^*(\omega_V^{h_T})\ne 0$ on $\pi_V^{-1}(V_0)$, where $\omega_V$ is a K\"ahler form on $V$. The $h$-dimension of $T_\infty$ is its $h$-dimension over $V$.
\end{definition}

\begin{remark}
	If $h_T$ is the $h$-dimension of $T_\infty$, then we have $\max\{N-p-n, 0\}\le h_T\le \min\{N-p, N-n\}$. The $h$-dimension $h_T$ of $T_\infty$ is said to be minimal, if $h_T=\max\{N-p-n, 0\}$.
\end{remark}
\begin{remark}\label{rmk:independence_of_admissible_map}
	Thanks to \cite[Proposition 2.5]{KV}, tangent currents are independent of the choice of holomorphic admissible maps $(\tau_i)_{i\in I}$. In particular, they are well-defined on any intersection $\pi_V^{-1}(U_i\cap V)\cap \pi_V^{-1}(U_j\cap V)$ whenever the limits exist.
\end{remark}

\begin{definition}[Definition 3.6 in \cite{DS18}]\label{def:shadow}
	Let $T_\infty$ be a tangent current of $T$ along $V$ with its $h$-dimension $h_T$ over an open subset $V_0$ in $V$. Let $\Omega$ be a smooth closed $\left(N-p-h_T, N-p-h_T\right)$-form on $\pi_V^{-1}(V_0)$ whose restriction to each fiber of $\pi_V$ is cohomologous to a linear subspace in this fiber. The shadow of $T_\infty$ on $V_0$ is the positive closed $(N-n-h_T, N-n-h_T)$-current $T_\infty^h:=\left(\pi_V\right)_*\left(T_\infty\wedge\Omega\right)$ with support in $\pi_V\left(\supp\left(T_\infty\right)\right)\cap V_0$. The shadow of $T_\infty$ is defined to be the shadow on $V$.
\end{definition}

The following property is important in our approach.
\begin{remark}\label{rmk:indep_shadow}
	\cite[Proposition 3.5]{DS18} says that the shadow $T_\infty^h$ of $T_\infty$ is independent of the choice of $\Omega$. 
\end{remark}

In the rest of this section, based on Remark \ref{rmk:independence_of_admissible_map}, we consider the following case: a domain $U\subset\C^N$ and a complex linear submanifold $V\subset U$ of codimension $n$ such that $V$ is bounded and simply connected with smooth boundary. General complex manifolds will be treated in Section \ref{sec:compact}. 
\medskip

We will use the coordinates $x=(x', x'')\in U$ such that $x'=(x_1, \ldots, x_{N-n})$, $x''=(x_{N-n+1}, \ldots, x_N)$ and $V=\{x\in U: x''=0\}$. The projective compactification of the normal bundle of $V$ in $U$ can be written as $\overline{E}=V\times\P^n$ with the canonical projection map $\pi_V:\overline{E}\to V$ of $\overline{E}$ onto $V$. Also, in our case, it makes sense to consider the projection onto the fiber space $\P^n$, which we denote by $\pi_F:\overline{E}\to\P^n$. We can extend the coordinates $(x', x'')$ to $V\times\P^n$ in the following way. We use the coordinates $\left(x', [x'': t]\right)$ for a point in $V\times \P^n$. We identify $(x', x'')\in V\times \C^n=E$ with $(x', [x'': 1])\in V\times \P^n=\overline{E}$. 
Hence, we may consider $U$ as a subset of $\overline{E}$ and the holomorphic admissible map $\tau: U \to \overline{E}$ becomes the inclusion map. In particular $\tau_*T=\mathbf{1}_UT$ in $\overline{E}$ for $T\in\Cc_p(U)$.\medskip 

When $n<N$, we use the K\"ahler form $\omega_V=dd^c|x'|^2$ on $V$ and the Fubini-Study form $\omega_F=\frac{1}{2}dd^c\log\left(|x''|^2+|t|^2\right)$ on each fiber $\P^n$ of $\pi_V:\overline{E}\to V$. Then, $\omega:=\pi_V^*\omega_V+\pi_F^*\omega_F$ defines a K\"ahler form on $\overline{E}$. When $n=N$, we just use $\omega:=\pi_F^*\omega_F$.\medskip

Together with $\omega_F^i$, we will also use the form $\Omega_i$ on the fiber space $\P^n$ for $i=1, \ldots, n$, defined below, which is one of the main ingredients in this work. 
\begin{lemma}
	\label{lem:useful_form} Let $M>0$ be a sufficiently large fixed real number. Let $[x'':t]$ denote the homogeneous coordinates for $\P^n$. We identify $x''\in\C^n$ with $[x'':1]\in\P^n$. Let $\alpha$ be a smooth positive closed $(1, 1)$-form on $\P^n$ defined by 
	\begin{align*}
		\alpha:=\left\{\begin{array}{ll}
			dd^c\chi\left(\left(\log\left|x''\right|\right)+M\right) & \textrm{ for } x''\in \C^n\\
			dd^c\log\left|x''\right| & \textrm{ for } [x'':t] \textrm{ in a neighborhood of } \P^n\setminus \C^n \textrm{ in }\P^n.
		\end{array}\right.
	\end{align*}
	For each $i=1, \ldots, n$, let $\Omega_i$ be a smooth form defined by
	\begin{align*}
		\Omega_i:=\alpha\wedge\left(dd^c\log\left|x''\right|\right)^{i-1}.
	\end{align*}
	Then, $\Omega_i$ is a smooth positive closed $(i, i)$-form in the cohomology class $\left\{\omega_F^i\right\}$ with support away from the set $\{x''=0\}$. In particular, $\Omega_n$ has compact support in $\C^n \subset\P^n$.
\end{lemma}
	
\begin{proof}
	For $x''\in\C^n$ with $|x''|>e^{1-M}$, $\alpha = dd^c\log|x''|$. Hence, $\alpha$ is a well-defined smooth $(1, 1)$-form on $\P^n$. Positivity and closedness are local properties and therefore, $\Omega_i$ is positive, closed and of bidegree $(i, i)$. Notice that $\alpha=0$ if $x''\in\C^n$ satisfies $|x''|<e^{-1-M}$. So, $\Omega_i$ has support away from the set $\{x''=0\}$. From direct computations, one can check that $\Omega_i\in\left\{\omega_F^i\right\}$ and that $\Omega_n$ has compact support in $\C^n \subset\P^n$.
\end{proof}	

We introduce some more functions and forms. We define
\begin{align*}
	u(x)=\log |x''|=\frac{1}{2}\log\left(\sum_{i=N-n+1}^N|x_i|^2\right)
\end{align*}
on $U$. King's residue formula in \cite{King} says that we have 
\begin{align}\label{eq:King_residue}
	(dd^c u)^n=[V]
\end{align}
in the sense of currents. For $\theta\in\C^*$ with $|\theta|\ll 1$, we define 
$$u^\Kmc_\theta:=\chi(u-\log |\theta|)+\log |\theta|\quad\textrm{ and }\quad u^\Tmc_\theta:=\frac{1}{2}\log\left(|x''|^2+|\theta|^2\right).$$
Both $u^\Kmc_\theta$ and $u^\Tmc_\theta$ decreasingly converge to $u$ as $|\theta|\rightarrow 0$. With these approximating functions, we consider related approximations of $(dd^c u)^i$ for $i=1, \cdots, n$.
\begin{enumerate}
	\item $\Kmc^i_\theta:=(dd^c u^\Kmc_\theta)\wedge(dd^c u)^{i-1}$ for $i=1, \ldots, n$;
	\item $\Tmc^i_\theta:=(dd^c u^\Tmc_\theta)^i$ for $i=1, \ldots, n$.\medskip
\end{enumerate}

The following proposition best describes our approach. We relate tangent currents to complex Monge-Amp\`ere type currents.

\begin{proposition}\label{prop:main_true_general}
Let $T\in\Cc_p(U)$ admit a tangent current $T_\infty$ along $V$. Let $m\in\{\max\{n-p, 0\}, \ldots, \min\{N-p, n\}\}$. Suppose that the family $\left(T\wedge \Kmc_\theta^m\wedge \pi_V^*\omega_V^{N-p-m}\right)_{0<|\theta|\ll 1}$ of currents of order $0$ of maximal bidegree has locally uniformly bounded mass in $U$. Then, for every smooth test $(N-p-m, N-p-m)$-form $\varphi$ on $V$, there exists a sequence $\left(\theta_k\right)_{k\in\N}\subset\C^*$ converging to $0$ such that the sequence $\left(T\wedge \Kmc_{\theta_k}^m\wedge \pi_V^*\varphi\right)_{k\in\N}$ converges to a current $L_\varphi$ of order $0$ of maximal bidegree satisfying
	\begin{align}\label{eq:limit_vs_shadow}
		\left\langle\mathbf{1}_VL_\varphi, 1\right\rangle=\left\langle(\pi_V)_*\left(T_\infty\wedge \pi_F^*\Omega_m\right), \varphi\right\rangle,
	\end{align}
	where $\Omega_m$ is the smooth positive closed $(m, m)$-form as in Lemma \ref{lem:useful_form}.
\end{proposition}
\begin{remark}
	 In Proposition \ref{prop:main_true_general}, the case $n=N$ is allowed. When $n=N$, the constraint on $m$ forces $N-p-m=0$ and we can take as $\varphi$ a function. Then, the same proof works.
\end{remark}

\begin{remark}
	When $m=N-p-h_T$ for the $h$-dimension $h_T$ of $T_\infty$, the form $\Omega_{N-p-h_T}$ in the right hand side of \eqref{eq:limit_vs_shadow} can be replaced with any smooth closed $(N-p-h_T, N-p-h_T)$-form in the de Rham cohomology class $\{\omega_F^{N-p-h_T}\}$ as in Remark \ref{rmk:indep_shadow}. However, this may not be true if $m<N-p-h_T$.
\end{remark}

\begin{proof}	
	We may suppose that $\varphi$ is positive as any smooth test form can be written as a linear combination of positive smooth test forms. Let $\left(\lambda_k\right)_{k\in\N}\subset\C^*$ be a sequence such that $\left(A_{\lambda_k}\right)_*\left(\mathbf{1}_UT\right)\to T_\infty$ in $E$ and $\lambda_k\to\infty$ as $k\to\infty$. Let $M>0$ be the constant as in Lemma \ref{lem:useful_form}. From the hypothesis, the sequence $\left(T\wedge \Kmc_{e^{-M}\lambda_k^{-1}}^m\wedge\pi_V^*\varphi\right)_{k\in\N}$ has locally uniformly bounded mass, and therefore, there exists a subsequence $(\lambda_{k_l})$ such that $\left(T\wedge \Kmc_{e^{-M}\lambda_{k_l}^{-1}}^m\wedge\pi_V^*\varphi\right)_{l\in\N}$ converges in $U$ in the sense of currents. By passing to this convergent subsequence, we assume that $\left(T\wedge \Kmc_{e^{-M}\lambda_k^{-1}}^m\wedge\pi_V^*\varphi\right)_{k\in\N}$ converges and let $L_\varphi$ denote the limit current, which is positive and therefore, of order $0$. Notice that we have $\mathbf{1}_{U\setminus V}L_\varphi=\mathbf{1}_{U\setminus V}\left(T\wedge (dd^cu)^m\wedge\pi_V^*\varphi\right)$ as $dd^cu$ is smooth outside $V$.
	\medskip
	
	We first prove $\left\langle\mathbf{1}_VL_\varphi, 1\right\rangle\le\left\langle(\pi_V)_*\left(T_\infty\wedge \pi_F^*\Omega_m\right), \varphi\right\rangle$.
	Let $W\subset V$ be a relatively compact open subset of $V$ such that $\supp\,\varphi\subseteq W$. Let $\chi_F:\C^n\to [0,1]$ be a smooth function with compact support defined on the fiber space such that $\chi_F\equiv 1$ in a neighborhood of $0\in \C^n$ and that $\overline{W \times \supp\,\chi_F}\subset U$. It is not difficult to see from Definition \ref{def:tangent_current} that in $E\cap \pi_V^{-1}(W)$, the sequence $\left(\left(A_{\lambda_k}\right)_*(\left(\pi_F^*\chi_F\right)\mathbf{1}_UT)\right)_{k\in\N}$ converges to the same tangent current $T_\infty$.
	\medskip
	
	Over $U$, for each $k\in\N$, we have
	\begin{align}
		\label{eq:key_idea}		&\left(A_{\lambda_k}\right)^*(\pi_F^*\Omega_m)
		=dd^c\chi\left(\left(\log\left|\lambda_kx''\right|\right)+M\right) \wedge(dd^c\log|\lambda_kx''|)^{i-1}\\
		\notag&\quad\quad\quad=dd^c\chi\left(\log\left|x''\right|+\log\left|e^M\lambda_k\right|\right) \wedge(dd^c\log|x''|)^{i-1}\\
		\notag&\quad\quad\quad=dd^c\left(\chi\left(\log\left|x''\right|+\log\left|e^M\lambda_k\right|\right)-\log\left|e^M\lambda_k\right|\right) \wedge(dd^c\log|x''|)^{i-1}=\Kmc^m_{e^{-M}\lambda_k^{-1}}.
	\end{align}
	Hence, for each $k\in\N$, we have
	\begin{align*}
		&\int_{\overline{E}}\left(A_{\lambda_k}\right)_*\left(\left(\pi_F^*\chi_F\right)\mathbf{1}_UT\right)\wedge\pi_F^*\Omega_m\wedge\pi_V^*\varphi=\int_{\overline{E}}\left(\pi_F^*\chi_F\right)\mathbf{1}_UT\wedge \left(A_{\lambda_k}\right)^*\left(\pi_F^*\Omega_m\right)\wedge\pi_V^*\varphi\\
		&\quad\quad\quad=\int_{U}\left(\pi_F^*\chi_F\right)T\wedge\left(A_{\lambda_k}\right)^*\left(\pi_F^*\Omega_m\right)\wedge\pi_V^*\varphi=\int_{U}T\wedge \Kmc^m_{e^{-M}\lambda_k^{-1}} \wedge\left[\left(\pi_F^*\chi_F\right)\left(\pi_V^*\varphi\right)\right].
	\end{align*}
	Let $\varepsilon>0$ be given such that $\varepsilon\ll 1$. Let $\chi^{\infty}_\varepsilon:\P^n\to [0,1]$ be a smooth function such that $\chi^\infty_\varepsilon\equiv 1$ in a neighborhood of $\P^n\setminus \C^n$ and its support sits inside the $\varepsilon$-neighborhood of $\P^n\setminus \C^n$, where the distance is with respect to the norm associated with the form $\omega_F$. Then, we can write
	\begin{align}
		\label{eq:1_total}&\int_{U}T\wedge \Kmc^m_{e^{-M}\lambda_k^{-1}} \wedge\left[\left(\pi_F^*\chi_F\right)\left(\pi_V^*\varphi\right)\right]=\int_{\overline{E}}\left(A_{\lambda_k}\right)_*\left(\left(\pi_F^*\chi_F\right)\mathbf{1}_UT\right)\wedge\pi_F^*\Omega_m\wedge\pi_V^*\varphi\\
		\label{eq:1_near_infty}&=\int_{\overline{E}}\pi_F^*\chi_\varepsilon^\infty\left(A_{\lambda_k}\right)_*\left(\left(\pi_F^*\chi_F\right)\mathbf{1}_UT\right)\wedge\pi_F^*\Omega_m\wedge\pi_V^*\varphi\\
		\label{eq:1_off_infty}&\quad\quad\quad\quad\quad+\int_{\overline{E}}(1-\pi_F^*\chi_\varepsilon^\infty)\left(A_{\lambda_k}\right)_*\left(\left(\pi_F^*\chi_F\right)\mathbf{1}_UT\right)\wedge\pi_F^*\Omega_m\wedge\pi_V^*\varphi.
	\end{align}
	For \eqref{eq:1_near_infty}, we have
	\begin{align*}
		\int_{\overline{E}}\pi_F^*\chi_\varepsilon^\infty\left(A_{\lambda_k}\right)_*\left(\left(\pi_F^*\chi_F\right)\mathbf{1}_UT\right)\wedge\pi_F^*\Omega_m\wedge\pi_V^*\varphi=\int_{\overline{E}}\left(\pi_F^*\chi_F\right)\mathbf{1}_UT\wedge\left(A_{\lambda_k}\right)^*\left(\pi_F^*\left(\chi_\varepsilon^\infty\Omega_m\right)\right)\wedge\pi_V^*\varphi.
	\end{align*}
	We have $\pi_F^*\Omega_m=(dd^cu)^m$ on the support of $\pi_F^*\chi_\varepsilon^\infty$ and $\left(A_{\lambda_k}\right)^*(dd^cu)^m=(dd^cu)^m$ on $\overline{E}\setminus V$. Hence, we get
	\begin{align*}
		&\int_{\overline{E}}\pi_F^*\chi_\varepsilon^\infty\left(A_{\lambda_k}\right)_*\left(\left(\pi_F^*\chi_F\right)\mathbf{1}_UT\right)\wedge\pi_F^*\Omega_m\wedge\pi_V^*\varphi\\
		&\quad\quad\quad\quad\quad=\int_{A_{\lambda_k}^{-1}\left(\supp\,\left(\pi_F^*\chi_\varepsilon^\infty\right)\right)}\left(\left(\pi_F^*\chi_\varepsilon^\infty\right)\circ A_{\lambda_k}\right)\mathbf{1}_UT\wedge (dd^cu)^m\wedge \left[\left(\pi_F^*\chi_F\right)\left(\pi_V^*\varphi\right)\right].
	\end{align*}
	Since $0\le \left(\pi_F^*\chi_\varepsilon^\infty\right)\circ A_{\lambda_k}\le 1$ and $A_{\lambda_k}^{-1}\left(\supp\,\left(\pi_F^*\chi_\varepsilon^\infty\right)\right)\cap V=\emptyset$ for all $k\in\N$, positivity implies
	\begin{align*}
		\int_{\overline{E}}\pi_F^*\chi_\varepsilon^\infty\left(A_{\lambda_k}\right)_*\left(\left(\pi_F^*\chi_F\right)\mathbf{1}_UT\right)\wedge\pi_F^*\Omega_m\wedge\pi_V^*\varphi\le \int_{U\setminus V}\left(\pi_F^*\chi_F\right)T\wedge (dd^cu)^m\wedge \pi_V^*\varphi.
	\end{align*}
	Since $\left(T\wedge \Kmc^m_{e^{-M}\lambda_k^{-1}}\wedge\pi_V^*\varphi\right)_{k\in\N}$ converges in $U$ and the support of $\left(1-\pi_F^*\chi_\varepsilon^\infty\right)\left(\pi_V^*\varphi\right)$ is compact, we let $k\to\infty$ to get from \eqref{eq:1_total} that
	\begin{align*}
		\int_U \left(\pi_F^*\chi_F\right)L_\varphi&\le \int_{U\setminus V}\left(\pi_F^*\chi_F\right)T\wedge(dd^cu)^m\wedge \pi_V^*\varphi+ \int_{\overline{E}} \left(1-\pi_F^*\chi_\varepsilon^\infty\right)T_\infty\wedge \pi_F^*\Omega_m\wedge \pi_V^*\varphi.
	\end{align*}
	The convergence of $\left(\left(A_{\lambda_k}\right)_*\left(\left(\pi_F^*\chi_F\right)\mathbf{1}_UT\right)\right)_{k\in\N}$ to $T_\infty$ in \eqref{eq:1_off_infty} is from the discussion in the above. Letting $\varepsilon \to 0$ and applying the argument in the beginning of the proof, we have
	\begin{align}\label{eq:2_oneside}
		\int_U \mathbf{1}_VL_\varphi \le \int_E T_\infty\wedge \pi_F^*\Omega_m\wedge \pi_V^*\varphi.
	\end{align}
	For $T_\infty$ has no mass on $\overline{E}\setminus E$.
	\medskip
	
	We prove the other direction. As $T_\infty$ trivially extends across $H_\infty=\overline{E}\setminus E$ to $\overline{E}$, we have
	\begin{align*}
		\int_E T_\infty\wedge\pi_F^*\Omega_m\wedge\pi_V^*\varphi &= \lim_{\varepsilon\to 0}\int_{\overline{E}}(1-\pi_F^*\chi_\varepsilon^\infty) T_\infty\wedge\pi_F^*\Omega_m\wedge\pi_V^*\varphi\\
		&=\lim_{\varepsilon\to 0}\lim_{k\to\infty}\int_{\overline{E}}(1-\pi_F^*\chi_\varepsilon^\infty)\left(A_{\lambda_k}\right)_*\left(\mathbf{1}_UT\right)\wedge\pi_F^*\Omega_m\wedge\pi_V^*\varphi.
	\end{align*}
	Let $\delta>0$ be given. Let $\chi_\delta^0:U \to [0,1]$ be another smooth function with compact support such that $\chi_\delta^0\equiv 1$ in a neighborhood of $\overline{W}$ in $U$ and that $\supp\,\chi_\delta^0$ lies inside the $\delta$-neighborhood of $V$, where the distance is with respect to the Euclidean metric. Then, arguing as previously, we have
	\begin{align*}
		&\int_{\overline{E}}(1-\pi_F^*\chi_\varepsilon^\infty)\left(A_{\lambda_k}\right)_*\left(\mathbf{1}_UT\right)\wedge\pi_F^*\Omega_m\wedge\pi_V^*\varphi=\int_{\overline{E}}(1-\left(\pi_F^*\chi_\varepsilon^\infty\right)\circ A_{\lambda_k})\left(\mathbf{1}_UT\right)\wedge\Kmc^m_{e^{-M}\lambda_k^{-1}}\wedge \pi_V^*\varphi\\
		&\quad\quad\quad\quad\quad\le \int_{\overline{E}}\chi_\delta^0\mathbf{1}_UT\wedge \Kmc^m_{e^{-M}\lambda_k^{-1}} \wedge \pi_V^*\varphi = \int_U\chi_\delta^0T\wedge \Kmc^m_{e^{-M}\lambda_k^{-1}} \wedge \pi_V^*\varphi
	\end{align*}
	for all sufficiently large $k\in\N$ as the support of $(1-\left(\pi_F^*\chi_\varepsilon^\infty\right)\circ A_{\lambda_k})\pi_V^*\varphi$ shrinks to $W$. Letting $k\to\infty$ and $\varepsilon\to 0$ in this order, we obtain
	\begin{align*}
		\int_E T_\infty\wedge\pi_F^*\Omega_m\wedge\pi_V^*\varphi\le \int_U\chi_\delta^0 L_\varphi
	\end{align*}
	for all sufficiently small $\delta>0$. Finally, we let $\delta\to 0$ and we get
	\begin{align}\label{eq:2_the_other_side}
		\int_E T_\infty\wedge\pi_F^*\Omega_m\wedge\pi_V^*\varphi\le \int_U\mathbf{1}_V L_\varphi.
	\end{align}
	Take $\theta_k={e^{-M}\lambda_k^{-1}}$ for $k\in\N$ and the inequalities \eqref{eq:2_oneside} and \eqref{eq:2_the_other_side} conclude the desired equality.
\end{proof}

Under a slightly stronger condition, a similar proof to that of Proposition \ref{prop:main_true_general} gives a more explicit characterization as below. Because of the similarity, we omit the proof.
\begin{proposition}\label{prop:main_general}
	Let $T\in\Cc_p(U)$ admit a tangent current $T_\infty$ along $V$. Let $m\in\{0, \ldots, \min\{N-p, n\}\}$. 
	Suppose that the family $\left(T\wedge \Kmc_\theta^m\right)_{0<|\theta|\ll 1}$ has locally uniformly bounded mass in $U$. Then, there exists a positive closed $(p+m, p+m)$-current $T_\Kmc^{p+m}$ on $U$ such that $T_\Kmc^{p+m}$ is a limit current of the family $\left(T\wedge \Kmc_\theta^m\right)_{0<|\theta|\ll 1}$ satisfying
	\begin{align*}
		\left(\pi_V\right)_*\left(\mathbf{1}_VT_\Kmc^{p+m}\right)=(\pi_V)_*\left(T_\infty\wedge \pi_F^*\Omega_m\right),
	\end{align*}
	where $\Omega_m$ is the smooth positive closed $(m, m)$-form as in Lemma \ref{lem:useful_form}.
	In particular, when $m=N-p-h_T$ for the $h$-dimension $h_T$ of $T_\infty$, then $\left(\pi_V\right)_*\left(\mathbf{1}_VT_\Kmc^{N-h_T}\right)$ equals the shadow $T^h_\infty$ of $T_\infty$.
\end{proposition}

When $m=n$, we obtain the explicit characterization as well.
\begin{proposition}\label{cor:max_bidegree}
	Let $T\in\Cc_p(U)$ admit a tangent current $T_\infty$ along $V$. 
	Suppose that the positive integers $N$, $n$ and $p$ satisfy $n\le N-p$ and that the family $\left(T\wedge\Kmc_\theta^n\wedge\pi_V^*\omega_V^{N-p-n}\right)_{0<|\theta|\ll 1}$ has locally uniformly bounded mass in $U$. Then, the family $\left(T\wedge\Kmc_\theta^n\right)_{0<|\theta|\ll 1}$ has locally uniformly bounded mass in $U$ and there exists a positive closed $(p+n, p+n)$-current $T_\Kmc^{p+n}$ on $U$ such that $T_\Kmc^{p+n}$ is a limit current of the family $\left(T\wedge\Kmc_\theta^n\right)_{0<|\theta|\ll 1}$ satisfying
	\begin{align*}
		(\pi_V)_*T_\Kmc^{p+n}=(\pi_V)_*\left(T_\infty\wedge\pi_F^*\Omega_n\right),
	\end{align*}
	where $\Omega_n$ is the smooth positive closed $(n, n)$-form as in Lemma \ref{lem:useful_form}.
	In particular, when the $h$-dimension of $T_\infty$ is minimal, then $\left(\pi_V\right)_*T_\Kmc^{p+n}$ equals the shadow $T^h_\infty$ of $T_\infty$.
\end{proposition}

\begin{proof}
	For the first assertion, it is enough to observe that for a bidegree reason, we have
	\begin{align*}
		T\wedge\Kmc_\theta^n\wedge\omega^{N-p-n}=T\wedge\Kmc_\theta^n\wedge\pi_V^*\omega_V^{N-p-n}.
	\end{align*}
	Then, we apply Proposition \ref{prop:main_general}. 
	Since the limit currents of $\left(T\wedge\Kmc_\theta^n\right)_{0<|\theta|\ll 1}$ are supported in $V$, we have $\mathbf{1}_VT_\Kmc^{p+n}=T_\Kmc^{p+n}$.
\end{proof}

We obtain a characterization of the $h$-dimension as a direct consequence of Proposition \ref{prop:main_true_general}.
\begin{theorem}\label{thm:h_dim_K}
	Let $T\in\Cc_p(U)$ admit a tangent current $T_\infty$ along $V$. Let $h_T\in \{\max\{N-p-n, 0\}, \ldots, \min\{N-p, N-n\}\}$ be such that the family 
	\(\left(T\wedge \Kmc_\theta^i\wedge \pi_V^*\omega_V^{N-p-i}\right)_{0<|\theta|\ll 1}\) 
	has locally uniformly bounded mass in \(U\) for 
	\(i=\max\{0,n-p\},\ldots,N-p-h_T\). Suppose	that the limit currents of 
	\(\left(T\wedge \Kmc_\theta^i\wedge \pi_V^*\omega_V^{N-p-i}\right)_{0<|\theta|\ll 1}\) 
	for \(i=\max\{0,n-p\},\ldots,N-p-h_T-1\) have no mass on \(V\) and that there exists a limit current of 
	\(\left(T\wedge \Kmc_\theta^{N-p-h_T}\wedge \pi_V^*\omega_V^{h_T}\right)_{0<|\theta|\ll 1}\) 
	having mass on \(V\). Then, the $h$-dimension of $T_\infty$ is $h_T$. If $h_T=\max\{N-p-n, 0\}$ and the limit currents of $\left(T\wedge \Kmc_\theta^i\wedge \pi_V^*\omega_V^{N-p-i}\right)_{0<|\theta|\ll 1}$ for $i=\max\{n-p, 0\},\ldots,\min\{N-p, n\}$ have no mass on $V$, then the $h$-dimension of $T_\infty$ is $0$.
\end{theorem}

If we use $\omega_F^i$ and $\left(\Tmc_\theta^i\right)_{i=1, \ldots, n}$ instead of $\Omega_i$ and $\left(\Kmc_\theta^i\right)_{i=1, \ldots, n}$, we obtain a sufficient condition for the existence of tangent currents as in Theorem \ref{thm:existence_tangent_currents}.
\begin{proposition}\label{prop:tangent_2_shadows}
	Let $T\in\Cc_p(U)$. Let $\varphi:U\to \R_{\ge 0}$ be a smooth function with compact support. Then, $\left(\left(A_{\lambda}\right)_*(\varphi T)\right)_{1\ll|\lambda|}$ has uniformly bounded mass in $\overline{E}$ if and only if the family $\Big(\varphi T\wedge\Tmc_{\lambda^{-1}}^i\wedge \pi_V^*\omega_V^{N-p-i}\Big)_{1\ll|\lambda|}$ has uniformly bounded mass in $U$ for every $i=\max\{n-p, 0\}, \ldots, \min\{N-p, n\}$. 
\end{proposition}

\begin{proof}
	The restriction of $\omega_F$ to $\C^n$ can be written as $\frac{1}{2}dd^c\log(1+|x''|^2)$. We have
	\begin{align*}
		\left(A_\lambda\right)^*(\pi_F^*\omega_F)&=\left(A_\lambda\right)^*\left[\frac{1}{2}dd^c \log (1+|x''|^2)\right]
		=\frac{1}{2}dd^c \log \left(|x''|^2+|\lambda|^{-2}\right)=dd^c u_{\lambda^{-1}}^\Tmc.
	\end{align*}
	Since the support of $\left(A_{\lambda}\right)_*(\varphi T)$ stays in $E$, we have
	\begin{align*}
		&\left\langle \left(A_{\lambda}\right)_*(\varphi T), \omega^{N-p}\right\rangle
		=\sum_{i=0}^n c_i\int_{E} \left(A_{\lambda}\right)_*(\varphi T)\wedge \left(\pi_F^*\omega_F\right)^i\wedge (\pi_V^*\omega_V)^{N-p-i}\\
		&=\sum_{i=0}^n c_i\int_{U} \varphi T\wedge \left[\left(A_{\lambda}\right)^*(\pi_F^*\omega_F)^i\right]\wedge \left(\pi_V^*\omega_V\right)^{N-p-i}=\sum_{i=0}^n c_i\int_{U} \varphi T\wedge\Tmc_{\lambda^{-1}}^i\wedge \left(\pi_V^*\omega_V\right)^{N-p-i},
	\end{align*}
	where $c_i$'s are some non-nagative constants independent of $\lambda\in\C^*$. All the terms other than $i=\max\{n-p, 0\}, \ldots, \min\{N-p, n\}$ vanish for a bidegree reason.
\end{proof}

\begin{theorem}\label{thm:existence_tangent_currents}
	Let $T\in\Cc_p(U)$. Suppose that for every $i=\max\{n-p, 0\}, \ldots, \min\{N-p, n\}$, the family $\left(T\wedge \Tmc_{\lambda^{-1}}^i\wedge\pi_V^*\omega_V^{N-p-i}\right)_{1\ll|\lambda|}$ has locally uniformly bounded mass in $U$. Then, tangent currents of $T$ along $V$ exist.
\end{theorem}

\begin{proof}
	Let $(V_j)_{j\in\N}$ be a sequence of relatively compact open subsets in $V$ such that $\overline{V_j}\subset V_{j+1}$ for every $j\in\N$ and $\bigcup_{j\in\N}V_j=V$. Let $\chi_{V_j}:U\to [0, 1]$ be a smooth function with compact support such that $\chi_{V_j}\equiv 1$ in a neighborhood of $\overline{V_j}$ in $U$. Then, by Proposition \ref{prop:tangent_2_shadows}, the sequence $\left(\left(A_\lambda\right)_*(\chi_{V_j}T)\right)_{\lambda\in\C^*}$ has uniformly bounded mass in $E$ for each $j\in\N$. Let $K$ be a compact subset of $E$. Then, there exists an $j\in\N$ such that $K\subset \pi_V^{-1}(V_j)$. Also, for all $\lambda\in\C^*$ \textcolor{black}{with} sufficiently large $|\lambda|$, $\left(A_\lambda\right)_*\chi_{V_j}\equiv 1$ on $K$. Hence, we can find a sequence $\left(\lambda_k\right)$ such that $\left(A_{\lambda_k}\right)_*\left(\mathbf{1}_UT\right)$ converges on $K$.\medskip
	
	We find a sequence $\left(E_l\right)_{l\in\N}$ of relatively compact open subsets in $E$ such that $\bigcup_{l\in\N}E_l = E$. For each $l\in\N$, we apply the above arguments with $K$ replaced by $E_l$ and then, we use the classical argument to obtain a tangent current.
\end{proof}

The following lemma will be used several times later when we discuss integrability conditions. The point is that $dd^cu_\theta^\Tmc $ is strictly positive, but $dd^cu$ lacks strict positivity in one direction. However, due to the properties of the exterior algebra, multiplying a single $dd^cu_\theta^\Tmc $ completes the missing direction and we obtain the estimate.
\begin{lemma}\label{lem:T_vs_K}
	There exists a constant $C>0$ such that for every $i=1, \ldots, n$, we have
	\begin{align*}
		\Tmc_\theta^i\le Cdd^cu_\theta^\Tmc \wedge(dd^cu)^{i-1}
	\end{align*}
	in $U\setminus V$ for all $\theta\in\C^*$ with $|\theta|\ll1$. Here, the positivity is the strong positivity for forms.
\end{lemma}

\begin{proof}
	For notational convenience, in this proof, we will use $w$ instead of $x''$, that is, $w_1=x_{N-n+1}, \ldots, w_j=x_{N-n+j}, \ldots, w_n=x_{N}$.	
	On $\{w\ne 0\}$, we have
	\begin{align*}
		dd^cu &=\frac{\sum_{j=1}^n \sqrt{-1}dw_j\wedge d\overline{w_j}}{|w|^2}- \frac{\sqrt{-1}\left(\sum_{j=1}^n\overline{w_j}dw_j\right)\wedge\left(\sum_{j=1}^nw_jd\overline{w_j}\right)}{|w|^4}\quad\textrm{ and }\\
		dd^cu^\Tmc_\theta &=\frac{\sum_{j=1}^n \sqrt{-1}dw_j\wedge d\overline{w_j}}{|w|^2+|\theta|^2}- \frac{\sqrt{-1}\left(\sum_{j=1}^n\overline{w_j}dw_j\right)\wedge\left(\sum_{j=1}^nw_jd\overline{w_j}\right)}{\left(|w|^2+|\theta|^2\right)^2}\\
		&=\frac{|\theta|^2}{\left(|w|^2+|\theta|^2\right)^2}\sum_{j=1}^n \sqrt{-1}dw_j\wedge d\overline{w_j}+\frac{|w|^4}{\left(|w|^2+|\theta|^2\right)^2}dd^cu.\\
	\end{align*}
	Since $dd^cu$ is positive, we have $$0\le dd^cu_\theta^\Tmc \le \frac{|\theta|^2}{\left(|w|^2+|\theta|^2\right)^2}\sum_{j=1}^n \sqrt{-1}dw_j\wedge d\overline{w_j}+dd^cu$$
	on $\{w\ne 0\}$. Hence, for the proof of the lemma, it is enough to show that for each $i=1, \ldots, n$, there exists a constant $C_i>0$ such that on $\{w\ne 0\}$, we have
	\begin{align*}
		\left(\frac{|\theta|^2}{\left(|w|^2+|\theta|^2\right)^2}\sum_{j=1}^n \sqrt{-1}dw_j\wedge d\overline{w_j}\right)^i\le C_i \left(\frac{|\theta|^2}{\left(|w|^2+|\theta|^2\right)^2}\sum_{j=1}^n \sqrt{-1}dw_j\wedge d\overline{w_j}\right)\wedge (dd^c u)^{i-1}.
	\end{align*}
	
	Without loss of generality, we consider $0<|w_1|<1$, $|w_j|<2|w_1|$ for $j=2, \ldots, n$ and use the coordinates $(v_1, \ldots, v_n)$, where $v_1=w_1$ and $v_j=w_j/w_1$ for $j=2, \ldots, n$. For $k=2,\ldots, n$, we have	
	\begin{align*}
		dw_k=d(v_1v_k)=v_1dv_k+v_kdv_1\quad\textrm{ and }\quad d\overline{w_k}=d(\overline{v_1}\overline{v_k})=\overline{v_1}d\overline{v_k}+\overline{v_k}d\overline{v_1}.
	\end{align*}
	Since $\sqrt{-1}(v_1dv_k-v_kdv_1)\wedge(\overline{v_1}d\overline{v_k}-\overline{v_k}d\overline{v_1})$ is a positive form for $k=2, \ldots, n$, we have
	\begin{align*}
		&\frac{|\theta|^2}{\left(|w|^2+|\theta|^2\right)^2}(\sqrt{-1}dw_k\wedge d\overline{w_k})\wedge\left(\frac{|\theta|^2}{\left(|w|^2+|\theta|^2\right)^2}\sum_{j=1}^n \sqrt{-1}dw_j\wedge d\overline{w_j}\right)^{i-1}\\
		&\le\frac{|\theta|^2}{\left(|w|^2+|\theta|^2\right)^2}\sqrt{-1}(v_1dv_k+v_kdv_1)\wedge(\overline{v_1}d\overline{v_k}+\overline{v_k}d\overline{v_1})\wedge\left(\frac{|\theta|^2}{\left(|w|^2+|\theta|^2\right)^2}\sum_{j=1}^n \sqrt{-1}dw_j\wedge d\overline{w_j}\right)^{i-1}\\
		&\quad\quad\quad+\frac{|\theta|^2}{\left(|w|^2+|\theta|^2\right)^2}\sqrt{-1}(v_1dv_k-v_kdv_1)\wedge(\overline{v_1}d\overline{v_k}-\overline{v_k}d\overline{v_1})\wedge\left(\frac{|\theta|^2}{\left(|w|^2+|\theta|^2\right)^2}\sum_{j=1}^n \sqrt{-1}dw_j\wedge d\overline{w_j}\right)^{i-1}\\
		&=\frac{2|\theta|^2}{\left(|w|^2+|\theta|^2\right)^2}\sqrt{-1}(|v_1|^2dv_k\wedge d\overline{v_k}+|v_k|^2dv_1\wedge d\overline{v_1})\wedge\left(\frac{|\theta|^2}{\left(|w|^2+|\theta|^2\right)^2}\sum_{j=1}^n \sqrt{-1}dw_j\wedge d\overline{w_j}\right)^{i-1}.
	\end{align*}
	
	Note that with respect to the coordinates in our consideration, $|v_k|$'s are bounded by $2$ for $k=2, \ldots, n$. Inductively applying the above inequality, we see that on $\{w\ne 0\}$, the form $\displaystyle \left(\frac{|\theta|^2}{\left(|w|^2+|\theta|^2\right)^2}\sum_{j=1}^n \sqrt{-1}dw_j\wedge d\overline{w_j}\right)^i$ is bounded by a linear combination of $(i, i)$-forms of the following types with positive coefficients: 
	\begin{align}
		\label{eq:(1)}&\left(\sqrt{-1}\right)^i\left(\frac{|\theta|^2}{\left(|w|^2+|\theta|^2\right)^2}\right)^i|v_1|^{2i-2}dv_1\wedge d\overline{v_1}\wedge dv_{j_1}\wedge d\overline{v_{j_1}}\wedge \cdots \wedge dv_{j_{i-1}}\wedge d\overline{v_{j_{i-1}}},\\ \label{eq:(2)}&\left(\sqrt{-1}\right)^i\left(\frac{|\theta|^2}{\left(|w|^2+|\theta|^2\right)^2}\right)^i|v_1|^{2i}dv_{j_1}\wedge d\overline{v_{j_1}}\wedge \cdots \wedge dv_{j_{i}}\wedge d\overline{v_{j_{i}}},
	\end{align}
	where $j_1, \ldots, j_i\ge 2$ and $|w|^2=|v_1|^2\left(1+\sum_{j=2}^n|v_j|^2\right)$.
	\medskip
	
	We consider \eqref{eq:(1)}. Since $v_1=w_1$ and $|v_j|<2$ for $j=2, \ldots, n$, from the relationship between the arithmetic and geometric means, we have
	\begin{align*}
		\eqref{eq:(1)}&\le \left(\sqrt{-1}\right)^i\left(\frac{|\theta|^2}{\left(|w|^2+|\theta|^2\right)^2}\right)\left|\frac{v_1}{w}\right|^{2i-2}dv_1\wedge d\overline{v_1}\wedge dv_{j_1}\wedge d\overline{v_{j_1}}\wedge \cdots \wedge dv_{j_{i-1}}\wedge d\overline{v_{j_{i-1}}}\\
		&\le \left(\sqrt{-1}\right)^i\left(\frac{|\theta|^2}{\left(|w|^2+|\theta|^2\right)^2}\sum_{j=1}^n dw_j\wedge d\overline{w_j}\right)\wedge dv_{j_1}\wedge d\overline{v_{j_1}}\wedge \cdots \wedge dv_{j_{i-1}}\wedge d\overline{v_{j_{i-1}}}.
	\end{align*}
	With respect to $v_1, \ldots, v_n$, $dd^cu$ is written as
	\begin{align*}
		dd^c u=\frac{1}{2}dd^c\log\left(1+\sum_{j=2}^n|v_j|^2\right)
	\end{align*}
	and on a bounded domain $\displaystyle \left\{(v_2, \ldots, v_n)\in\C^{n-1}:|v_2|<2,\ldots, |v_n|<2\right\}$, $\displaystyle \frac{1}{2}dd^c\log\left(1+\sum_{j=2}^n|v_j|^2\right)$ is equivalent to $\displaystyle dd^c\left(\sum_{j=2}^n|v_j|^2\right)$. Hence, we see that there exists a constant $c_1>0$ such that
	\begin{align*}
		\eqref{eq:(1)}\le c_1 \left(\frac{|\theta|^2}{\left(|w|^2+|\theta|^2\right)^2}\sum_{j=1}^n \sqrt{-1}dw_j\wedge d\overline{w_j}\right)\wedge (dd^c u)^{i-1}.
	\end{align*}
	\medskip
	
	We consider \eqref{eq:(2)}. For $k=2, \ldots, n$, we have
	\begin{align*}
		&2\sqrt{-1}dw_k\wedge d\overline{w_k}+ 2\sqrt{-1}|v_k|^2dv_1\wedge d\overline{v_1}\\
		&\quad=\sqrt{-1}|v_1|^2dv_k\wedge d\overline{v_k}+ \sqrt{-1}\left(|v_1|^2dv_k\wedge d\overline{v_k} + 2v_1\overline{v_k}dv_k\wedge d\overline{v_1}+ 2v_k\overline{v_1}dv_1\wedge d\overline{v_k}+4|v_k|^2dv_1\wedge d{\overline{v_1}}\right)\\
		&\quad=\sqrt{-1}|v_1|^2dv_k\wedge d\overline{v_k}+\sqrt{-1}(v_1dv_k+2v_kdv_1)\wedge(\overline{v_1}d\overline{v_k}+2\overline{v_k}d\overline{v_1})\ge \sqrt{-1}|v_1|^2dv_k\wedge d\overline{v_k}.
	\end{align*}
	We have $|v_k|<2$ for $k=2, \ldots, n$. Arguing as in (1), we obtain from the above inequality that there exists $c_2>0$ such that 
	\begin{align*}
		\eqref{eq:(2)}&\le\left(\sqrt{-1}\right)^i\left(\frac{|\theta|^2}{\left(|w|^2+|\theta|^2\right)^2}\right)\left|\frac{v_1}{w}\right|^{2i-2}|v_1|^2dv_{j_1}\wedge d\overline{v_{j_1}}\wedge \cdots dv_{j_{i}}\wedge d\overline{v_{j_{i}}}\\
		&\le c_2 \left(\frac{|\theta|^2}{\left(|w|^2+|\theta|^2\right)^2}\sum_{j=1}^n \sqrt{-1}dw_j\wedge d\overline{w_j}\right)\wedge (dd^c u)^{i-1}.
	\end{align*}
	The constants $c_1$ and $c_2$ are independent of $\theta$. The above two inequalities prove the claim.
\end{proof}

\section{Integrability Conditions}\label{sec:integrability}

In Section \ref{sec:tangent_currents}, we have seen that once the product of $T$ with $(dd^cu)^m$ is defined via the family $\left(\Kmc_\theta^m\right)_{0<|\theta|\ll 1}$ for some $m\in\{1, \ldots, n\}$, we can study properties of tangent currents of $T$ along $V$ such as the $h$-dimension and existence of tangent currents.\medskip

In this section, we introduce two natural integrability conditions for the product: one associated with the classical complex Monge-Amp\`ere type product, and the other with the relative non-pluripolar product. For the relative non-pluripolar product on compact K\"ahler manifolds, see \cite{VU21-1}. We employ the notations in Section \ref{sec:tangent_currents}. We work on a domain $U\subset\C^N$ and a complex linear submanifold $V\subset U$ of codimension $n$ such that $V$ is bounded and simply connected with smooth boundary.\medskip

The main theorem of this section, Theorem \ref{thm:uniqueness_tangent_current}, states that if a positive closed current on $U$ satisfies an integrability condition called Condition $(\mathrm{K}-\max)$ along $V$, then it admits a unique tangent current along $V$ with the minimal $h$-dimension, whose shadow can be described by an integral representation.
\medskip

Below is a notation for a product of a current with $(dd^cu)^m$ via the family $\left(\Kmc_\theta^m\right)_{0<|\theta|\ll 1}$.
\begin{definition}\label{def:K-product}
	Let $T'\in\Cc_{p'}(U)$. Let $m'\in\{0, \ldots, \min\{N-p', n\}\}$.
	When $T'$ is positive, the positive $(p'+m', p'+m')$-current $\left\langle T'\wedge (dd^cu)^{m'}\right\rangle_\Kmc$ on $U$ is defined to be the limit
	\begin{align*}
		\left\langle T'\wedge(dd^cu)^{m'}\right\rangle_\Kmc:=\lim_{\theta\to 0}T'\wedge \Kmc^{m'}_\theta
	\end{align*}
	in the sense of currents, provided that the limit exists. When $T'$ is a current that can be written as a linear combination $\sum_{i=1}^{n_{T'}}\alpha_iT'_i$ of positive currents $T'_i$ for $i=1, \ldots, n_{T'}$, we define 
	\begin{align*}
		\left\langle T'\wedge(dd^cu)^{m'}\right\rangle_\Kmc:=\sum_{i=1}^{n_{T'}}\alpha_i\left\langle T'_i\wedge(dd^cu)^{m'} \right\rangle_\Kmc,
	\end{align*}
	provided that all the limits $\left\langle T'_i\wedge(dd^cu)^{m'}\right\rangle_\Kmc$ for $i=1, \ldots, n_{T'}$ exist.
\end{definition}

\begin{remark}
	If $T'$ is closed and $\left\langle T'\wedge(dd^cu)^{m'}\right\rangle_\Kmc$ exists, then $\left\langle T'\wedge(dd^cu)^{m'}\right\rangle_\Kmc$ is closed.
\end{remark}

\subsection{Classical complex Monge-Amp\`ere type product}\label{subsec:classical} In this subsection, we assume $n<N$.

\begin{lemma}\label{lem:basic_conv_KT}
	Let $T\in\Cc_p(U)$. Let $i\in\{\max\{n-p, 1\}, \ldots, \min\{N-p, n\}\}$. Suppose that the current $\left\langle T\wedge(dd^cu)^{i-1}\wedge \pi_V^*\psi\right\rangle_\Kmc$ of order $0$ of maximal bidegree on $U$ is well-defined for every smooth $(N-p-i+1, N-p-i+1)$-form $\psi$ on $V$ and that 
	\begin{align*}
		u\in L_\loc^1\left(\left\langle T\wedge(dd^cu)^{i-1}\wedge \pi_V^*\omega_V^{N-p-i+1}\right\rangle_\Kmc\right).
	\end{align*}
	Then, the current $\left\langle T\wedge(dd^cu)^i\wedge \pi_V^*\varphi\right\rangle_\Kmc$ of order $0$ of maximal bidegree on $U$ is well-defined for every smooth $(N-p-i, N-p-i)$-form $\varphi$ on $V$. 
	For $\theta\in\C^*$ with $|\theta|\ll 1$, the limit
	\begin{align*}
		\left\langle \left(T\wedge dd^c u_\theta^\Tmc \right)\wedge (dd^cu)^{i-1}\wedge \pi_V^*\varphi\right\rangle_\Kmc=\lim_{|\theta'|\to 0}\left(T\wedge dd^cu_\theta^\Tmc \right)\wedge \Kmc^{i-1}_{\theta'} \wedge\pi_V^*\varphi
	\end{align*}
	exists and we have 
	\begin{align*}
		\lim_{\theta\to 0}\left\langle \left(T\wedge dd^c u_\theta^\Tmc \right)\wedge (dd^cu)^{i-1}\wedge  \pi_V^*\varphi\right\rangle_\Kmc =\left\langle T\wedge(dd^cu)^i\wedge \pi_V^*\varphi\right\rangle_\Kmc.
	\end{align*}
\end{lemma}

For the proof, we need the following lemma.
\begin{lemma}\label{lem:restriction_lem}
	Let $T'$ be a positive current of maximal bidegree on $U$. Let $T'|_V:=\mathbf{1}_V T'$ be the restriction of $T'$ to $V$. Then, for every smooth test function $f$ on $U$, we have
	\begin{align*}
		\left\langle T'|_V, f\right\rangle = \left\langle T'|_V, \pi_V^*f|_V\right\rangle.
	\end{align*}
\end{lemma}

\begin{proof}
	Let $\varepsilon>0$ be such that $\varepsilon\ll 1$. Then, we have
	\begin{align*}
		\left|\int_{V_\varepsilon} (f-\pi_V^*f|_V)T'\right|\le \sup_{V_\varepsilon}\left|f-\pi_V^*f|_V\right|\int_{V_\varepsilon\cap \supp\,f}T'\le \varepsilon\|f\|_{C^1}\|T'\|_{\supp\,f},
	\end{align*}
	where $V_\varepsilon$ is the $\varepsilon$-neighborhood of $V$ in $U$.
	Hence, as $\varepsilon$ shrinks to $0$, we obtain the equality.
\end{proof}

\begin{proof}[Proof of Lemma \ref{lem:basic_conv_KT}] We may assume that $\varphi$ is positive. 
	We first prove that the family
	\begin{align*}
		\left(T\wedge \Kmc^i_\theta\wedge \pi_V^*\varphi\right)_{0<|\theta|\ll 1}
	\end{align*}
	has locally uniformly bounded mass in $U$. Since $u$ is smooth in $U\setminus V$ and $u_\theta^\Kmc$ decreasingly converges to $u$, it suffices to prove that the family has locally uniformly bounded mass around $V$. Let $\chi_V:V\to[0,1]$ and $\chi_F:\C^n\to[0,1]$ be smooth functions with compact support such that $\chi_F\equiv 1$ in a neighborhood of $0$ and that $\supp\,\pi_V^*\chi_V\cap \supp\,\pi_F^*\chi_F$ is compact in $U$. It is enough to show that the integral $\int (\pi_F^*\chi_F)(\pi_V^*\chi_V)T\wedge \Kmc^i_\theta\wedge \pi_V^*\omega_V^{N-p-i}$ is uniformly bounded. 
	Since $dd^cu_\theta^\Kmc=0$ near $V$, for $\theta'\in\C^*$ with $|\theta'|\ll|\theta|$, we have
	\begin{align}
		\notag&\int (\pi_F^*\chi_F)(\pi_V^*\chi_V)T\wedge \Kmc^i_\theta\wedge \pi_V^*\omega_V^{N-p-i}=\int(\pi_F^*\chi_F)(\pi_V^*\chi_V)T\wedge dd^c u_\theta^\Kmc \wedge \Kmc_{\theta'}^{i-1}\wedge \pi_V^*\omega_V^{N-p-i}\\
		\label{eq:King_bounded_1}&=\int u_\theta^\Kmc (\pi_F^*\chi_F) dd^c(\pi_V^*\chi_V)\wedge T\wedge \Kmc_{\theta'}^{i-1}\wedge \pi_V^*\omega_V^{N-p-i}\\
		\label{eq:King_bounded_2}&\quad\quad\quad +\int u_\theta^\Kmc \left(dd^c\left((\pi_V^*\chi_V)(\pi_F^*\chi_F)\right)-(\pi_F^*\chi_F) dd^c(\pi_V^*\chi_V)\right) \wedge T\wedge \Kmc_{\theta'}^{i-1}\wedge  \pi_V^*\omega_V^{N-p-i}.
	\end{align}
	The region of integration in \eqref{eq:King_bounded_2} remains uniformly separated from $V$ with respect to $\theta$ and the form $dd^cu$ is smooth there. We see that \eqref{eq:King_bounded_2} converges as $\theta'\to 0$ and $\theta\to 0$ in this order. Hence, for all $\theta\in\C^*$ with $|\theta|\ll 1$, \eqref{eq:King_bounded_2} is uniformly bounded.
	\medskip
	
	We look into \eqref{eq:King_bounded_1}. Since $\chi_V$ is smooth, the measure
	\begin{align*}
		u_\theta^\Kmc (\pi_F^*\chi_F)dd^c(\pi_V^*\chi_V)\wedge T\wedge \Kmc_{\theta'}^{i-1}\wedge \pi_V^*\omega_V^{N-p-i}
	\end{align*}
	can be written as a difference of two positive measure, each of which is dominated by $-u_\theta^\Kmc (\pi_F^*\chi_F)T\wedge  \Kmc_{\theta'}^{i-1}\wedge \pi_V^*\left(\widetilde{\chi_V}\omega_V^{N-p-i+1}\right)$ up to a multiplicative constant independent of $\theta$, where $\widetilde{\chi_V}:V\to [0,1]$ is a smooth function with compact support such that $\widetilde{\chi_V}\equiv 1$ on the support of $\chi_V$ and that $\supp\,\pi_V^*\widetilde{\chi_V}\cap \supp\,\pi_F^*\chi_F$ is compact in $U$ after modifying $\chi_F$ if necessary. From the hypotheses, $u_\theta^\Kmc (\pi_F^*\chi_F)T\wedge  \Kmc_{\theta'}^{i-1}\wedge \pi_V^*\left(\widetilde{\chi_V}\omega_V^{N-p-i+1}\right)$ converges as $\theta'\to 0$ and then $\theta\to 0$. So, for all $\theta\in\C^*$ with $|\theta|\ll 1$, \eqref{eq:King_bounded_1} is uniformly bounded. (When $1<n-p$ and $i=n-p$, \eqref{eq:King_bounded_1} vanishes.)
	\medskip
	
	We show that the family $\left(T\wedge \Kmc_\theta^i\wedge \pi_V^*\varphi\right)_{0<|\theta|\ll 1}$ of positive currents of maximal bidegree on $U$ has a unique limit current. Observe that every limit current coincides on $U\setminus V$. Hence, we are interested in the restrictions of the limit currents to $V$. The limit currents are positive, and therefore, by Lemma \ref{lem:restriction_lem}, it suffices to check whether $\left\langle\mathbf{1}_VL, \pi_V^*g\right\rangle$ is independent of $L$, where $L$ is a limit current of the family $\left(T\wedge \Kmc_\theta^i\wedge \pi_V^*\varphi\right)_{0<|\theta|\ll 1}$ and $g:V\to \R$ is a smooth test function.	Let $L^0:=\mathbf{1}_{U\setminus V}L$. Then, we have
	\begin{align*}
		\left\langle \mathbf{1}_VL, \pi_V^*g\right\rangle=\left\langle \mathbf{1}_VL, \left(\pi_F^*\chi_F\right)\left(\pi_V^*g\right)\right\rangle=\left\langle L, \left(\pi_F^*\chi_F\right)\left(\pi_V^*g\right)\right\rangle-\left\langle L^0, \left(\pi_F^*\chi_F\right)\left(\pi_V^*g\right)\right\rangle,
	\end{align*}
	with the function $\chi_F:\C^n\to [0,1]$ modified so that $\supp\pi_V^*g\cap \supp\pi_F^*\chi_F$ is a compact subset of $U$, if necessary.
	As pointed out previously, every limit current of $\left(T\wedge \Kmc_\theta^i\wedge \pi_V^*\varphi\right)_{0<|\theta|\ll 1}$ coincides on $U\setminus V$ and the second term $\left\langle L^0, \left(\pi_F^*\chi_F\right)\left(\pi_V^*g\right)\right\rangle$ is independent of the choice of the limit current $L$. So, it is enough to show that the first term $\left\langle L, \left(\pi_F^*\chi_F\right)\left(\pi_V^*g\right)\right\rangle$ is independent of the choice of the limit current $L$. It is equivalent to verifying that $\displaystyle \lim_{\theta\to 0}\left\langle\left(\pi_F^*\chi_F\right)T\wedge \Kmc_{\theta}^i\wedge \pi_V^*\varphi, \pi_V^*g\right\rangle$ exists. 
	\medskip
	
	As previously, for $\theta'\in\C^*$ with $|\theta'|\ll|\theta|$, we have
	\begin{align}
		\notag&\left\langle\left(\pi_F^*\chi_F\right)T\wedge \Kmc_{\theta}^i\wedge \pi_V^*\varphi, \pi_V^*g\right\rangle=\int \left(\pi_F^*\chi_F\right)\left(\pi_V^*\left(g\varphi\right)\right)\wedge T\wedge dd^c u_{\theta}^\Kmc\wedge \Kmc_{\theta'}^{i-1}\\
		\label{eq:King_convergence_1}&\quad\quad\quad=\int u_{\theta}^\Kmc\left(\pi_F^*\chi_F\right)dd^c\left(\pi_V^*\left(g\varphi\right)\right)\wedge T\wedge \Kmc_{\theta'}^{i-1}\\
		\label{eq:King_convergence_2}&\quad\quad\quad\quad\quad\quad+\int u_{\theta}^\Kmc \left[dd^c\left(\left(\pi_F^*\chi_F\right)\left(\pi_V^*\left(g\varphi\right)\right)\right)-\left(\pi_F^*\chi_F\right)dd^c\left(\pi_V^*\left(g\varphi\right)\right)\right]\wedge T\wedge \Kmc_{\theta'}^{i-1}.
	\end{align}
	In the same way as above, \eqref{eq:King_convergence_2} converges as $\theta'\to 0$ and then $\theta\to 0$. From the hypotheses, \eqref{eq:King_convergence_1} also converges as $\theta'\to 0$ and then $\theta\to 0$. (When $1<n-p$ and $i=n-p$, \eqref{eq:King_convergence_1} vanishes.)
	\medskip

	The assertions for $u_\theta^\Tmc $ can be proved in the same way as above. We simply replace $u_\theta^\Kmc$ by $u_\theta^\Tmc $.
\end{proof}

Below, we introduce an integrability condition, which arises naturally in the inductive definition of a product of $T$ with $(dd^cu)^n$.
\begin{definition}[Condition $(\mathrm{K})$]\label{defn:Condition(I)}
	Let $T\in\Cc_p(U)$. Let $m\in\{\max\{n-p, 1\}, \ldots, \min\{N-p, n\}\}$. We say that $T$ satisfies {Condition $(\mathrm{K})$ along $V$ up to codimension $m$} if
	\begin{align}\label{eq:local_integrability}
		u\in L_{\loc}^1\left(\left\langle T\wedge (dd^cu)^{i-1}\wedge\pi_V^*\omega_V^{N-p-i+1}\right\rangle_\Kmc\right)
	\end{align}
	in $U$ inductively holds from $i=\max\{n-p, 1\}$ through $m$. When $m=\min\{N-p, n\}$, $T$ is said to satisfy Condition $(\mathrm{K}-\max)$ along $V$. 	
\end{definition}

\begin{remark}\label{rmk:integrability}
	For each $i\in\{\max\{n-p, 1\}, \ldots, m\}$, the integrability condition \eqref{eq:local_integrability} is equivalent to saying that the positive measure $\left\langle T\wedge (dd^cu)^{i-1}\wedge\pi_V^*\omega_V^{N-p-i+1}\right\rangle_\Kmc$ has no mass on $V$ and that the positive measure $-uT\wedge (dd^cu)^{i-1}\wedge\pi_V^*\omega_V^{N-p-i+1}$ on $U\setminus V$ is locally integrable around $V$. Indeed, by definition, for every compact subset $K\subset U$, we have
	\begin{align*}
		\int_K u\left\langle T\wedge (dd^cu)^{i-1}\wedge\pi_V^*\omega_V^{N-p-i+1}\right\rangle_\Kmc=\lim_{\theta\to 0}\int_K u^\Kmc_\theta\left\langle T\wedge (dd^cu)^{i-1}\wedge\pi_V^*\omega_V^{N-p-i+1}\right\rangle_\Kmc.
	\end{align*}
\end{remark}

The following proposition is straightforward.
\begin{proposition}
	Let $T\in\Cc_p(U)$ satisfy Condition $(\mathrm{K})$ along $V$ up to codimension $m$ for some $m\in \{\max\{n-p, 1\}, \ldots, \min\{N-p, n\}\}$. Let $\varphi$ be a smooth closed $(q, q)$-form on $V$, where $q\in \{0, \ldots, \min\{N-p, n\}-m\}$. Then, the product $T\wedge \pi_V^*\varphi$ also satisfies Condition $(\mathrm{K})$ along $V$ up to codimension $m$.
\end{proposition}

The next proposition finds a relationship between the $h$-dimension and Condition $(\mathrm{K})$. 
\begin{proposition}\label{prop:h-dim}
	Let $T\in\Cc_p(U)$ admit a tangent current $T_\infty$ along $V$. Let $m\in \{\max\{n-p, 1\}, \ldots, \min\{N-p, n\}\}$. Then, if $T$ satisfies Condition $(\mathrm{K})$ along $V$ up to codimension $m$, then the $h$-dimension of $T_\infty$ is at most $N-p-m$. In particular, if $T$ satisfies Condition $(\mathrm{K}-\max)$ along $V$, then the $h$-dimension of $T_\infty$ is minimal.
\end{proposition}

\begin{proof}
	Condition $(\mathrm{K})$ along $V$ up to codimension $m$ implies that $\displaystyle \lim_{\theta\to 0}T\wedge \Kmc_\theta^i\wedge \pi_V^*\omega_V^{N-p-i}$ exists for $i=1, \ldots, m$ and each limit current has no mass on $V$ for $i=0, \ldots, m-1$. Hence, Theorem \ref{thm:h_dim_K} says that the $h$-dimension of $T_\infty$ is at most $N-p-m$.
\end{proof}

We further look into the relationship between Condition $(\mathrm{K}-\max)$ and tangent currents. The following theorem is the main result related to Condition $(\mathrm{K}-\max)$ and we obtain Theorem \ref{thm:main_intersection_1} as its particular case later in Section \ref{sec:intersection}.
\begin{theorem}
\label{thm:uniqueness_tangent_current}
	Let $T\in\Cc_p(U)$ satisfy Condition $(\mathrm{K}-\max)$ along $V$. Then, there exists a unique tangent current $T_\infty$ of $T$ along $V$, and its $h$-dimension $h_T$ is minimal. Furthermore, for a smooth test $(h_T, h_T)$-form $\varphi$ on $V$, its shadow $T_\infty^h$ satisfies
	\begin{align*}
		\left\langle T_\infty^h, \varphi\right\rangle=\langle \mathbf{1}_V\langle T\wedge (dd^cu)^{N-p-h_T}\wedge\pi_V^*\varphi\rangle_\Kmc, 1\rangle.
	\end{align*}
	In particular, if $p\le N-n$, for a smooth test $(N-p-n, N-p-n)$-form $\varphi$ on $V$, its shadow $T_\infty^h$ is computed as
	\begin{align*}
		\left\langle T_\infty^h, \varphi\right\rangle=\int_{U\setminus V} T\wedge \left(u \left(dd^c u\right)^{n-1}\right)\wedge dd^c\Phi,
	\end{align*}
	where $\Phi$ is a smooth $(N-p-n, N-p-n)$-form on $U$ with compact support such that $\Phi=\pi_V^*\varphi$ in a neighborhood of $V$.
\end{theorem}

For the proof, we use Proposition \ref{prop:uniqueness_from_shadow}. Its proof is straightforward from \cite[Lemma 3.4]{DS18}.
\begin{proposition}\label{prop:uniqueness_from_shadow}
	Suppose that $T'$ admits tangent currents along $V$, all of which have the minimal $h$-dimension over $V$. Then, the current $T'$ admits a unique tangent current along $V$ if and only if the shadow is independent of the choice of tangent current.
\end{proposition}

\begin{proof}[Proof of Theorem \ref{thm:uniqueness_tangent_current}]
	Lemma \ref{lem:basic_conv_KT} and Condition $(\mathrm{K}-\max)$ imply that $\big(\big\langle \left(T \wedge dd^cu_\theta^\Tmc \right)\wedge (dd^cu)^{i-1}\wedge \pi_V^*\omega_V^{N-p-i}\big\rangle_\Kmc\big)_{0<|\theta|\ll 1}$ has locally uniformly bounded mass for $i=\max\{n-p, 1\}, \ldots, \min\{N-p, n\}$. Notice that for a bidegree reason, $T$ has no mass on $V$. By Lemma \ref{lem:T_vs_K}, 
	the family $\left(T\wedge \Tmc_\theta^i \wedge\pi_V^*\omega_V^{N-p-i}\right)_{0<|\theta|\ll 1}$ has locally uniformly bounded mass for $i=\max\{n-p, 1\}, \ldots, \min\{N-p, n\}$. By Theorem \ref{thm:existence_tangent_currents}, $T$ has a tangent current $T_\infty$ along $V$. Lemma \ref{lem:T_vs_K}, Lemma \ref{lem:basic_conv_KT}, Proposition \ref{prop:h-dim} and Condition $(\mathrm{K}-\max)$ imply that its $h$-dimension is minimal.\medskip
	
	Let $T_\infty$ be a tangent current. Let $m=\min\{N-p, n\}$. Then, Condition $(\mathrm{K}-\max)$ and Lemma \ref{lem:basic_conv_KT} imply that for every smooth test $(N-p-m, N-p-m)$-form $\varphi$, the sequence $\langle T\wedge\Kmc_\theta^m\wedge\pi_V^*\varphi\rangle$ converges to $\langle T\wedge (dd^cu)^m\wedge\pi_V^*\varphi\rangle_\Kmc$. So, Proposition \ref{prop:main_true_general} implies that
	\begin{align*}
		\langle T_\infty^h, \varphi\rangle = \langle \left(\pi_V\right)_*(T_\infty\wedge\pi_F^*\Omega_m),\varphi\rangle = \langle \mathbf{1}_V\langle T\wedge (dd^cu)^m\wedge\pi_V^*\varphi\rangle_\Kmc, 1\rangle,
	\end{align*}
	which means the shadow is independent of the choice of $T_\infty$. Then, Proposition \ref{prop:uniqueness_from_shadow} says that the tangent current is unique.
	\medskip
	
	The integral representation of the shadow of the unique tangent current in the case of $p\le N-n$ will be completed in the following proposition.	
\end{proof}

\begin{proposition}\label{prop:compute_tangent_currents}
	Let $T$ be as in Theorem \ref{thm:uniqueness_tangent_current} with $p\le N-n$. Then, for a smooth test $(N-p-n, N-p-n)$-form $\varphi$ on $V$, we have
	\begin{align*}
		\left\langle \mathbf{1}_V \left\langle T\wedge (dd^cu)^n\wedge\pi_V^*\varphi\right\rangle_\Kmc, 1\right\rangle=\left\langle\left\langle T\wedge (dd^cu)^n\wedge\pi_V^*\varphi\right\rangle_\Kmc, 1\right\rangle=\int_{U\setminus V} T\wedge \left(u \left(dd^c u\right)^{n-1}\right)\wedge dd^c\Phi,
	\end{align*}
	where $\Phi$ is a smooth $(N-p-n, N-p-n)$-form on $U$ with compact support such that $\Phi=\pi_V^*\varphi$ in a neighborhood of $V$.
\end{proposition}

\begin{proof}
	Condition $(\mathrm{K}-\max)$ and Lemma \ref{lem:basic_conv_KT} imply that the current $\left\langle T\wedge (dd^cu)^n\wedge\pi_V^*\varphi\right\rangle_\Kmc$ is well-defined. Since the support of $\Kmc_\theta^n$ shrinks to $V$ as $\theta\to 0$, we have
	\begin{align*}
		&\left\langle \mathbf{1}_V \left\langle T\wedge (dd^cu)^n\wedge\pi_V^*\varphi\right\rangle_\Kmc, 1\right\rangle=\left\langle\left\langle T\wedge (dd^cu)^n\wedge\pi_V^*\varphi\right\rangle_\Kmc, 1\right\rangle=\lim_{\theta\to 0}\int_U T\wedge \Kmc_\theta^n\wedge \pi_V^*\varphi\\
		&\quad\quad\quad=\lim_{\theta\to 0}\int_U T\wedge \Kmc_\theta^n\wedge \Phi=\lim_{\theta\to 0}\lim_{\theta'\to 0}\int_U T\wedge (dd^cu_\theta^\Kmc)\wedge\Kmc_{\theta'}^{n-1}\wedge \Phi\\
		&\quad\quad\quad=\lim_{\theta\to 0}\lim_{\theta'\to 0}\left(\int_U \left(\pi_F^*\chi_F\right)u_\theta^\Kmc T\wedge\Kmc_{\theta'}^{n-1}\wedge dd^c\left(\pi_V^*\varphi\right) + \int_U \left(1-\pi_F^*\chi_F\right)u_\theta^\Kmc T\wedge\Kmc_{\theta'}^{n-1}\wedge dd^c\Phi\right)\\
		&\quad\quad\quad=\lim_{\theta\to 0}\int_{U\setminus V} u_\theta^\Kmc  (dd^cu)^{n-1}\wedge T \wedge dd^c\Phi=\int_{U\setminus V} u(dd^cu)^{n-1}\wedge T \wedge dd^c\Phi.
	\end{align*}
	Here, $\chi_F:\C^n\to [0,1]$ is a smooth function with sufficiently small compact support such that $\chi_F\equiv 1$ in a neighborhood of $0$. The second to last equality is from Lemma \ref{lem:basic_conv_KT}, Remark \ref{rmk:integrability} and the fact that $V \cap \supp \left(1-\pi_F^*\chi_F\right)=\emptyset$. The last convergence is from the integrability condition.
\end{proof}

We extract from the above proof the following:
\begin{align}\label{eq:convergence_V}
	\left\langle\left\langle T\wedge (dd^cu)^n\wedge\pi_V^*\varphi\right\rangle_\Kmc, 1\right\rangle=\lim_{\theta\to 0}\int_U T\wedge \Kmc_\theta^n\wedge \pi_V^*\varphi
\end{align}
Observe that for a bidegree reason, we have $\left(T\wedge \Kmc_\theta^n\right)\wedge \psi\ne 0$ only when $\psi$ contains a linear combination of smooth forms of the type $f\pi_V^*\gamma$, where $f$ is a smooth function on $U$ and $\gamma$ is a smooth test $(N-p-n, N-p-n)$-form on $V$. So, \eqref{eq:convergence_V} implies that $\left(T\wedge\Kmc_\theta^n\right)_{0<|\theta|\ll 1}$ has locally uniformly bounded mass in $U$. The limit currents of $\left(T\wedge\Kmc_\theta^n\right)_{0<|\theta|\ll 1}$ have support in $V$. Hence, Lemma \ref{lem:restriction_lem} implies that \eqref{eq:convergence_V} proves the convergence of $\left(\left\langle T\wedge \Kmc_\theta^n, f\pi_V^*\gamma\right\rangle\right)_{0<|\theta|\ll 1}$ as $\theta\to 0$, where $f$ is a smooth function on $U$ and $\gamma$ is a smooth test $(N-p-n, N-p-n)$-form on $V$. Hence, summarizing the discussion, we obtain the following proposition.
\begin{proposition}\label{prop:convergence_V}
	Let $T$ be as in Proposition \ref{prop:compute_tangent_currents}. Then, $\left(T\wedge \Kmc_\theta^n\right)_{0<|\theta|\ll 1}$ converges as $\theta\to 0$. 
\end{proposition}

When $T$ is smooth, $T\wedge \Kmc_\theta^n$ converges to $T\wedge [V]$. Inspired by this, we define the following.
\begin{definition}\label{def:wedge_C}
	Let $T$ be as in Proposition \ref{prop:compute_tangent_currents}.	We define the current $T\wedge_C[V]$ on $U$ by
	\begin{align}\label{eq:def_classical_product}
		T\wedge_C[V]:=\left\langle T\wedge (dd^cu)^n\right\rangle_\Kmc=\lim_{\theta\to 0}T\wedge \Kmc_\theta^n
	\end{align}
	in the sense of currents.
\end{definition}

Note that $T_\infty\wedge[V]$ can be understood as a double current in de Rham's language (\cite{deRham}).
\begin{proposition}\label{prop:T_wedge_V}
	Let $T$ be as in Proposition \ref{prop:compute_tangent_currents}. Let $T_\infty$ be the unique tangent current of $T$. Then, we have
	\begin{align*}
		T_\infty\wedge [V]=T\wedge_C [V]
	\end{align*}
	in the sense of currents on $U$. 
	
\end{proposition}
\begin{proof}
	The current $T_\infty$ is a current in $x'$, where $(x', x'')\in E$ is as in Section \ref{sec:tangent_currents}. So, we have
	\begin{align*}
		&\left(\pi_V\right)_*\left(T_\infty\wedge[V]\right)=\left(\pi_V\right)_*\left(T_\infty\wedge\pi_F^*\Omega_n\right)=\lim_{\lambda\to\infty}\left(\pi_V\right)_*\left(\left(A_\lambda\right)_*\left(\mathbf{1}_UT\right)\wedge\pi_F^*\Omega_n\right)\\
		&\quad=\lim_{\lambda\to\infty}\left(\pi_V\right)_*\left(\mathbf{1}_UT\wedge\left(A_\lambda\right)^*\left(\pi_F^*\Omega_n\right)\right)=\lim_{\lambda\to\infty}\left(\pi_V\right)_*\left(\mathbf{1}_UT\wedge \Kmc_{e^{-M}\lambda^{-1}}^n\right)=\left(\pi_V\right)_*\left(\mathbf{1}_UT\wedge_C[V]\right),
	\end{align*}
	where $\Omega_n$ is the form as in Lemma \ref{lem:useful_form}. The last equality comes from Condition $(\mathrm{K}-\max)$. Then, as in the argument for Proposition \ref{prop:convergence_V}, Lemma \ref{lem:restriction_lem} implies the desired equality.
\end{proof}

The following can be considered as a version of \cite[Proposition 3.5 and Remark 4.9]{DS18}. We will need this later in Subsection \ref{subsec:slicing} for the study of the relationship between Condition $(\mathrm{K}-\max)$ and slicing theory. The proof is straightforward and so omitted.
\begin{proposition}\label{prop:T_to_measure}
	Let $T$ be as in Proposition \ref{prop:compute_tangent_currents}. 
	Let $\varphi$ be a smooth positive closed $(N-p-n, N-p-n)$-form on $U$ such that $\varphi|_V\ne 0$. Then, the family $\left((A_{\lambda})_*(\mathbf{1}_UT\wedge\varphi)\right)_{1\ll|\lambda|}$ admits a unique tangent current $(T\wedge\varphi)_\infty$ along $V$ with the minimal $h$-dimension. Furthermore, its shadow is a Radon measure that satisfies
	\begin{align*}
		(T\wedge\varphi)_\infty^h = T_\infty^h\wedge \varphi|_V.
	\end{align*}
\end{proposition}

\subsection{Relative non-pluripolar product}\label{subsec:general_CondII} We consider another integrability condition related to the existence of the limit $\left\langle T\wedge (dd^cu)^m \right\rangle_\Kmc$ for $m\in \{1, \ldots, \min\{N-p, n\}\}$. 
Here, we allow $n=N$ and do not assume $n<N$. Also, notice that in Subsection \ref{subsec:classical}, it was not clear whether $\mathbf{1}_{U\setminus V}T\wedge(dd^cu)^m$ can be trivially extended across $V$ when $u\in L_{\loc}^1\left(\left\langle T\wedge (dd^cu)^{m-1}\wedge\pi_V^*\omega_V^{N-p-m+1}\right\rangle_\Kmc\right)$, for $m\in\{\max\{n-p, 1\}, \ldots, \min\{N-p, n-1\}\}$.

\begin{definition}[Condition $(\mathrm{K}^*)$]\label{def:cond_NP}
	Let $T\in\Cc_p(U)$. Let $m\in\{1, \ldots, \min\{N-p, n\}\}$. We say that $T$ satisfies Condition $(\mathrm{K}^*)$ along $V$ at codimension $m$ if the Radon measure $-uT\wedge (dd^cu)^{m-1}\wedge\omega^{N-p-m+1}$ on $U\setminus V$ is locally integrable around $V$. We say that $T$ satisfies Condition $(\mathrm{K}^*)$ along $V$ up to codimension $m$ if it satisfies Condition $(\mathrm{K}^*)$ along $V$ in codimensions ranging from $1$ to $m$. When $m=\min\{N-p, n\}$, $T$ is said to satisfy Condition $(\mathrm{K}^*-\max)$ along $V$. 	
\end{definition}


\begin{proposition}\label{prop:Kmc}
	Let $T\in\Cc_p(U)$. Let $m\in\{1, \ldots, \min\{N-p, n\}\}$. Suppose that $T$ satisfies Condition $(\mathrm{K}^*)$ along $V$ at codimension $m$. 
	Then, the limit 
	\begin{align*}
		\langle T\wedge (dd^cu)^m\rangle_\Kmc:=\lim_{\theta\to 0}T\wedge \Kmc^{m}_\theta
	\end{align*}
	exists in the sense of currents on $U$.
\end{proposition}

\begin{proof}
	Since $u<-1$ near $V$, both the currents $T\wedge(dd^cu)^{m-1}$ and $uT\wedge(dd^cu)^{m-1}$ defined on $U\setminus V$ have locally finite mass around $V$. So, they extend trivially across $V$ and the extensions are denoted by $\langle T\wedge (dd^cu)^{m-1}\rangle_0$ and $\langle uT\wedge (dd^cu)^{m-1}\rangle_0$, respectively. The function $u$ is locally integrable with respect to the trace measure of $\langle T\wedge(dd^cu)^{m-1}\rangle_0$.
	\medskip
	
	Let $\varphi$ be a smooth test $(N-p-m, N-p-m)$-form on $U$. Since $V\cap\supp\left(dd^cu_\theta^\Kmc\right)=\emptyset$, we have
	\begin{align*}
		\langle T\wedge \Kmc_\theta^m, \varphi\rangle&=\int T\wedge (dd^cu_\theta^\Kmc)\wedge(dd^cu)^{m-1}\wedge\varphi\\
		&=\int (dd^cu_\theta^\Kmc)\wedge \langle T\wedge(dd^cu)^{m-1}\rangle_0\wedge\varphi=\int u_\theta^\Kmc \langle T\wedge(dd^cu)^{m-1}\rangle_0\wedge dd^c\varphi
	\end{align*}
	and the integrability condition says that the limit converges as $\theta\to 0$.
\end{proof}

Together with Proposition \ref{prop:main_general}, Proposition \ref{prop:Kmc} provides a method to compute the shadow of tangent currents as below. 
\begin{theorem}
\label{thm:shadow_general}
	Let $T\in\Cc_p(U)$ admit a tangent current $T_\infty$ along $V$. Let $h_T$ be the $h$-dimension of $T_\infty$. Suppose that $T$ satisfies Condition $(\mathrm{K}^*)$ along $V$ at codimension $N-p-h_T$. Then, the shadow $T_\infty^h$ of $T_\infty$ can be described as follows:
	\begin{align*}
		&T_\infty^h=\left(\pi_V\right)_*\left(\mathbf{1}_V\langle T\wedge\left(dd^cu\right)^{N-p-h_T}\rangle_\Kmc\right),
	\end{align*}
	where we have $\displaystyle \left\langle \left\langle T\wedge(dd^cu)^{N-p-h_T}\right\rangle_\Kmc, \Phi \right\rangle = \int_{U\setminus V}u T\wedge (dd^cu)^{N-p-h_T-1}\wedge dd^c\Phi$ for a smooth test $(h_T, h_T)$-form $\Phi$ on $U$.
\end{theorem}

The same proof as in Proposition \ref{prop:main_true_general} gives the following theorem.
\begin{theorem}\label{thm:K'}
	Let $T\in\Cc_p(U)$ satisfy Condition $(\mathrm{K^*}-\max)$ along $V$. Suppose that the limit currents $\displaystyle \langle T\wedge\left(dd^cu\right)^m\rangle_\Kmc$ for $m=1, \ldots, \min\{N-p, n\}-1$ have no mass on $V$. Then, there exists a unique tangent current of $T$ along $V$ with the minimal $h$-dimension. Its shadow is computed as in Theorem \ref{thm:shadow_general}. In particular, if $n\le N-p$, we have the same integral representation as in Theorem \ref{thm:uniqueness_tangent_current}.
\end{theorem}

\begin{remark}
	In general, although the limit currents $\left\langle T\wedge(dd^cu)^i\right\rangle_\Kmc$ for $i=1, \ldots, \min\{N-p, n\}$ exist, it is not clear whether the tangent current is uniquely determined.
\end{remark}


\section{Intersection of Positive Closed Currents}\label{sec:intersection}

We apply the existence of the unique tangent current with the minimal $h$-dimension to the study of the intersection--the Dinh-Sibony product--of positive closed currents on domains.\medskip

We recall the Dinh-Sibony product. With the definition of the tangent current as in Definition \ref{def:tangent_current}, the Dinh-Sibony product can be defined on an arbitrary complex manifold.
\begin{definition}[Definition 5.9 in \cite{DS18}]\label{defn:DSproduct}
	Let $X$ be a complex manifold of dimension $n$. Let $S_i\in\Cc_{s_i}(X)$ for $i=1, \ldots, k$, where $1\le s:=s_1+ \cdots+ s_k\le n$. Consider a positive closed current $\pi_1^*S_1\wedge\cdots\wedge \pi_k^*S_k$ on $X^k$, where $\pi_i:X^k\to X$ denotes the canonical projection onto the $i$-th factor for $i=1, \ldots, k$. Let $\Delta$ denote the diagonal submanifold in $X^k$. Assume that there is a unique tangent current $(\pi_1^*S_1\wedge\cdots\wedge \pi_k^*S_k)_\infty$ of $\pi_1^*S_1\wedge\cdots\wedge \pi_k^*S_k$ along $\Delta$ and that its $h$-dimension is minimal. The Dinh-Sibony product of $S_1, \ldots, S_k$ is defined to be the shadow $(\pi_1^*S_1\wedge\cdots\wedge \pi_k^*S_k)^h_\infty$ of $(\pi_1^*S_1\wedge\cdots\wedge \pi_k^*S_k)_\infty$.
\end{definition}

Based on Remark \ref{rmk:independence_of_admissible_map}, throughout this section, we assume $X$ to be a bounded simply connected domain $D$ with smooth boundary in $\C^n$. Let $\omega_{\rm euc}$ denote the Euclidean K\"ahler form on $\C^n$.\medskip

We apply the results in Sections \ref{sec:tangent_currents} and \ref{sec:integrability} to the following case. We take $U=D^k$ and $N=kn$. Let $\pi_i:D^k\to D$ denote the canonical projection onto the $i$-th factor for $i=1, \ldots, k$. We use the coordinates $(x_1,  \ldots, x_k)\in D^k$ so that $x_i\in D\subset \C^n$ and $\pi_i(x_1, \ldots, x_k)=x_i$ for $i=1,\ldots, k$. We take as $V$ the diagonal submanifold $\Delta$ of $D^k$ defined by $\Delta:=\{(x,  \ldots, x)\in D^k: x\in D\}$. We take $E=\Delta\times \C^{(k-1)n}$ and $\overline{E}=\Delta\times \P^{(k-1)n}$. Let $\pi_\Delta: \overline{E}\to \Delta$ and $\pi_F:\overline{E}\to \P^{(k-1)n}$ denote the canonical projections onto $\Delta$ and $\P^{(k-1)n}$, respectively. We use coordinates $(x, w_1, \ldots, w_{k-1})\in E$ and $(x, [w_1: \cdots: w_{k-1}: t])\in \overline{E}$. We identify $(x, w_1, \ldots, w_{k-1})\in E$ with $(x, [w_1: \cdots: w_{k-1}: 1])\in \overline{E}$. Then, we have $\pi_{\Delta}\left(x, [w_1: \cdots: w_{k-1}: t]\right)=x$ and $\pi_F\left(x, [w_1: \cdots: w_{k-1}: t]\right)=[w_1: \cdots: w_{k-1}: t]$. We regard $D^k$ as a subset of $\Delta\times \C^{(k-1)n}$. The inclusion map $\iota:D^k\to E$ is written as $\iota(x_1, \cdots, x_k)=(x_k, x_1-x_k, \ldots, x_{k-1}-x_k)$. We use $\omega_{\Delta}=dd^c|x|^2$ on $\Delta$, $\omega_F=\frac{1}{2}dd^c\log\left(\sum_{i=1}^{k-1}|w_i|^2+|t|^2\right)$ on $\P^{(k-1)n}$ and $\omega=\pi_{\Delta}^*\omega_{\Delta}+\pi_F^*\omega_F$ on $\overline{E}$. We use 
$\omega_{D^k}=\pi_1^*\omega_{\rm euc}+\cdots+\pi_k^*\omega_{\rm euc}$. In these coordinates, we may say $\pi_{\Delta}^*\omega_{\Delta}=\pi_k^*\omega_{\rm euc}$ in $D^k\subset \overline{E}$. We take $\displaystyle u=\frac{1}{2}\log \sum_{i=1}^{k-1} |x_i-x_k|^2$. We consider $S_i\in\Cc_{s_i}(D)$ for $i=1, \ldots, k$ and $T=\pi_1^*S_1\wedge\cdots\wedge \pi_k^*S_k$, where $1\le s:=s_1+\cdots+ s_k\le n$.\medskip

Condition $(\mathrm{K}-\max)$ for the existence of the unique tangent current with the minimal $h$-dimension leads to a condition for the Dinh-Sibony product of $S_i\in\Cc_{s_i}(D)$ for $i=1, \ldots, k$. 
Note that $\pi_1^*S_1 \wedge\cdots\wedge\pi_k^*S_k$ is well-defined as $\pi_1^*S_1\otimes\cdots\otimes\pi_k^*S_k$. For a rigorous treatment of this, refer to the notion of double current in \cite{deRham}.
\begin{definition}\label{defn:K-integrable_many}
	We say that $S_1, \ldots, S_k$ satisfy Condition $(\mathrm{I})$ if $\pi_1^*S_1\wedge\cdots \wedge\pi_k^*S_k$ satisfies Condition $(\mathrm{K}-\max)$ along $\Delta$ as in Definition \ref{defn:Condition(I)}.
\end{definition}

\begin{remark}
	In our coordinates, for each $j=1, \ldots, k$, Condition $(\mathrm{I})$ is equivalent to saying
	\begin{align*}
		u\in L_{\loc}^1\left(\left\langle \pi_1^*S_1\wedge\cdots\wedge \pi_{j}^*\left(S_j\wedge \omega_{\rm euc}^{kn-s-i+1}\right)\wedge\cdots \wedge\pi_k^*S_k\wedge (dd^cu)^{i-1}\right\rangle_\Kmc\right)
	\end{align*}
	in $D^k$ inductively from $i=1$ through $(k-1)n$.
\end{remark}

Propositions \ref{prop:compute_tangent_currents} and \ref{prop:convergence_V} can be interpreted as follows:
\begin{proposition}\label{prop:1stMainProp}
	Let $S_1, \ldots, S_k$ satisfy Condition $(\mathrm{I})$.
	Let $\varphi$ be a smooth test $(n-s, n-s)$-form with compact support on $D$. Let $\Phi$ be a smooth $(n-s, n-s)$-test form with compact support on $D^k$ such that $\Phi=\pi_k^*(\varphi)$ in a neighborhood of $\Delta$. 
	Then, the following integral is well-defined and its value is independent of the choice of $\Phi$.
	\begin{align}\label{eq:King_wedge}
		\left\langle \left(S_1\wedge \cdots\wedge S_k\right)_K, \varphi\right\rangle:=\int_{D^k\setminus \Delta} \pi_1^*S_1 \wedge \cdots \wedge \pi_k^*S_k\wedge\left(u(dd^cu)^{(k-1)n-1}\right)\wedge dd^c\Phi,
	\end{align}
	which defines a positive closed $(s, s)$-current on $D$.
\end{proposition}

\begin{definition}\label{defn:King_Def_wedge}
	Let $S_1, \ldots, S_k$ satisfy Condition $(\mathrm{I})$. The wedge product of $S_1, \ldots, S_k$ in the sense of King's residue formula is defined by \eqref{eq:King_wedge} in Proposition \ref{prop:1stMainProp} and denoted by $\left(S_1\wedge\cdots\wedge S_k\right)_K$.
\end{definition}

Theorem \ref{thm:uniqueness_tangent_current} gives us the main theorem of this work on the intersection of positive closed currents, which shows that Definition \ref{defn:King_Def_wedge} is actually intrinsic.
\smallskip

\noindent{\bf Theorem \ref{thm:main_intersection_1}.} {\it
	Let $S_1, \ldots, S_k$ satisfy Condition $(\mathrm{I})$. Then, there exists a unique tangent current of $\pi_1^*S_1 \wedge \ldots\wedge \pi_k^*S_k$ along $\Delta$ with the minimal $h$-dimension. Its shadow is exactly the positive closed $(s, s)$-current $\left(S_1\wedge\ldots\wedge S_k\right)_K$ on $D$, where we have
\begin{align*}
	\left\langle \left(S_1\wedge \cdots\wedge S_k\right)_K, \varphi\right\rangle:=\int_{D^k\setminus \Delta} \pi_1^*S_1 \wedge \cdots \wedge \pi_k^*S_k\wedge\left(u(dd^cu)^{(k-1)n-1}\right)\wedge dd^c\Phi
\end{align*}
for a smooth test $(n-s, n-s)$-form $\varphi$ on $D$ and for a smooth $(n-s, n-s)$-form $\Phi$ with compact support on $D^k$ such that $\Phi=\pi_k^*\varphi$ in a neighborhood of $\Delta$. In particular, the Dinh-Sibony product of $S_1, \ldots, S_k$ is well-defined and equal to $\left(S_1\wedge \cdots \wedge S_k\right)_K$.
}
\medskip

The following proposition is straightforward.
\begin{proposition}
	Suppose that $S_1, \ldots, S_{j-1}, S_{j+1}, \ldots, S_k$ are smooth for some $j=1, \ldots, k$. Then $S_1, \ldots, S_k$ satisfy Condition $(\mathrm{I})$ and we have $\left(S_1\wedge \cdots \wedge S_k\right)_K=S_1\wedge\cdots\wedge S_k$.
\end{proposition}

The following looks intuitively obvious and the proof is straightforward. So, we omit its proof. The general case where $S_3$ is not smooth needs to be further investigated.
\begin{proposition}\label{prop:smooth_associativity}
	Suppose that $S_1\in\Cc_{s_1}(D)$ and $S_2\in\Cc_{s_2}(D)$ satisfy Condition $(\mathrm{I})$  and $S_3\in \Cc_{s_3}(D)$ is smooth, where $1\le s_1+s_2+s_3\le n$. Then, $S_1$, $S_2$, $S_3$ satisfy Condition $(\mathrm{I})$ , $S_1$ and $\left(S_2\wedge S_3\right)$ satisfy Condition $(\mathrm{I})$ , and we have
	\begin{align*}
		\left(S_1\wedge S_2\wedge S_3\right)_K =\left(S_1\wedge \left( S_2\wedge S_3\right)\right)_K=\left(\left(S_1\wedge S_2\right)_K\wedge S_3\right)_K.
	\end{align*}
\end{proposition}

\section{Tangent Currents and Intersections of Positive Closed Currents \\on Complex Manifolds}\label{sec:compact}

We apply our results to 
general complex manifolds including compact K\"ahler manifolds. In addition, we present a local version of superpotentials.
\subsection{Tangent currents and intersections of positive closed currents} Let $X$ be a complex manifold of dimension $N$ and $V\subset X$ a complex submanifold of codimension $n$. 
Let $T\in\Cc_p(X)$. 
\begin{definition}
	A decent cover of $V$ in $X$ is a collection $(U_i)_{i\in I}$ of open subsets of $X$ with the following properties:
	\begin{enumerate}
		\item each $U_i$ is biholomorphic to a domain $D_{V,i}$ in $\C^N$ via $\phi_i:U_i\to D_{V,i}$ such that $\phi_i(V\cap U_i)$ is a bounded simply connected linear complex submanifold of $\phi_i(U_i)$ with smooth boundary;
		\item we have $V\subset\bigcup_{i\in I}U_i$;
		\item for each $x\in V$, there are only finitely many $i\in I$ with $x\in U_i$.
	\end{enumerate}
	For simplicity, a decent cover of $X$ in $X$ is called a decent cover of $X$.
\end{definition}
We can always find a decent cover of $V$ in $X$. The statements in the rest of this section are straight from Remark \ref{rmk:independence_of_admissible_map} and corresponding statements in Sections \ref{sec:tangent_currents} and \ref{sec:integrability}.
\begin{theorem}
	Suppose that there exists a decent cover $(U_i)_{i\in I}$ of $V$ in $X$ such that on each $U_i$, Theorem \ref{thm:existence_tangent_currents} holds. Then, there exists a tangent current of $T$ along $V$.
\end{theorem}

We can also describe the $h$-dimension in terms of the local language. Below is an example induced from Proposition \ref{prop:h-dim}.
\begin{proposition}
	Let $T$ admit a tangent current $T_\infty$ along $V$. Suppose that there exist an $m\in \{\max\{n-p, 0\}, \ldots, \min\{N-p, n\}\}$ and a decent cover $(U_i)_{i\in I}$ of $V$ in $X$ such that the current $T$ satisfies Condition $(\mathrm{K})$ along $V$ up to codimension $m$ on $U_i$ for each $i\in I$. Then, the $h$-dimension of $T_\infty$ is at most $N-p-m$. In particular, if $T$ satisfies Condition $(\mathrm{K}-\max)$ along $V$ on $U_i$ for each $i\in I$, then the $h$-dimension of $T_\infty$ is minimal. 
\end{proposition}

From Theorem \ref{thm:uniqueness_tangent_current}, we obtain the following:
\begin{theorem}
	 Suppose that there exists a decent cover $(U_i)_{i\in I}$ of $V$ in $X$ such that on each $U_i$, $T$ satisfies Condition $(\mathrm{K}-\max)$ along $V$. Then, there exists a unique tangent current $T_\infty$ of $T$ along $V$ and its $h$-dimension is minimal. Its local representation is given as in Theorem \ref{thm:uniqueness_tangent_current}.
\end{theorem}

In particular, the intersection of positive closed currents can also be described. To avoid potential confusion, in this section, we denote by $Y$ a complex manifold of dimension $n_Y$, on which we consider the intersection of positive closed currents. 
\begin{theorem}\label{thm:intersection_cpt}
	Let $S_i\in \Cc_{s_i}(Y)$ for $i=1, \ldots, k$, where $1\le s:=s_1+\cdots+s_k\le n_Y$. Let $(U_j)_{j\in J}$ be a locally finite cover of $Y$ such that each $U_j$ is biholomorphic to a bounded simply connnected domain in $\C^{n_Y}$ with smooth boundary. Suppose that on each $U_j$, $S_1, \ldots, S_k$ satisfy Condition $(\mathrm{I})$. Then, the Dinh-Sibony product of $S_1, \ldots, S_k$ is well-defined. Its local representation is as in Theorem \ref{thm:main_intersection_1}.
\end{theorem}

\begin{remark}\label{rmk:singularity_K}
	As an analogy, one can view the Green quasi-potential kernel in \cite{DS09} as a holomorphic image of a quasi-potential of $\delta_0-\eta^n=0$ on $\P^n$, where $\delta_0$ is the Dirac measure at the origin $0\in\C^n\subset \P^n$ and $\eta$ is a smooth positive closed $(1, 1)$ form cohomologous to a generic hyperplane. For the coordinates $[z_1: \cdots: z_n: t]\in\P^n$, King's residue formula says that $\delta_0=\left(dd^c u\right)^n$ in the sense of currents, where $u=\log\sqrt{\sum_{i=1}^n|z_i|^2}$. We have $\delta_0-\eta^n=\sum_{i=1}^n\left(dd^cu-\eta\right)\wedge(dd^cu)^{i-1}\wedge\eta^{n-i}$ and its quasi-potential is the sum of $\widetilde{u}(dd^cu)^{i-1}\wedge\eta^{n-i}$ for $i=1, \ldots, n$, where $\widetilde{u}$ is a quasi-potential of $dd^cu-\eta$ so that $u - \widetilde{u}$ is a smooth function on $\P^n$. Condition $(\mathrm{I})$ may be regarded as a condition in terms of the collection $\left(\widetilde{u}(dd^cu)^{i-1}\wedge\eta^{n-i}\right)_{i=1, \ldots, n}$. 
	We expect that the collection $\left(\Fc_{S, j}^i\right)_{i\in\{1, \ldots, n\}, j\in J}$ of local potential superfunctions, introduced below, is to a superpotential as the collection of locally defined plurisubharmonic functions is to a globally defined quasi-plurisubharmonic function.
\end{remark}

\begin{remark}
	Other results such as Theorems \ref{thm:h_dim_K}, \ref{thm:shadow_general} and \ref{thm:K'} in Sections \ref{sec:tangent_currents} and \ref{sec:integrability} also generalize to general complex manifolds in the same way as above.
\end{remark}

\subsection{Local potential superfunctions}\label{subsec:superfunctions}
We introduce a collection of functions associated with a positive closed current, which resembles superpotentials on complex projective spaces. The difference is that its regularity is not as strong as that of superpotentials on complex projective spaces since it is not clear how Stokes' theorem can be applied in our local approach.\medskip

Let $Y$ be a complex manifold of dimension $n_Y$ as previously. Let $\mathfrak{Y}=Y_1\times Y_2$ denote the product of two copies of $Y$ and $\pi_l:\mathfrak{Y}\to Y_l$ the canonical projection onto the $l$-th factor for $l=1, 2$. Let $\Delta$ denote the diagonal submanifold of $\mathfrak{Y}$. Let $\left(U_j^i\right)_{j\in J}$ for $i=0, \ldots, n$ be coverings of $Y$ such that each $U_j^i$ is biholomorphic to a bounded simply connected domain in $\C^n$ with smooth boundary and that $\overline{U_j^i}\subset U_j^{i-1}$ for each $j\in J$ and for each $i=1, \ldots, n$. For each $j\in J$, we use the same coordinate chart for $U_j^0, \ldots, U_j^{n}$ and on this coordinate chart, we consider the Euclidean K\"ahler form $\omega_{\rm euc}$ and the function $u$ as defined in Section \ref{sec:intersection} when $k=2$. For each $i=1, \ldots, n$ and for each $j\in J$, let $\chi_j^i:U_j^{i-1}\times U_j^{i-1}\to [0,1]$ be a smooth function with compact support satisfying $\overline{U_j^i\times U_j^i}\subset\{\chi_j^i\equiv 1\}$. For notational convenience, we denote by $[U_j^i]$ and $[\chi_j^i]$ the families $\left(U_j^i\right)_{i\in \{0, 1, \ldots, n\}, j\in J}$ and $\left(\chi_j^i\right)_{i\in\{1, \ldots, n\}, j\in J}$, respectively. We can always find such $[U_j^i]$ and $[\chi_j^i]$.
\medskip

Let $S\in\Cc_s(Y)$. Let $j\in J$. We consider the following superfunctions, by which we mean functions whose domain and range are a subset of the space of positive closed currents and a subset of the set of extended real numbers, respectively. For $R\in\Cc_{n-s+1}(Y)$, we define
\begin{align*}
	\Fc_{S, j, \theta}^1(R):=\int_{U^0_j\times U^0_j} \chi_j^1u_\theta^\Kmc  \pi_1^*S\wedge \pi_2^*(R\wedge \omega_{\rm euc}^{n-1}).
\end{align*}
This family of functions is decreasing as $|\theta|$ decreases to $0$. We allow $-\infty$ as value and we take the limit as $\theta\to 0$. We denote by $\Fc_{S, j}^1$ the limit function and by $\Dc_{S, j}^1:=\{R\in\Cc_{n-s+1}(Y): \Fc_{S, j}^1(R)>-\infty\}$ its domain of finite values. For $R\in\Dc_{S, j}^1$, according to Lemma \ref{lem:basic_conv_KT}, $\big\langle \pi_1^*S\wedge \pi_2^*\left(R\wedge\omega_{\rm euc}^{n-2}\right)\wedge dd^cu\big\rangle_\Kmc$ is well-defined on $U_j^1\times U_j^1$. Then, for $R\in \Dc_{S, j}^1$ (not on the whole space of $\Cc_{n-s+1}(Y)$), we define
\begin{align*}
	\Fc_{S, j, \theta}^2(R):=\int_{U_j^1\times U_j^1} \chi_j^2u_\theta^\Kmc  \left\langle\pi_1^*S\wedge \pi_2^*\left(R\wedge \omega_{\rm euc}^{n-2}\right)\wedge dd^cu \right\rangle_\Kmc,
\end{align*}
which decreases as $|\theta|$ decreases. As previously, we take the limit and denote by $\Fc_{S, j}^2$ the limit function on $\Dc_{S, j}^1$; we define its associated set $\Dc_{S, j}^2$. Inductively, we define $\Fc_{S, j, \theta}^i$, $\Fc_{S, j}^i$ and $\Dc_{S, j}^i$ for $i=1, \ldots, n$. For $i=1, \ldots, n$ and for $R\in \Dc_{S, j}^i$, we have
\begin{align*}
	\Fc_{S, j, \theta}^i(R)&=\int_{U_j^{i-1}\times U_j^{i-1}} \chi_j^iu_\theta^\Kmc  \left\langle\pi_1^*S\wedge \pi_2^*\left(R\wedge \omega_{\rm euc}^{n-i}\right)\wedge (dd^cu)^{i-1} \right\rangle_\Kmc\\
	&=\int_{\left(U_j^{i-1}\times U_j^{i-1}\right)\setminus \Delta} \chi_j^iu_\theta^\Kmc\left(dd^cu\right)^{i-1}\wedge\pi_1^*S\wedge \pi_2^*\left(R\wedge \omega_{\rm euc}^{n-i}\right)
\end{align*}
and
\begin{align*}
	\Fc_{S, j}^i(R)&=\int_{U_j^{i-1}\times U_j^{i-1}} \chi_j^iu\left\langle\pi_1^*S\wedge \pi_2^*\left(R\wedge \omega_{\rm euc}^{n-i}\right)\wedge (dd^cu)^{i-1} \right\rangle_\Kmc\\
	&=\int_{\left(U_j^{i-1}\times U_j^{i-1}\right)\setminus \Delta} \chi_j^iu\left(dd^cu\right)^{i-1}\wedge\pi_1^*S\wedge \pi_2^*\left(R\wedge \omega_{\rm euc}^{n-i}\right).
\end{align*}

Each $\Fc^i_{S, j, \theta}$ and $\Fc^i_{S, j}$ have the following regularity property.
\begin{proposition}\label{prop:usc_local}
	For $\theta\in\C^*$ with $|\theta|\ll1$, for $i\in\{1, \ldots, n\}$ and for $j\in J$, $\Fc^i_{S, j, \theta}$ and $\Fc^i_{S, j}$ are upper-semicontinuous on $\Dc_{S, j}^i$ in the following sense. Let $R\in\Dc_{S,j}^i$ be a current and $\left(R_k\right)_{k\in\N}\subset\Dc^i_{S,j}$ be a sequence of currents such that $R_k\to R$ as $k\to\infty$ in the sense of currents. Then, we have
	\begin{align*}
		\limsup_{k\to \infty}\Fc_{S, j, \theta}^i(R_k)\le \Fc_{S, j, \theta}^i(R)\quad\textrm{ and }\quad\limsup_{k\to \infty}\Fc_{S, j}^i(R_k)\le \Fc_{S, j}^i(R).
	\end{align*}
\end{proposition}
\begin{proof}
	It is enough to show that $\Fc^i_{S, j, \theta}(R)$ is upper-semicontinuous on $\Dc_{S, j}^i$ as $\Fc_{S, j}^i$ is the decreasing limit of $\left(\Fc_{S, j, \theta}^i(R)\right)_{0<|\theta|\ll 1}$ on $\Dc^i_{S, j}$.
	\medskip
	
	If $\limsup_{k\to \infty}\Fc_{S, j, \theta}^i(R_k)=-\infty$, then the inequality trivially holds. Since $R\in\Dc_{S,j}^i$, we may assume that there exists $M_R>0$ satisfying $\Fc_{S, j, \theta}^i(R_k)>-M_R$ for all $k\in\N$ and $\Fc_{S, j, \theta}^i(R)>-M_R$.\medskip
	
	Since $\Fc_{S, j, \theta}^i(R_k)>-M_R$ for all $k\in\N$, the sequence $\displaystyle \left( \chi_j^iu_\theta^\Kmc  \left\langle\pi_1^*S\wedge \pi_2^*\left(R_k\wedge \omega_{\rm euc}^{n-i}\right)\wedge (dd^cu)^{i-1} \right\rangle_\Kmc\right)_{k\in\N}$ of Radon measures has uniformly bounded mass. Their supports sit in a fixed compact subset $\supp\,\chi_j^i$. Hence, there exists a convergent subsequence and we will denote by $L$ its limit. Since $L$ is negative, $L$ can be decomposed into $L=\mathbf{1}_\Delta L+ \mathbf{1}_{\left(U_j^{i-1}\times U_j^{i-1}\right)\setminus \Delta} L$. Since $dd^cu$ is smooth outside $\Delta$ and $\chi_j^i u_\theta^\Kmc \left\langle\pi_1^*S\wedge \pi_2^*\left(R\wedge \omega_{\rm euc}^{n-i}\right)\wedge (dd^cu)^{i-1} \right\rangle_\Kmc$ has no mass on $\Delta$, we have
	\begin{align*}
		\mathbf{1}_{\left(U_j^{i-1}\times U_j^{i-1}\right)\setminus \Delta} L=\chi_j^iu_\theta^\Kmc \left\langle\pi_1^*S\wedge \pi_2^*\left(R\wedge \omega_{\rm euc}^{n-i}\right)\wedge (dd^cu)^{i-1} \right\rangle_\Kmc.
	\end{align*}
	Since $\mathbf{1}_\Delta L$ is negative, we see that $L\le \Fc_{S, j, \theta}^i(R)$.
\end{proof}

\begin{definition}
	Given localizing data $\left([U_j^i], [\chi_j^i]\right)$, we call $\left(\Fc_{S, j}^i\right)_{i\in\{1, \ldots, n\}, j\in J}$ the collection of local potential superfunctions associated with $\left([U_j^i], [\chi_j^i]\right)$, and $\Fc_{S, j}^i$ is called the local potential superfunction on $U_j^i$. If the localizing data are understood, we simply call it a collection of local potential superfunctions.
\end{definition}

The following two propositions are paraphrases of Theorem \ref{thm:intersection_cpt} when $k=2$.
\begin{proposition}
	Let $S\in\Cc_s(Y)$ and $R\in\Cc_r(Y)$, where $1\le s+r\le n$. Then, $Y$ admits a locally finite cover such that Theorem \ref{thm:intersection_cpt} holds with $k=2$ if and only if there exist localizing data $\left([U_j^i], [\chi_j^i]\right)$ such that $\Fc_{S, j}^i(R\wedge \omega_{\rm euc}^{n-s-r+1})>-\infty$ for every $i\in \{1, \ldots, n\}$ and for every $j\in J$. 
\end{proposition}

Further, for each $i=1, \ldots, n$, we define
\begin{align*}
	\Dc_S^i:=\bigcap_{j\in J}\Dc_{S, j}^i.
\end{align*}
By convention, $\Dc_S^0=\Cc_{n-s+1}(Y)$. By construction, we have $\Dc_S^{i+1}\subseteq \Dc_S^i$. So, $\Dc_S^n$ can be used to define the collection of currents satisfying Condition $(\mathrm{I})$ with the current $S$. The following proposition is straightforward.
\begin{proposition}\label{prop:cpt_supftn}
	Let $(Y,\omega_Y)$ be a (not necessarily compact) K\"ahler manifold of dimension $n_Y$. Let $S\in\Cc_s(Y)$ and $R\in\Cc_r(Y)$, where $1\le s+r\le n$. Then, $Y$ admits a locally finite cover such that Theorem \ref{thm:intersection_cpt} holds with $k=2$ if and only if there exist localizing data $\left([U_j^i], [\chi_j^i]\right)$ such that
	$$R\wedge \omega_Y^{n-s-r+1}\in\Dc_S^n.$$
	In this case, for a smooth test $(n-s-r, n-s-r)$-form $\varphi$ on $Y$, we have
	\begin{align*}
		\left\langle (S\wedge R)_K, \varphi\right\rangle = \sum_{j\in J}\left(\Fc_{S, j}^n(R\wedge (M_{\varphi, j} \omega_Y^{n-s-r+1}+dd^c(\chi_j\varphi)))-\Fc_{S, j}^n(R\wedge (M_{\varphi, j} \omega_Y^{n-s-r+1}))\right),
	\end{align*}
	where the collection $(\chi_j)_{j\in J}$ is a partition of unity subordinate to the cover $\left(U^n_j\right)_{j\in J}$ and $M_{\varphi, j}>0$ is a constant such that $M_{\varphi, j}\omega_Y^{n-s-r+1}+dd^c\left(\chi_j\varphi\right)$ is positive in the sense of forms for each $j\in J$.
\end{proposition}

\section{Regularization of Positive Closed Currents}\label{sec:regularization}
In this section, we consider the compatibility of regularizations of positive closed currents with tangent currents. In particular, we study the relationships among Condition $(\mathrm{I})$, tangent currents, slicing and King's work in \cite{King-1}. Throughout this section, we consider a bounded simply connected domain $D$ with smooth boundary in $\C^n$.
\medskip

In this section, it is convenient to consider coordinates for $\C^n$ and $\C^n\times \C^n$ instead of $D$ and $D^2$. It fits our approach without any issues. We use $x$ for the coordinates of $\C^n$ and $\omega_{\rm euc}=dd^c|x|^2$ for the Euclidean K\"ahler form on $\C^n$. 
Let $\pi_i:{\C^n\times\C^n}\to \C^n$ be the canonical projection onto the $i$-th factor for $i=1, 2$. We will use $(x, y)$ for the coordinates of ${\C^n\times\C^n}$ so that we have $\pi_1(x, y)=x$ and $\pi_2(x, y)=y$. 
Let $\Delta:=\{(x, x)\in {\C^n\times\C^n}\}$ be the diagonal submanifold. We will also use the K\"ahler form $\omega_{x,y}=\pi_1^*\omega_{\rm euc}+\pi_2^*\omega_{\rm euc}$ on $\C^n\times\C^n$.
\subsection{Semi-regular transformations on a domain} Semi-regular transforms on compact K\"ahler manifolds were introduced and used by Dinh-Sibony (\cite{DS10}). Here, we consider essentially the same notion on domains. For our purpose, we define it in terms of regularity instead of blow-up. 
\medskip

Let $q\in \{0, \ldots, n\}$. Let $Q$ be a $(q, q)$-form on $\C^n\times \C^n$ smooth outside $\Delta$ such that for some constant $M_Q>0$, we have
\begin{align}\label{eq:condn_Q}
	Q\le M_Q\omega_{x,y},\,\,|Q|\le -M_Q[\log\dist(\cdot, \Delta)]\dist(\cdot, \Delta)^{2-2n}\,\,\textrm{ and }\,\, |\nabla Q|\le M_Q\dist(\cdot, \Delta)^{1-2n},
\end{align}
where $|Q|$ and $|\nabla Q|$ mean the sum of the absolute values of the coefficients of $Q$ and the sum of the absolute values of the derivatives of the coefficients of $Q$, respectively, and the first inequality is understood in the sense of currents.
\medskip

The following proposition is quite straightforward.
\begin{proposition}\label{prop:semi-regular_integral_rep}
	Let $s\in \{n-q, \ldots, n\}$. Let $S$ be a smooth $(s, s)$-form with compact support in $D$. The form $(\pi_2)_*(\pi_1^*S\wedge Q)$ is a well-defined $(s+q-n, s+q-n)$-form in $\C^n$ and has $C^1$ coefficients. If $Q$ can be written as a form of $x-y=(x_1-y_1, \ldots, x_n-y_n)$, then $(\pi_2)_*(\pi_1^*S\wedge Q)$ has $C^\infty$ coefficients.
	For each $y\in\C^n$, it is computed as
	\begin{align*}
		\Lc^Q(S):=(\pi_2)_*(\pi_1^*S\wedge Q)(y)=\int_{x\in D\setminus \{y\}} S(x)\wedge Q(x, y).
	\end{align*}
\end{proposition}

When $S$ is as in Proposition \ref{prop:semi-regular_integral_rep}, for any smooth test $(2n-s-q, 2n-s-q)$-form $\varphi$ on $D$, we have 
\begin{align}\label{eq:semi-regular_integral}
	\left\langle \Lc^Q(S), \varphi\right\rangle=\int_{(x,y)\in {D^2}\setminus \Delta} S(x)\wedge \varphi(y)\wedge Q(x, y)=\left\langle S, \Lc^{\rho^*Q}(\varphi)\right\rangle,
\end{align}
where $\rho:\C^n\times \C^n\to \C^n\times \C^n$ is the involution defined by $\rho(x, y)=(y, x)$. So, for a general current $S$, we define the semi-regular transform as below:
\begin{definition}\label{defn:general_Q}
	Let $S\in\Cc_s(D)$ have bounded mass on $D$, where $n-q\le s\le n$. The semi-regular transform $\Lc^Q(S)$ of $S$ is an $(s+q-n, s+q-n)$-current defined by
	\begin{align*}
		\left\langle \Lc^Q(S), \varphi\right\rangle:=\left\langle S, \mathbf{1}_D\Lc^{\rho^*Q}(\varphi)\right\rangle
	\end{align*}
	for a smooth test $(2n-s-q, 2n-s-q)$-form $\varphi$ on $D$, where $\rho:\C^n\times \C^n\to \C^n\times \C^n$ is the involution defined by $\rho(x, y)=(y, x)$. The bidegree of the transform is defined to be $(q-n, q-n)$.
	\medskip
	
	If $Q$ is smooth, negative, or closed, then its associated transform $\Lc^Q(\cdot)$ is said to be smooth, negative, or closed, respectively.
\end{definition}
Note that the pairing on the right hand side makes sense since $\Lc^{\rho^*Q}(\varphi)$ has $C^1$ coefficients on $\C^n$ and $S$ is a positive current with bounded mass. Such $\Lc^Q(S)$ defines a current.
\medskip

The following two propositions can be proved via direct computations.
\begin{proposition}\label{prop:Q_representations}
	Let $S$ be as in Definition \ref{defn:general_Q}. Let $\varphi$ be a smooth test $(2n-s-q, 2n-s-q)$-form on $D$. The integrals
	\begin{align*}
		\int_{{D^2}} \pi_1^*S\wedge \pi_2^*\varphi \wedge Q\quad\textrm{ and }\quad\int_{{D^2}\setminus\Delta} \pi_1^*S\wedge \pi_2^*\varphi \wedge Q
	\end{align*}
	are well-defined and equal to $\left\langle\Lc^Q(S), \varphi\right\rangle$. In particular, we can write
	\begin{align*}
		\Lc^Q(S)=\int_{x\in D\setminus\{y\}}S(x)\wedge Q(x, y)=(\pi_2)_*(\pi_1^*S\wedge Q).
	\end{align*}
\end{proposition}

\begin{remark}
	When $S$ is positive, $\Lc^Q(S)$ is a form with $L^{1+1/(n-1)}$-coefficients. For instance, see \cite[Theorem 2.3.1]{DS10}.
\end{remark}

\begin{proposition}
	When $Q$ is smooth, then $\Lc^Q(S)$ is smooth.
\end{proposition}

We recall the standard regularization by convolution for later use. Let $\Cc_s(\overline{D})$ be the set of positive closed $(s, s)$-currents defined in a neighborhood of $\overline{D}$. Let $S\in\Cc_s(\overline{D})$. Let $g:\C^n\to\R_{\ge 0}$ be a smooth function with compact support in the unit ball such that $\int_{\C^n} g\,d\mu=1$, where $\mu$ is the standard Lebesgue measure. For $z\in\C^n$, we define $\tau_z(x)=x-z$. Let $\epsilon>0$ be sufficiently small so that in the $\epsilon$-neighborhood of $\overline{D}$, the current $S$ is well-defined. The standard convolution formula gives us
\begin{align*}
	S_\epsilon(x) = \int_{z\in\C^n}\tau_z^*S(x) \left(\frac{1}{\epsilon^n}g\left(\frac{z}{\epsilon}\right)\right)d\mu(z).
\end{align*}
\begin{proposition}\label{prop:continuity_transform}
	Let $S\in\Cc_s(\overline{D})$. 
	Then, $\lim_{\epsilon \to 0}\Lc^Q(S_\epsilon)=\Lc^Q(S)$ in $D$ in the sense of currents.
\end{proposition}

\subsection{A regularization of positive closed currents}

From King's residue formula in \cite{King}, 
we have
\begin{align}\label{eq:king}
	(dd^c u)^n=[\Delta]
\end{align}
in the sense of currents in $\C^n\times\C^n$, where we have $u=\log |x-y|$.
\medskip

We consider two approximations of $u$ in $\C^n\times\C^n$ as in Section \ref{sec:tangent_currents}: $$u_\theta^\Kmc :=\chi(u-\log |\theta|)+\log |\theta|\quad\textrm{ and }\quad u_\theta^\Tmc :=\frac{1}{2}\log\left(|x-y|^2+|\theta|^2\right),$$
where $\theta\in\C^*$ such that $|\theta|\ll 1$. Also, the following forms were defined there:
\begin{enumerate}
	\item $\Kmc^i_\theta:=(dd^c u_\theta^\Kmc )\wedge(dd^c u)^{i-1}$ for $i=1, \cdots, n$;
	\item $\Tmc^i_\theta:=(dd^c u_\theta^\Tmc )^i$ for $i=1, \cdots, n$.
\end{enumerate}
Observe that the forms (1) and (2) are all invariant under the involution $\rho(x, y)=(y, x)$. All the kernels define semi-regular transforms of currents due to their regularity properties.
\begin{proposition}\label{prop:king_general}
	Let $S\in\Cc_s(D)$. 
	Then, the current $\Lc^{\Kmc_\theta^n}(S)$ is smooth, positive and of bidegree $(s, s)$. The family $\left(\Lc^{\Kmc_\theta^n}(S)\right)_{0<|\theta|\ll 1}$ converges to $S$ as $\theta\to 0$ in the sense of currents. 
	On a relatively compact open subset $D'$ of $D$, the restriction of $\Lc^{\Kmc_\theta^n}(S)$ to $D'$ is a smooth positive closed $(s, s)$-current in $D'$ for all $\theta\in\C^*$ with $|\theta|\ll 1$. 
\end{proposition}

\begin{proof}
	The smoothness and positivity of $\Lc^{\Kmc_\theta^n}(S)$ are straightforward. It suffices to prove that for a relatively compact open subset $D'$ of $D$, the restriction of $\Lc^{\Kmc_\theta^n}(S)$ to $D'$ is a smooth positive closed $(s, s)$-current when $|\theta|\ll 1$ and $\Lc^{\Kmc_\theta^n}(S)\to S$ in $D'$ as $\theta\to 0$.\medskip
	
	We first consider smooth $S\in\Cc_s(D)$. Let $\varphi$ be a smooth test $(n-s, n-s)$-form on $D$ such that $\supp\,\varphi\subset D'$. Let $D''$ be another relatively compact open subset of $D$ such that $\overline{D'}\subset D''$. Let $\chi_{D'}:D\to [0,1]$ be a smooth function with compact support in $D''$ such that $\{\chi_{D'}\equiv 1\}$ contains a neighborhood of $\overline{D'}$. We assume that $|\theta|$ be sufficiently small that $\left(\overline{D'}\right)_\theta\subset \{\chi_{D'}\equiv 1\}$. Since $(\pi_1^*\chi_{D'})(\pi_2^*\varphi)$ is a smooth form with compact support in ${{D^2}}$, we have
	\begin{align*}
		\left\langle \Lc^{\Kmc_\theta^n}(S), \varphi\right\rangle &:= \left\langle \mathbf{1}_{D} S, \Lc^{\Kmc_\theta^n}(\varphi)\right\rangle = \left\langle \chi_{D'} S, \Lc^{\Kmc_\theta^n}(\varphi)\right\rangle=\int_{{D^2}}\pi_1^*(\chi_{D'} S)\wedge \pi_2^* \varphi\wedge \Kmc_\theta^n\\
		&=\int_{{D^2}} u_\theta^\Kmc \pi_1^*S \wedge dd^c[(\pi_1^*\chi_{D'})(\pi_2^*\varphi)]\wedge(dd^c u)^{n-1}.
	\end{align*}
	In the same way, we have
	\begin{align*}
		\left\langle S, \varphi\right\rangle = \int_{{D^2}} u\pi_1^*S \wedge dd^c[(\pi_1^*\chi_{D'})(\pi_2^*\varphi)]\wedge(dd^c u)^{n-1}
	\end{align*}
	and therefore, we have
	\begin{align*}
		&\left|\left\langle \Lc^{\Kmc_\theta^n}(S), \varphi\right\rangle-\left\langle S, \varphi\right\rangle\right|=\left|\int_{{D^2}} (u_\theta^\Kmc -u)\pi_1^*S \wedge dd^c[(\pi_1^*\chi_{D'})(\pi_2^*\varphi)]\wedge(dd^c u)^{n-1}\right|\\
		&\quad\quad\quad\quad\quad\le C_1\left\|\pi_1^*(\chi_{D'})\pi_2^*(\varphi)\right\|_{C^2}\int_{{D^2}} (u_\theta^\Kmc -u)\pi_1^*S \wedge (\pi_1^*\chi_{D''})(\pi_2^*\chi_{D''})\omega_{x,y}^{n-s}\wedge(dd^c u)^{n-1},
	\end{align*}
	where $C_1>0$ is a constant independent of $\theta$ and $\chi_{D''}:D\to[0,1]$ is a smooth function with compact support such that $\chi_{D''}\equiv 1$ on $\overline{D''}$. 
	The singularity of $u\wedge(dd^c u)^{n-1}$ is of the form $\left[\log \dist(\cdot, \Delta)\right]\left(\dist(\cdot, \Delta)^{2-2n}\right)$ and the support of $(u_\theta^\Kmc -u)$ is a tubular neighborhood of $\Delta$ whose radius is proportional to $|\theta|>0$. So, the form $(\pi_1)_*[(u_\theta^\Kmc -u)(\pi_1^*\chi_{D''})(\pi_2^*\chi_{D''})\omega_{x,y}^{n-s}\wedge(dd^c u)^{n-1}]$ is a $C^1$-form whose $C^1$-norm is bounded by $|\theta|$ up to a multiplicative constant independent of $\theta$. The support of the form $(\pi_1)_*[(u_\theta^\Kmc -u)(\pi_1^*\chi_{D''})(\pi_2^*\chi_{D''})\omega_{x,y}^{n-s}\wedge(dd^c u)^{n-1}]$ sits inside $\supp\,\chi_{D''}$. Hence, for some constant $C_2>0$ independent of $\theta$, we have
	\begin{align*}
		\left|\left\langle \Lc^{\Kmc_\theta^n}(S), \varphi\right\rangle-\left\langle S, \varphi\right\rangle\right|\le C_2|\theta|\left\|(\pi_1^*\chi_\varphi)(\pi_2^* \varphi)\right\|_{C^2}\|S\|_{\supp\,\chi_{D''}}.
	\end{align*}
	
	For general $S\in\Cc_s(D)$, we use the regularization $S_\epsilon$ of $S$ obtained as above. When $0<\epsilon\ll 1$, we have $\|S_\epsilon\|_{\supp\,\chi_{D''}}\le C_3\|S\|_{D'''}$ for some relatively compact open subset $D'''$ in $D$ containing $\supp\,\chi_{D''}$ and for some constant $C_3>0$ independent of $\epsilon$. So, the convergence of $\Lc^{\Kmc_\theta^n}(S_\epsilon) \to S_\epsilon$ in $D'$ is uniform with respect to $\epsilon$ and therefore, Proposition \ref{prop:continuity_transform} implies that the convergence in $D'$ is true for $S\in\Cc_s(D)$.
	\medskip
	
	We prove the closedness of $\Lc^{\Kmc_\theta^n}(S)$. Assume that $\varphi=d\psi$ for some smooth form $\psi$ with compact support in $D'$. By Proposition \ref{prop:Q_representations}, we have
	\begin{align*}
		\left\langle \Lc^{\Kmc_\theta^n}(S), d\psi\right\rangle=\int_{{D^2}}\pi_1^*S \wedge \pi_2^*(d\psi)\wedge \Kmc_\theta^n=\int_{{D^2}}\pi_1^*S \wedge d\left((\pi_2^*\psi)\wedge \Kmc_\theta^n\right).
	\end{align*}
	Notice that from our choice of $\theta$ in the beginning of the proof, the support of $(\pi_2^*\psi)\wedge\Kmc_\theta^n$ is compact in ${D^2}$. So, the Stokes theorem concludes that $\Lc^{\Kmc_\theta^n}(S)$ is closed when restricted to $D'$.
\end{proof}
\begin{remark}
	If $S$ is just a current in $\Cc_s(D)$, it is not clear whether $\Lc^{\Kmc_\theta^n}(S)$ is closed in $D$. However, Proposition \ref{prop:king_general} implies that if $S$ is a current in $\Cc_s(\overline{D})$, then for all $\theta\in\C^*$ such that $S$ is defined in $\left(\overline{D}\right)_\theta$, $\Lc^{\Kmc_\theta^n}(S)$ is a smooth current in $\Cc_s(D)$.
\end{remark}

Below is a $C^\alpha$-norm estimate of $\Lc^{\Kmc_\theta^n}(S)$.
\begin{proposition}\label{prop:C_alpha_King}
	Let $S\in\Cc_s(D)$. Let $K$ be a compact subset of $D$ and $\alpha$ a non-negative integer. Then, there exists a constant $c_{\alpha, K}>0$ such that for all $\theta\in \C^*$ with $K_\theta\subset D$, we have
	\begin{align*}
		\|\Lc^{\Kmc_\theta^n}(S)\|_{C^\alpha, K}\le c_{\alpha, K} |\theta|^{-2n-\alpha}\|S\|_{K_\theta}.
	\end{align*}
\end{proposition}

\begin{proof}
	We have $\supp\,\Kmc_\theta^n\subset \{(x, y)\in {D^2}: e^{-1}|\theta|\le |x-y|\le e|\theta|\}=:W_\theta$. Over the region $W_\theta$, $dd^c u$ is smooth. We have 
	$\|(dd^c u)^{n-1}\|_{C^\alpha, W_\theta\cap \pi_2^{-1}(K)}\le c'_{\alpha, K} |\theta|^{2-2n-\alpha}$ and $\|dd^cu_\theta^\Kmc \|_{C^\alpha, W_\theta\cap \pi_2^{-1}(K)}\le c'_{\alpha, K}|\theta|^{-2-\alpha}$ for some constant $c'_{\alpha, K}>0$ independent of $\theta$. Since $\supp\,\Lc^{\Kmc_\theta^n}(S)\subset K_\theta$, we have
	\begin{align*}
		\|\Lc^{\Kmc_\theta^n}(S)\|_{C^\alpha, K}\le c_{\alpha, K} |\theta|^{-2n-\alpha}\|S\|_{K_\theta}
	\end{align*}
	for some constant $c_{\alpha, K}>0$ independent of $\theta$.
\end{proof}

The regularizing semi-regular transform $\Lc^{\Kmc_\theta^n}(\cdot)$ is compatible with the product $(S_1\wedge \cdots\wedge S_k)_K$.
\begin{theorem}\label{thm:conv_many}
	Let $S_i\in\Cc_{s_i}(D)$ for $i=1, \ldots, k$ satisfy Condition $(\mathrm{I})$, where $1\le s:=s_1+\cdots+ s_k\le n$. Then, for each $j=1, \ldots, k$, we have
	\begin{align*}
		\Lc^{\Kmc_\theta^n}(S_1)\wedge \cdots \wedge \Lc^{\Kmc_\theta^n}(S_{j-1})\wedge S_j\wedge \Lc^{\Kmc_\theta^n}(S_{j+1})\wedge\cdots\wedge \Lc^{\Kmc_\theta^n}(S_k)\to \left(S_1\wedge \cdots\wedge S_k\right)_K
	\end{align*}
	in the sense of currents, as $\theta\to 0$.
\end{theorem}
We prove the case of $k=3$ and $j=1$ as the general case can be treated in the same way.
\begin{proof}[Proof of the case of $k=3$ and $j=1$]
	In this proof, we write $D^3=D_1\times D_2\times D_3$, where each $D_i$ is a copy of $D$; we also use subscripts under pairings such as $\langle\cdot\rangle_{D^3}$ and $\langle\cdot\rangle_E$ to denote spaces. We use the notations in Section \ref{sec:intersection} with $k=3$. We denote by $(x, w_1, w_2)$ the point $(x_1, x_2, x_3)\in D^3$, where $x=x_3$, $w_1=x_1-x_3$ and $w_2=x_2-x_3$. We denote by $\Omega_n^1$ and $\Omega_n^2$ the forms $\Omega_n$ as in Lemma \ref{lem:useful_form} with $x''$ replaced by $w_1$ and $w_2$, respectively. 
	\medskip
	
	From the assumption, the current $\pi_1^*S_1\wedge \pi_2^*S_2\wedge \pi_3^*S_3$ satisfies Condition $(\mathrm{K}-\max)$ along $\Delta$. Let $\varphi$ be a smooth test $(n-s, n-s)$-form on $D$. With respect to our coordinates, $\varphi$ can be understood as a form on $\Delta$ and so, we may say $\pi_{\Delta}^*\varphi=\pi_3^*\varphi$. We have $A_\lambda(x, w_1, w_2)=(x, \lambda w_1, \lambda w_2)$. Then, as in the proof of Proposition \ref{prop:main_true_general}, from Theorem \ref{thm:main_intersection_1} together with the definition of the shadow in Definition \ref{def:tangent_current}, we get
	\begin{align*}
		&\left\langle \left(S_1\wedge S_2\wedge S_3\right)_K, \varphi \right\rangle_D=\lim_{\theta\to 0}\left\langle(\pi_{\Delta})_*\left((A_{\theta^{-1}})_*\left(\mathbf{1}_{D^3}\pi_1^*S_1\wedge \pi_2^*S_2\wedge \pi_3^*S_3\right)\wedge \pi_F^*\Omega_n^1\wedge \pi_F^*\Omega_n^2\right), \varphi\right\rangle_\Delta\\
		&\quad\quad\quad=\lim_{\theta\to 0}\left\langle (A_{\theta^{-1}})_*\left(\mathbf{1}_{D^3}\pi_1^*S_1\wedge \pi_2^*S_2\wedge \pi_3^*S_3\right)\wedge \pi_F^*\Omega_n^1\wedge \pi_F^*\Omega_n^2, \pi_\Delta^*\varphi\right\rangle_E\\
		&\quad\quad\quad=\lim_{\theta\to 0}\left\langle \pi_1^*S_1\wedge \pi_2^*S_2\wedge \pi_3^*S_3\wedge (A_{\theta^{-1}})^*\left(\pi_F^*\Omega_n^1\wedge \pi_F^*\Omega_n^2\right), \pi_3^*\varphi\right\rangle_{D^3}\\
		&\quad\quad\quad=\lim_{\theta\to 0}\left\langle \pi_1^*S_1\wedge \pi_2^*S_2\wedge \pi_3^*S_3\wedge \Kmc_{e^{-M}\theta}^{n, 1}\wedge \Kmc_{e^{-M}\theta}^{n, 2}, \pi_3^*\varphi\right\rangle_{D^3},
	\end{align*}
	where $M>0$ is the constant as in Lemma \ref{lem:useful_form} and  $\Kmc_{e^{-M}\theta}^{n, i}$ denotes the form $\Kmc_{e^{-M}\theta}^n$ with the variable $w_i$ for $i=1, 2$. 
	The current $\pi_1^*S_1(x_1)\wedge \pi_2^*S_2(x_2)\wedge \Kmc_{e^{-M}\theta}^{n, 1}(x_1, x_3)\wedge \Kmc_{e^{-M}\theta}^{n, 2}(x_2, x_3)$ can be considered as a smooth form in $x_3$. 
	Hence, the above limit can be written as
	\begin{align*}
		\left\langle \left(S_1\wedge S_2\wedge S_3\right)_K, \varphi\right\rangle_D
		=\lim_{\theta\to 0}\left\langle (\pi_3)_*\left(\pi_1^*S_1\wedge \pi_2^*S_2\wedge \Kmc_{e^{-M}\theta}^{n, 1}\wedge \Kmc_{e^{-M}\theta}^{n, 2}\right)\wedge S_3, \varphi\right\rangle_D.
	\end{align*}
	We observe that for each $x_3\in D_3$, $(\pi_3)_*$ is the integration over the fiber $D_1\times D_2$. So, for each $x_3\in D_3$, since $\pi_1^*S_1\wedge \Kmc_{e^{-M}\theta}^{n, 1}$ is a form with measures in $x_1$ as coefficients and $\pi_2^*S_2\wedge \Kmc_{e^{-M}\theta}^{n, 2}$ is a form with measures in $x_2$ as coefficients, the Fubini theorem or the notion of double currents implies that for each $x_3\in D_3$, we can split the form $\pi_1^*S_1\wedge \pi_2^*S_2\wedge \Kmc_{e^{-M}\theta}^{n, 1}\wedge \Kmc_{e^{-M}\theta}^{n, 2}$ into the product of a form in $x_1$ and another form in $x_2$. So, we get
	\begin{align*}
		\left\langle \left(S_1\wedge S_2\wedge S_3\right)_K, \varphi\right\rangle_D
		&=\lim_{\theta\to 0}\left\langle (\pi_3)_*\left(\pi_1^*S_1\wedge \Kmc_{e^{-M}\theta}^{n, 1}\right)\wedge (\pi_3)_* \left( \pi_2^*S_2\wedge \Kmc_{e^{-M}\theta}^{n, 2}\right)\wedge S_3, \varphi\right\rangle_D\\
		&=\lim_{\theta\to 0}\left\langle \Lc^{\Kmc_{e^{-M}\theta}^n}(S_1)\wedge \Lc^{\Kmc_{e^{-M}\theta}^n}(S_2)\wedge S_3, \varphi\right\rangle_D.
	\end{align*}
	In the above, we abuse notation in the sense that the same notation $\pi_3$ is used for projections $\pi_3:D^3\to D_3$, $\pi_3:D_2\times D_3\to D_3$ and $\pi_3:D_3\times D_1\to D_3$. Or one can say that we identify $D_2\times D_3$ with $\{{x}'_1\}\times D_2\times D_3$ for some ${x}'_1\in D_1$ and $D_3\times D_1$ with $D_1\times \{{x}'_2\}\times D_3$ for some ${x}'_2\in D_2$. In the same way, we used $\pi_1$ and $\pi_2$.
\end{proof}

\begin{proposition}
	Suppose that $S_i\in\Cc_{s_i}(D)$ for $i=1, \ldots, k$ satisfy Condition $(\mathrm{I})$, where $1\le s:=s_1+ \cdots+ s_k\le n$. Then, the product $\left(S_1\wedge \cdots \wedge S_k\right)_K$ is symmetric.
\end{proposition}

\begin{remark}
	In the case of $S\in\Cc_s(D)$ and a complex submanifold $Z\subset D$ of codimension $m$ satisfying \eqref{eq:King_residue} for some $m\in\{1, \ldots, n-s\}$, we have two notions of wedge product: $S\wedge_C[Z]$ and $\left(S\wedge[Z]\right)_K$ provided that both of them are well-defined. Then, one can relate $S\wedge_C[Z]$ to the tangent current of $S$ along $Z$ on the normal bundle of $Z$ in $D$ while $\left(S\wedge[Z]\right)_K$ the tangent current of $\pi_1^*S\wedge\pi_2^*[Z]$ along the diagonal submanifold $\Delta\subset D^2$ on the normal bundle of $\Delta$ in $D^2$. 
	By \cite[Lemma 2.3]{Vu}, we see that they are identical. See also \cite[Lemma 5.4]{DS18}.
\end{remark}

\subsection{Slicing theory and currents defined by analytic varieties}\label{subsec:slicing}

In this subsection, we consider the standard regularization by convolution and slicing theory. The basic idea is to exploit Remark \ref{rmk:indep_shadow} 
together with \eqref{eq:key_idea}. 
We use the notations in Section \ref{sec:intersection} with $k=2$. For slicing theory, the reader is referred to \cite{Federer}.

\begin{theorem}\label{thm:standard_regularization} 
	Let $S_i\in\Cc_{s_i}(D)$ for $i=1, \ldots, k$ satisfy Condition $(\mathrm{I})$, where $1\le s:=s_1+\cdots+ s_k\le n$. Let $\left((S_i)_\epsilon\right)_{0<\epsilon\ll 1}$ denote the standard regularization of $S_i$ by convolution for $i=1, \ldots, k$. Then, for any $j\in\{1, \ldots, k\}$, we have
	\begin{align*}
		(S_1)_\epsilon\wedge \cdots\wedge(S_{j-1})_\epsilon\wedge S_j\wedge (S_{j+1})_\epsilon\wedge\cdots \wedge(S_k)_\epsilon \to (S_1\wedge \cdots\wedge  S_k)_K
	\end{align*}
	as $\epsilon\to 0$ in the sense of currents.
\end{theorem}
We consider the case of $k=2$. The general case is obtained as in the proof of Theorem \ref{thm:conv_many}.
\begin{proof}[Proof of the case of $k=2$]
We use the notations as in Proposition \ref{prop:continuity_transform} for the standard regularization by convolution. Let $\varphi$ be a smooth test $(n-s, n-s)$-form on $D$, where $s:=s_1+s_2$. We may assume that $\varphi=f\Theta$, where $f$ is a positive smooth function with compact support in $D$ and $\Theta$ is a smooth positive closed $(n-s, n-s)$-form on $D$. Indeed, any smooth test $(n-s, n-s)$-form can be written as a linear combination of such forms. 
\medskip

Theorem \ref{thm:main_intersection_1}, Proposition \ref{prop:T_to_measure} and Condition $(\mathrm{I})$ imply that the current $\pi_1^*S_1\wedge \pi_2^*(S_2\wedge\Theta)$ on ${D^2}$ admits a unique tangent current $(S_1\wedge S_2\wedge\Theta)_\infty$ along $\Delta$ and its $h$-dimension is minimal. We choose $\Omega=gd\mu=\frac{1}{n!}g(w)(dd^c|w|^2)^n$, where $g$ is as in Proposition \ref{prop:continuity_transform}. Observe that it has compact support in $\C^n$ and therefore, it can trivially be extended to $\P^n$ and it belongs to the same cohomology class as $[w=0]$ does. Notice that based on the coordinates $(x_2, x_1-x_2)=(x, w)\in E$, we have $\pi_\Delta^*(\left(\pi_2^*f\right)|_\Delta)=\pi_2^*f$, and we have $A_\lambda(x, w)=\left(x, \lambda w\right)$. Then, Proposition \ref{prop:T_to_measure} gives
\begin{align*}
	&\notag\left\langle \left(S_1\wedge S_2\right)_K, f\Theta\right\rangle=\left\langle \left(\pi_1^*S_1\wedge \pi_2^*S_2\right)_\infty^h, f\Theta\right\rangle=\left\langle \left(\pi_1^*S_1\wedge\pi_2^*(S_2\wedge\Theta)\right)_\infty^h, f\right\rangle\\
	\notag&=\left\langle (\pi_\Delta)_*\left((\pi_1^*S_1\wedge\pi_2^*(S_2\wedge\Theta))_\infty\wedge \pi_F^*\Omega\right), (\pi_2^*f)|_\Delta\right\rangle\\
	\notag&=\int_{{E}}\left(\pi_2^*f\right)(\pi_1^*S_1\wedge\pi_2^*(S_2\wedge\Theta))_\infty\wedge \left(\frac{1}{n!}g(w)(dd^c|w|^2)^n\right)\\
	\notag&=\lim_{\epsilon\to 0}\int_{{E}}\left(\pi_2^*f\right)\left(A_{\epsilon^{-1}}\right)_*(\mathbf{1}_{D^2}\pi_1^*S_1\wedge\pi_2^*(S_2\wedge \Theta))\wedge \left(\frac{1}{n!}g(w)(dd^c|w|^2)^n\right)\\
	\notag&=\lim_{\epsilon\to 0}\int_{{D^2}}\pi_1^*S_1\wedge\pi_2^*(S_2\wedge f\Theta)\wedge \left(A_{\epsilon^{-1}}\right)^*\left(\frac{1}{n!}g(w)(dd^c|w|^2)^n\right)\\
	&=\lim_{\epsilon\to 0}\int_{{D^2}}\pi_1^*S_1\wedge\pi_2^*(S_2\wedge f\Theta)\wedge \left(\frac{1}{n!\epsilon^{2n}}g\left(\frac{w}{\epsilon}\right)(dd^c|w|^2)^n\right)\\
	&=\lim_{\epsilon\to 0}\int_{y\in D}\left( \int_{w\in \C^n}S_1(y+w)\wedge \left(\frac{1}{\epsilon^{2n}}g\left(\frac{w}{\epsilon}\right)d\mu(w)\right) \right)\wedge S_2\wedge f\Theta=\lim_{\epsilon\to 0}\left\langle \left(S_1\right)_\epsilon\wedge S_2, f\Theta\right\rangle.	
\end{align*}
\end{proof}

Next, we examine slicing theory. Slicing is a generalization of the restriction of forms to level sets of a holomorphic submersion. Let $X$ and $F$ be two complex manifolds of dimension $N$ and $n$, respectively. Let $\pi_F:X\to F$ be a holomorphic submersion and $\Rc$ a $(N-m, N-m)$-current on $X$ with $m\ge n$. Assume that $\Rc$, $\partial \Rc$ and $\bar\partial \Rc$ are of order $0$. One can define the slice $\langle \Rc, \pi_F, \theta\rangle$ as below for almost every $\theta\in F$. This is a current of dimension $(m-n, m-n)$ on $\pi_F^{-1}(\theta)$. One may consider it as a current on $X$. The slicing commutes with $\partial$ and $\bar \partial$. In particular, if $\Rc$ is closed, then $\langle \Rc, \pi_F, \theta\rangle$ is also closed.
\medskip

Let $z$ denote the coordinates for $F$ and $\mu$ the standard Lebesgue measure. Let $\psi(z)$ be a positive smooth function with compact support such that $\int \psi d\mu =1$. Define $\psi_\varepsilon(z):=\epsilon^{-2n}\psi(\epsilon^{-1}z)$ and $\psi_{\theta, \epsilon}(z):=\psi_\epsilon(z-\theta)$ (the measure $\psi_{\theta, \epsilon}\mu$ approximates the Dirac mass at $\theta$). Then, for every smooth test form $\Psi$ of the right bidegree with compact support in $X$ one has
\begin{align*}
	\langle\Rc, \pi_F, \theta\rangle(\Psi)=\lim_{\epsilon\to 0}\langle \Rc\wedge\pi_F^*(\psi_{\theta, \epsilon}), \Psi\rangle
\end{align*}
when $\langle \Rc, \pi_F, \theta\rangle$ exists. This property holds for all choices of the functions $\psi$ and $\Psi$ such that $\pi_F$ is proper on $\supp\, \Psi \cap \supp\, \Rc$. Conversely, when the previous limit exists and is independent of $\psi$, it defines $\langle \Rc, \pi_F, \theta\rangle$ and one says that $\langle \Rc, \pi_F, \theta\rangle$ is well-defined. 
\medskip

We consider $X=D^2$, $N=2n$, $\pi_F:D^2\subset \overline{E}\to \P^n$, $\Rc=\pi_1^*S_1\wedge\pi_2^*S_2$, $\psi=g$ and $\theta=0$. Notice that Theorem \ref{thm:standard_regularization} is true for any choice of $g$. Together with Lemma \ref{lem:restriction_lem}, the arguments used in Proposition \ref{prop:convergence_V} and Theorem \ref{thm:standard_regularization} imply the following:
\begin{theorem}
	The slice $\left\langle \pi_1^*S_1\wedge \pi_2^*S_2, \pi_F, 0\right\rangle$ exists.
\end{theorem}
Condition $(\mathrm{I})$ is true for all the proper intersections of holomorphic cycles as will be shown in Subsection \ref{subsec:analytic_subset}, which shows that the Dinh-Sibony product, together with Condition $(\mathrm{I})$, extends King's work on the intersection of holomorphic cycles in \cite{King-1} as it uses slicing theory (\cite[Definition 4.1.3]{King-1}).

\section{Classical Examples}\label{sec:examples}
\subsection{Classical Lelong number}\label{subsec:Lelong} The Lelong number corresponds to the case where $V$ is a single point (\cite{DS18}, \cite{Vu}). We check that Theorem \ref{thm:K'} is applicable in this case. Let $U$ be an open subset of $\C^N$ containing $0$. We may assume that $V=\{0\}$ and $u=\log|x|$, where $x$ is the coordinates of $U$. The unbounded locus $L(u)$ of $u$ is $V$, which is of dimension $0$. By \cite[Theorem III.4.5]{Demailly} and \cite[Proposition III.4.9]{Demailly}, the following integral is always finite
\begin{align*}
	\int_U u(dd^c u)^i\wedge T\wedge \omega^{N-p-i}>-\infty
\end{align*}
for $0\le i\le N-p-1$. Hence, $T$ satisfies Condition $(\mathrm{K}^*-\max)$. Theorem \ref{thm:K'} implies that there exists a unique tangent current with $h$-dimension $0$. Harvey's representation of the Lelong number proves that the shadow of the unique tangent current equals the Lelong number.

\subsection{Intersections of analytic subsets}\label{subsec:analytic_subset}We apply Theorem \ref{thm:main_intersection_1} to the proper intersections of analytic subsets on complex manifolds. Here, complex manifolds are not limited to compact complex manifolds or K\"ahler manifolds, but rather general complex manifolds.

\begin{proposition}\label{prop:analytic_subset}
	Let $X$ be a complex manifold of dimension $n$. Let $H_1, \ldots, H_k$ be irreducible analytic subsets of pure codimension $h_1, \ldots, h_k$ defined on $X$, respectively, where $1\le h:=h_1+\cdots+ h_k\le n$. Suppose that the intersection $H_1\cap \cdots \cap H_k$ is an analytic subset of pure codimension $h$. Then, the currents $[H_1], \ldots, [H_k]$ of integration satisfy Condition $(\mathrm{I})$ and we have
	\begin{align*}
		\left([H_1]\wedge \cdots\wedge [H_k]\right)_K=\sum_{\alpha\in A} m_\alpha[C_\alpha],
	\end{align*}
	where $C_\alpha$ for $\alpha\in A$ are the irreducible components of the intersection $H_1\cap \cdots \cap H_k$, and the integer $m_\alpha$ is the multiplicity of intersection of $H_1, \ldots, H_k$ along the component $C_\alpha$.
\end{proposition}

\begin{proof}
Based on Remark \ref{rmk:independence_of_admissible_map}, it suffices to consider domains. We use the notations in Section \ref{sec:intersection}. We consider a domain $D\subset\C^n$ as in Section \ref{sec:intersection} and take $T=\pi_1^*[H_1]\wedge\ldots\wedge\pi_k^*[H_k]$. Let $\widehat{H}:=\supp\, T=\pi_1^{-1}(H_1)\cap\cdots\cap\pi_k^{-1}(H_k)$. For compact $K\subset D^k$, we consider the integrals:
\begin{align*}
	\int_{K} \left(\pi_1^*[H_1]\wedge\ldots\wedge\pi_k^*[H_k]\right)\wedge \left(u(dd^c u)^{i}\right)\wedge \left(\pi_{\Delta}^*\omega_{\Delta}\right)^{kn-h-i}\quad\textrm{ for }\,i=0, \ldots, (k-1)n-1.
\end{align*}
The unbounded locus $L(u)$ of $u$ is $\Delta$ and therefore, $L(u)\cap \widehat{H}$ is essentially $H_1\cap \cdots \cap H_k$, which is of complex dimension $n-h$. Since $2(n-h)<2(kn-h)-2(k-1)n+1\le 2(kn-h)-2i+1$, we have $\Hc_{2(kn-h)-2i+1}(H)=0$ for $i=1, \ldots, (k-1)n$. By \cite[Theorem III.4.5]{Demailly} and \cite[Proposition III.4.9]{Demailly}, the above integrals are all finite for every compact $K\subset D^k$ and $\pi_1^*[H_1]\wedge\ldots\wedge\pi_k^*[H_k]$ satisfies Condition $(\mathrm{K}-\max)$ along $\Delta$. Hence, Theorem \ref{thm:main_intersection_1} says that there exists a unique tangent current of $\pi_1^*[H_1]\wedge\ldots\wedge\pi_k^*[H_k]$ along $\Delta$ and its $h$-dimension is minimal. We denote by $[H]_\infty$ the unique tangent current. Subsection \ref{subsec:slicing} proves that the shadow $[H]_\infty^h$ of $[H]_\infty$ coincides with the intersection in King's work in \cite[Definition 4.1.3]{King-1}, which implies that it coincides with $\sum_{\alpha\in A} m_\alpha[C_\alpha]$.
\end{proof}

\begin{remark}\label{rmk:int_for_ex}
We can also use the integral formula in Theorem \ref{thm:main_intersection_1} to show $[H]_\infty^h=\sum_{\alpha\in A} m_\alpha[C_\alpha]$ if none of the irreducible components $C_\alpha$ for $\alpha\in A$ belongs to the singular part of $H_i$ for all $i=1, \ldots, k$. In this case, we have $m_\alpha=1$ for $\alpha\in A$. We compute the shadow $[H]_\infty^h$ of $[H]_\infty$. Let $\varphi$ be a smooth test $(n-h, n-h)$-form on $\Delta$ (or equivalently on $D_k$) and $\Phi$ a smooth form with compact support such that $\Phi=\pi_{\Delta}^*\varphi$ in a neighborhood of $\Delta$.
\medskip

As in Proposition \ref{prop:T_wedge_V}, the support of $[H]_\infty^h$ sits inside an analytic subset of dimension $n-h$ in $D^k$ and $[H]_\infty^h$ is positive and closed. By the support theorem of Siu, $[H]_\infty^h$ is a linear combination of currents $[C_\alpha]$ of integration for $\alpha\in A$. By considering $\varphi$ supported in a smaller open subset $D'\subset D$ if necessary, we may assume that the support of $[H]_\infty^h$ is irreducible in $D'$. Denote by $H$ the irreducible analytic subset. Then, $[H]_\infty^h$ should be a constant multiple of $[H]$. We can determine this constant by looking at the regular part of $H$. So, we may further shrink the support of $\varphi$ and assume that each $H_i\subset D$ is a complex submanifold of codimension $h_i$ and that their pairwise intersection is transversal. From \cite[Proposition III.4.12]{Demailly}, we see that $\pi_1^*[H_1]\wedge\ldots\wedge\pi_k^*[H_k]=\left[\widehat{H}\right]$. Then, the desired shadow is
\begin{align*}
	\left\langle \left[H\right]_\infty^h, \varphi\right\rangle = \int_{D^k} \left[\widehat{H}\right]\wedge \left(u(dd^c u)^{(k-1)n-1}\right) \wedge dd^c \Phi,
\end{align*}
where $\Phi$ is a smooth test form on $D^k$ as above. The set $L(u)\cap \widehat{H}$ is of dimension $n-h$.
As done previously, by \cite[Theorem III.4.5]{Demailly} and \cite[Proposition III.4.9]{Demailly}, we see that the currents $u(dd^c u)^{(k-1)n-1}\wedge \left[\widehat{H}\right]$ on $D^k$ and $u|_{\widehat{H}}(dd^c u|_{\widehat{H}})^{(k-1)n-1}$ on $\widehat{H}$ do not charge any mass on $L(u)\cap \widehat{H}$. Hence, by the same argument as above, we have
\begin{align*}
	&\int_{D^k} u(dd^c u)^{(k-1)n-1}\wedge \left[\widehat{H}\right]\wedge dd^c \Phi=\int_{D^k\setminus \Delta} u(dd^c u)^{(k-1)n-1}\wedge \left[\widehat{H}\right]\wedge dd^c \Phi\\
	&=\int_{\widehat{H}\setminus \Delta} u|_{\widehat{H}}\left(dd^c u|_{\widehat{H}}\right)^{(k-1)n-1}\wedge dd^c \Phi|_{\widehat{H}}=\int_{\widehat{H}} u|_{\widehat{H}}\left(dd^c u|_{\widehat{H}}\right)^{(k-1)n-1}\wedge dd^c \Phi|_{\widehat{H}}\\
	&\quad\quad\quad=\int_{\widehat{H}} \left(dd^c u|_{\widehat{H}}\right)^{(k-1)n}\wedge \Phi|_{\widehat{H}}=\int_{\widehat{H}} [\Delta\cap \widehat{H}]\wedge \Phi|_{\widehat{H}}=\int_{\widehat{H}} [\Delta\cap \widehat{H}]\wedge \left(\pi_{\Delta}^*\varphi\right)\Big|_{\widehat{H}}=\left\langle [H], \varphi\right\rangle.
\end{align*}
The third to last equality comes from King's residue formula. The last equality comes from the biholomorphism of $\Delta$ and $D$. Indeed, in the coordinates as in Section \ref{sec:intersection}, we have the desired equality. Hence, the constant equals $1$ and therefore, the shadow is $[H]$.
\end{remark}

\begin{remark}
	Slightly more generally, when $T$ is given by a current of integration on an analytic subset, which is in a generic position with respect to $V$, we can apply the same argument as above. Then, we see that $T$ has a unique tangent current along $V$ and its $h$-dimension is minimal. Moreover, the unique tangent current is simply the inverse image of the intersection of $T$ and $V$ under the projection of $E$ onto $V$.
\end{remark}

\subsection{Intersections with a positive closed $(1, 1)$-current}\label{subsec:bidegree11} 
Let $D$ be a simply connected domain in $\C^n$ with smooth boundary. We use $x$ for the coordinates of $D$. We use the notations in Section \ref{sec:intersection} with $k=2$. Let $S\in\Cc_1(D)$ and $R\in\Cc_r(D)$, where $1\le r\le n-1$. Let $f$ be a plurisubharmonic function such that $S=dd^cf$. Suppose that $f$ is locally integrable with respect to the trace measure of $R$. That is, the measure $|f|\left(R\wedge \omega_{\rm euc}^{n-r}\right)$ is locally finite. We prove that in this classical case, $S$ and $R$ satisfy Condition $(\mathrm{I})$.

\begin{proposition}\label{prop:(1, 1)}
	Let $S$ and $R$ be as above. Then, $S$ and $R$ satisfy Condition $(\mathrm{I})$  and the Dinh-Sibony product is equal to $dd^c(fR)$.
\end{proposition}

\begin{lemma}\label{lem:n-1_continuity}
	For $S\in\Cc_1(D)$ and $R\in\Cc_r(D)$, $\left\langle \pi_1^*S\wedge \pi_2^*\left(R\wedge\omega_{\rm euc}^{n-r}\right)\wedge (dd^cu)^{n-1} \right\rangle_\Kmc$ is well-defined and the map $S\to \left\langle \pi_1^*S\wedge \pi_2^*\left(R\wedge\omega_{\rm euc}^{n-r}\right)\wedge (dd^cu)^{n-1} \right\rangle_\Kmc$ is continuous in the sense of currents.
\end{lemma}

\begin{proof}
	For compact $K\subset D^2$, we consider the integrals
	\begin{align*}
		\int_{K}u\left\langle \pi_1^*S\wedge \pi_2^*\left(R\wedge\omega_{\rm euc}^{2n-r-1-i}\right)\wedge (dd^cu)^i\right\rangle_\Kmc
	\end{align*}
	for $i=0, \ldots, n-2$. Since $R\wedge\omega_{\rm euc}^{2n-r-1-i}$ is defined on $D$ and its bidegree is $2n-1-i$, this integral is $0$ for $i=0, 1, \ldots, n-2$. Lemma \ref{lem:basic_conv_KT} implies that the current $\left\langle \pi_1^*S\wedge \pi_2^*(R\wedge \omega_{\rm euc}^{n-r})\wedge (dd^cu)^{n-1}\right\rangle_\Kmc$ is well-defined.
	\medskip
	
	Let $(S_k)_{k\in\N}$ be a sequence of currents in $\Cc_1(D)$ such that $S_k\to S$ as $k\to \infty$ in the sense of currents. We first prove that the family $\left(\left\langle \pi_1^*S_k\wedge \pi_2^*\left(R\wedge \omega_{\rm euc}^{n-r}\right)\wedge(dd^cu)^{n-1}\right\rangle\right)_{n\in \N}$ has locally uniformly bounded mass around $\Delta$. Let $D'$ be a relatively compact open subset of $D$. Let $\chi_1:D\to [0,1]$ and $\chi_2:D\to [0,1]$ be smooth functions with compact support such that $\chi_2\equiv 1$ on $D'$ and that $\chi_1\equiv 1$ on a neighborhood of the support of $\chi_2$.  
	We basically follow the arguments used in the proof of Lemma \ref{lem:basic_conv_KT}. 
	Then, for $\theta\in\C^*$ with $|\theta|\ll 1$, we have
	\begin{align*}
		&\int_{D^2}\left(\pi_1^*\chi_1\right)\left(\pi_2^*\chi_2\right) \left\langle \pi_1^*S_k\wedge \pi_2^*\left(R\wedge \omega_{\rm euc}^{n-r}\right)\wedge(dd^cu)^{n-1}\right\rangle_\Kmc\\
		&\quad=\lim_{\theta\to 0}\int_{D^2}u_\theta^\Kmc \left(\pi_1^*\chi_1\right)dd^c\left(\pi_2^*\chi_2\right) \wedge \pi_1^*S_k\wedge \pi_2^*\left(R\wedge \omega_{\rm euc}^{n-r}\right)\wedge (dd^cu)^{n-2}\\
		&\quad\quad\quad+ \int_{D^2}u_\theta^\Kmc \left(dd^c\left(\left(\pi_1^*\chi_1\right)\left(\pi_2^*\chi_2\right)\right)-\left(\pi_1^*\chi_1\right)dd^c\left(\pi_2^*\chi_2\right)\right)\wedge \pi_1^*S_k\wedge \pi_2^*\left(R\wedge \omega_{\rm euc}^{n-r}\right)\wedge (dd^cu)^{n-2}.
	\end{align*}
	The first integral is $0$ for bidegree reason. Since the support of any derivative of $\chi_1$ is a definite distance away from $\Delta$, the second integral converges as $\theta\to 0$ and $k\to \infty$ in this order. Hence, the sequence has uniformly bounded mass.
	\medskip
	
	We show that the sequence $\left(\left\langle \pi_1^*S_k\wedge \pi_2^*\left(R\wedge \omega_{\rm euc}^{n-r}\right)\wedge(dd^cu)^{n-1}\right\rangle_\Kmc\right)_{n\in \N}$ converges and the limit is $\left\langle \pi_1^*S\wedge \pi_2^*\left(R\wedge \omega_{\rm euc}^{n-r}\right)\wedge(dd^cu)^{n-1}\right\rangle_\Kmc$. Outside the set $\Delta$, the limit currents of the sequence $\left(\left\langle \pi_1^*S_k\wedge \pi_2^*\left(R\wedge \omega_{\rm euc}^{n-r}\right)\wedge(dd^cu)^{n-1}\right\rangle_\Kmc\right)_{n\in \N}$ are equal to $\left\langle \pi_1^*S\wedge \pi_2^*\left(R\wedge \omega_{\rm euc}^{n-r}\right)\wedge (dd^cu)^{n-1}\right\rangle_\Kmc$. So, we only need to check the limit currents on $\Delta$. As in the proof of Lemma \ref{lem:basic_conv_KT}, by Lemma \ref{lem:restriction_lem}, in order to investigate the measure restricted to $\Delta$, we only need to consider a test function of the type $\pi_\Delta^*g$, where $g$ is a smooth test function on $\Delta$.\medskip
	
	Hence, as in Lemma \ref{lem:basic_conv_KT}, we claim that for a smooth test function $g$ defined on $\Delta$, we have
	\begin{align*}
		&\lim_{k\to\infty}\int_{D^2} \left(\pi_F^*\chi_F\right)\left(\pi_\Delta^*g\right)\left\langle \pi_1^*S_k\wedge \pi_2^*\left(R\wedge \omega_{\rm euc}^{n-r}\right)\wedge (dd^cu)^{n-1}\right\rangle_\Kmc\\
		&\quad\quad\quad\quad\quad=\int_{D^2} \left(\pi_F^*\chi_F\right)\left(\pi_\Delta^*g\right) \left\langle \pi_1^*S\wedge \pi_2^*\left(R\wedge \omega_{\rm euc}^{n-r}\right)\wedge(dd^cu)^{n-1}\right\rangle_\Kmc.
	\end{align*}
	Here, $\chi_F:\C^n\to [0,1]$ is a smooth function with compact support such that $\chi_F\equiv 1$ in a neighborhood of $0$ and that $\supp\, \pi_\Delta^*g\cap \supp\, \pi_F^*\chi_F$ is a compact subset of $D^2$. We have
	\begin{align*}
		&\lim_{k\to\infty}\int_{D^2} \left(\pi_F^*\chi_F\right)\left(\pi_\Delta^*g\right)\left\langle \pi_1^*S_k\wedge \pi_2^*\left(R\wedge \omega_{\rm euc}^{n-r}\right)\wedge(dd^cu)^{n-1}\right\rangle_\Kmc\\
		&\quad\quad\quad\quad\quad=\lim_{k\to\infty}\lim_{\theta\to 0}\int_{D^2} u_\theta^\Kmc \left(\pi_\Delta^*g\right)dd^c(\pi_F^*\chi_F)\wedge \pi_1^*S_k\wedge \pi_2^*\left(R\wedge \omega_{\rm euc}^{n-r}\right)\wedge(dd^cu)^{n-2}\\
		&\quad\quad\quad\quad\quad=u\left(\pi_\Delta^*g\right)dd^c(\pi_F^*\chi_F)\wedge \pi_1^*S\wedge \pi_2^*\left(R\wedge \omega_{\rm euc}^{n-r}\right)\wedge(dd^cu)^{n-2}.
	\end{align*}
	The first equality comes from the maximality of the bidegree of $R\wedge\omega_{\rm euc}^{n-r}$ as we have set $(x, w)=(x_2, x_1-x_2)$.
	The support of $dd^c(\pi_F^*\chi_F)$ does not intersect $\Delta$. So, if we express the integral $$\int_{D^2} \left(\pi_F^*\chi_F\right)\left(\pi_\Delta^*g\right) \left\langle \pi_1^*S\wedge \pi_2^*\left(R\wedge \omega_{\rm euc}^{n-r}\right)\wedge (dd^cu)^{n-1}\right\rangle_\Kmc$$ in the above form, we see that they are same.
\end{proof}

\begin{proof}[Proof of Proposition \ref{prop:(1, 1)}]
	We prove that $S$ and $R$ satisfy Condition $(\mathrm{I})$. The work in \cite{HKV} proves that the Dinh-Sibony product of $S$ and $R$ is $dd^c(fR)$. See also \cite{KV}.
	\medskip
	
	In the above proof, we have seen that for compact $K\subset D^2$, we have
	\begin{align*}
		\int_{K}u\left\langle \pi_1^*S\wedge \pi_2^*\left(R\wedge\omega_{\rm euc}^{2n-r-1-i}\right)\wedge (dd^cu)^i\right\rangle_\Kmc=0
	\end{align*}
	for $i=0, \ldots, n-2$. So, we consider the case of $i=n-1$.
	\medskip
	
	Let $D'$ be a relatively compact open subset of $D$. Let $\chi_1:D\to [0,1]$ and $\chi_2:D\to [0,1]$ be smooth functions as above in Lemma \ref{lem:n-1_continuity}. We may assume that $f$ is negative on the support of $\chi_2$. Let $(f_j)$ be a sequence of smooth plurisubharmonic functions decreasingly converging to $f$ and negative on the support of $\chi_2$. Since $f$ is locally integrable with respect to $(R\wedge\omega_{\rm euc}^{n-r})$, we have
	\begin{align*}
		0\ge \int_{D^2} \pi_1^*(\chi_1 f_j) \pi_2^*(\chi_2 R\wedge\omega_{\rm euc}^{n-r}) \wedge [\Delta]=\int_{D}\chi_2f_j R\wedge\omega_{\rm euc}^{n-r}\ge \int_{D}\chi_2f R\wedge\omega_{\rm euc}^{n-r} >-\infty.
	\end{align*}
	The current $\pi_2^*\left(\chi_2 R\wedge\omega_{\rm euc}^{n-r}\right)$ is closed since it is of maximal bidegree with respect to $y$. Hence, we have
	\begin{align}
		\notag&\int_{D^2} \pi_1^*(\chi_1 f_j) \pi_2^*(\chi_2 R\wedge\omega_{\rm euc}^{n-r}) \wedge [\Delta]=\int_{D^2} dd^c\pi_1^*(\chi_1 f_j) \wedge\pi_2^*(\chi_2 R\wedge\omega_{\rm euc}^{n-r}) \wedge u(dd^cu)^{n-1}\\
		\label{eq:last_1}&=\int_{D^2} \pi_1^*(\chi_1 dd^cf_j) \wedge\pi_2^*(\chi_2 R\wedge\omega_{\rm euc}^{n-r}) \wedge u(dd^cu)^{n-1}\\
		\label{eq:last_2}&\quad\quad\quad + \int_{D^2} \pi_1^*(dd^c(\chi_1f_j)-\chi_1dd^cf_j) \wedge\pi_2^*(\chi_2 R\wedge\omega_{\rm euc}^{n-r}) \wedge u(dd^cu)^{n-1}.
	\end{align}
	The first equality holds as $f_j$ is smooth.
	Observe that the support of any derivative of $\chi_1$ is disjoint from the support of $\chi_2$. So, over the support of $d\chi_1$, $d^c\chi_1$ or $dd^c\chi_1$, if the integral \eqref{eq:last_2} is first integrated with respect to $y$, then it is a smooth form. We consider one of the terms in the integral \eqref{eq:last_2}, which is of the form $$\int_D df_j\wedge d^c\chi_1\wedge \psi,$$ where $\psi=\int_{y\in D} \pi_2^*(\chi_2 R\wedge\omega_{\rm euc}^{n-r}) \wedge u(dd^cu)^{n-1}$ is a smooth $(n-1, n-1)$-form on $\supp\, d^c\chi_1$; the other two terms can be dealt in the same way. Since the support of $\chi_1$ is compact, we have
	\begin{align*}
		\int_D df_j\wedge d^c\chi_1\wedge \psi=\int_D f_j\wedge dd^c\chi_1\wedge\psi 
		+ \int_D f_j d^c\chi_1\wedge d\psi.
	\end{align*}
	It is bounded by $\|\chi_1\|_{C^2}\|\psi\|_{C^1, \supp\, d^c\chi_1}\|f_j\|_{L^1, \supp \chi_1}$ up to a multiplicative constant independent of $j$. So, \eqref{eq:last_2} is uniformly bounded. Hence, what we have obtained so far is that 
	\begin{align*}
		\eqref{eq:last_1}&=\int_{D^2} \pi_1^*(\chi_1 f_j) \pi_2^*(\chi_2 R\wedge\omega_{\rm euc}^{n-r}) \wedge [\Delta]-\eqref{eq:last_2}>-M_1
	\end{align*}
	for some constant $M_1>0$ independent of $j$.
	\medskip
	
	The function $S\to \int_{D^2}u\left\langle\pi_1^*\left(\chi_1S\right) \wedge\pi_2^*(\chi_2 R\wedge\omega_{\rm euc}^{n-r}) \wedge (dd^cu)^{n-1}\right\rangle_\Kmc$ is upper-semicontinuous as Lemma \ref{lem:n-1_continuity} says that $S\to \int_{D^2}u_\theta^\Kmc\left\langle\pi_1^*\left(\chi_1S\right) \wedge\pi_2^*(\chi_2 R\wedge\omega_{\rm euc}^{n-r}) \wedge (dd^cu)^{n-1}\right\rangle_\Kmc$ is continuous and $S\to \int_{D^2}u\left\langle\pi_1^*\left(\chi_1S\right) \wedge\pi_2^*(\chi_2 R\wedge\omega_{\rm euc}^{n-r}) \wedge (dd^cu)^{n-1}\right\rangle_\Kmc$ can be written as a decreasing limit of continuous functions $\left(S\to\int_{D^2}u_\theta^\Kmc\left\langle\pi_1^*\left(\chi_1S\right) \wedge\pi_2^*(\chi_2 R\wedge\omega_{\rm euc}^{n-r}) \wedge (dd^cu)^{n-1}\right\rangle_\Kmc\right)_{0<|\theta|\ll 1}$. By the upper-semicontinuity, since $\eqref{eq:last_1}\ge -M_1$ for $j\in\N$, we conclude that
	\begin{align*}
		\int_{D^2} u\left\langle\pi_1^*(\chi_1 dd^cf) \wedge\pi_2^*(\chi_2 R\wedge\omega_{\rm euc}^{n-r}) \wedge (dd^cu)^{n-1}\right\rangle_\Kmc>-M_1
	\end{align*}
	as desired. For every compact subset of $D^2$, we can always find valid $\chi_1$ and $\chi_2$ satisfying the support condition as above. Due to the negativity, $u\left\langle\pi_1^*(dd^cf) \wedge\pi_2^*(R\wedge\omega_{\rm euc}^{n-r}) \wedge (dd^cu)^{n-1}\right\rangle_\Kmc$ is locally integrable, which means $dd^cf$ and $R$ satisfy Condition $(\mathrm{I})$.
\end{proof}

\subsection{Positive closed currents with continuous superpotentials}\label{subsec:conti_sup} When a positive closed current has continuous superpotentials, the intersection defined by tangent currents coincides with that defined by superpotentials (\cite{VuMich}, \cite{DNV}). We show that on a compact K\"ahler manifold, a positive closed current with continuous superpotentials satisfies Condition $(\mathrm{I})$ with every positive closed current. For definitions and related properties of superpotentials on compact K\"ahler manifolds, we refer the reader to \cite{DS10} and \cite{DNV}. See also \cite{DS09} for the theory on complex projective spaces.
\medskip

Let $(X, \omega_X)$ be a compact K\"ahler manifold of dimension $n$. We can take local potential superfunctions as in Subsection \ref{subsec:superfunctions} with $Y=X$. We may take $J$ to be a finite set. Let $\Xf:=X_1\times X_2$ and $\pi_l:\Xf\to X_l$ the canonical projection onto the $l$-th factor for $l=1, 2$, where each $X_l$ is a copy of $X$. Let $\pi:\Xfh\to \Xf$ denote the blow-up of $\Xf$ along $\Delta$. Then, by a theorem of Blanchard, $\Xfh$ is a compact K\"ahler manifold and let $\omega_\Xfh$ denote a K\"ahler form on $\Xfh$. Let $\widehat{\Delta}$ denote the exceptional divisor. Let $\Pi_l:=\pi_l\circ\pi$ for each $l=1, 2$. Let $\alpha_{\widehat{\Delta}}$ be a smooth real closed $(1, 1)$-form on $\Xfh$ cohomologous to $[\widehat{\Delta}]$. Then, there exists a quasi-plurisubharmonic function $u_{\widehat{\Delta}}$ on $\Xfh$ such that $[\widehat{\Delta}]-\alpha_{\widehat{\Delta}}=dd^c u_{\widehat{\Delta}}$. 
\medskip

We claim that if $S\in\Cc_s(X)$ admits continuous superpotentials, each function $\Fc_{S, j}^i:\Dc_S^{i-1}\to \R\cup \{-\infty\}$ is upper-semicontinuous not just on $\Dc_{S, j}^i$ but on $\Dc_{S, j}^{i-1}$ in the following sense.
\begin{proposition}\label{prop:usc}
	Let $S\in\Cc_s(X)$ admit continuous superpotentials. Let $i\in\{1, \ldots, n\}$ and $j\in J$. Let $M>0$ be a constant. Let $R\in\Dc_{S, j}^{i-1}$ be a current such that $\Fc_{S, j}^{i-1}(R)\ge-M$. Let $\left(R_k\right)_{k\in\N}$ be a sequence of smooth currents in $\Dc_{S, j}^{i-1}$ such that $R_k\to R$ as $k\to\infty$ and that $\Fc_{S, j}^{i-1}(R_k)\ge-M$ for every $k\in\N$. Then, we have 
	\begin{align*}
		\limsup_{k\to\infty}\Fc_{S, j}^i(R_k)\le \Fc_{S, j}^i(R).
	\end{align*}
\end{proposition}

\begin{lemma}\label{lem:usc_conti}
	Let $S$, $R$ and $(R_k)_{k\in\N}$ be as above in Proposition \ref{prop:usc}. Then, we have 
	\begin{align*}
		\lim_{k\to\infty}\left\langle \pi_1^*S\wedge\pi_2^*\left(R_k\wedge\omega_{\rm euc}^{n-i}\right)\wedge(dd^cu)^{i-1}\right\rangle_\Kmc=\left\langle \pi_1^*S\wedge\pi_2^*\left(R\wedge\omega_{\rm euc}^{n-i}\right)\wedge(dd^cu)^{i-1}\right\rangle_\Kmc.
	\end{align*}
\end{lemma}
\begin{proof}
	When $i=1$, this is obvious. So, we may assume that $i\ge 2$. 
	We first claim that the sequence $\left\langle \pi_1^*S\wedge\pi_2^*\left(R_k\wedge\omega_{\rm euc}^{n-i}\right)\wedge(dd^cu)^{i-1}\right\rangle_\Kmc$ has locally uniformly bounded mass with respect to $k\in \N$.
	\medskip
	
	It suffices to consider local uniform boundedness of mass around $\Delta$ as $dd^cu$ is smooth outside $\Delta$. Let $\chi_1:U_j^{i-1}\to [0,1]$ be a smooth function with compact support and $\chi_2:U_j^{i-1}\to [0, 1]$ another smooth function with compact support such that $\chi_1\equiv 1$ on a neighborhood of $\supp\,\chi_2$. We have
	\begin{align*}
		&\int_{{U_j^{i-1}\times U_j^{i-1}}} \left(\pi_1^*\chi_1\right)\left(\pi_2^*\chi_2\right)\left\langle \pi_1^*S\wedge\pi_2^*\left(R_k\wedge\omega_{\rm euc}^{n-i}\right)\wedge (dd^cu)^{i-1}\right\rangle_\Kmc\\
		&\quad =\lim_{\theta\to 0}\int_{{U_j^{i-1}\times U_j^{i-1}}} u_\theta^\Kmc \left(\pi_1^*\chi_1\right)dd^c\left(\pi_2^*\chi_2\right) \wedge\pi_1^*S\wedge\pi_2^*\left(R_k\wedge\omega_{\rm euc}^{n-i}\right)\wedge(dd^cu)^{i-2}\\
		&\quad\,\, +\int_{{U_j^{i-1}\times U_j^{i-1}}}  u_\theta^\Kmc \left(dd^c\left(\left(\pi_1^*\chi_1\right)\left(\pi_2^*\chi_2\right)\right)-\left(\pi_1^*\chi_1\right)dd^c\left(\pi_2^*\chi_2\right)\right) \wedge\pi_1^*S\wedge\pi_2^*\left(R_k\wedge\omega_{\rm euc}^{n-i}\right)\wedge(dd^cu)^{i-2}.
	\end{align*}
	The first integral is uniformly bounded with respect to $k$ due to the assumption that $\Fc_{S, j}^{i-1}(R_k)\ge-M$. For the second integral, the region of integration is uniformly separated from $\Delta$ with respect to $\theta$, $j$. The function $u$ and the form $dd^cu$ are smooth on the region of integration. So, the second integral is uniformly bounded with respect to $k$.\medskip
	
	Now, we look into limit currents of $\left(\left\langle \pi_1^*S\wedge\pi_2^*\left(R_k\wedge\omega_{\rm euc}^{n-i}\right)\wedge (dd^cu)^{i-1}\right\rangle_\Kmc\right)_{k\in\N}$. 
	Notice that over $U_j^{i-1}\times U_j^{i-1}\setminus \Delta$, every limit current equals $\pi_1^*S\wedge\pi_2^*\left(R\wedge\omega_{\rm euc}^{n-i}\right)\wedge(dd^cu)^{i-1}$. So, we are interested in the restrictions of the limit currents to $\Delta$. By replacing $R_k$ by its subsequence, we may assume that the sequence converges. As in Lemma \ref{lem:basic_conv_KT}, due to Lemma \ref{lem:restriction_lem}, it suffices to consider a test function of the type $\pi_F^*\chi_\varepsilon\pi_\Delta^*f$, where $f:\left(U_j^{i-1}\times U_j^{i-1}\right) \cap\Delta\to \R$ is a smooth test function on $\left(U_j^{i-1}\times U_j^{i-1}\right) \cap\Delta$ and $\chi_\varepsilon:\C^n\to [0,1]$ is a smooth function with compact support in the $\varepsilon$-neighborhood of the origin such that $\chi_\varepsilon=1$ in a neighborhood of $0$. Here, $\pi_\Delta$ and $\pi_F$ are as in Section \ref{sec:intersection} with $k=2$. We may assume that $0<\varepsilon\ll 1$. We have 
	\begin{align}
		\label{eq:seq_total}&\int \left(\pi_F^*\chi_\varepsilon\right)\left(\pi_\Delta^*f\right)\left\langle \pi_1^*S\wedge\pi_2^*\left(R_k\wedge\omega_{\rm euc}^{n-i}\right)\wedge (dd^cu)^{i-1}\right\rangle_\Kmc\\
		\notag&\quad=\lim_{\theta\to 0}\int u_\theta^\Kmc  \left(\pi_F^*\chi_\varepsilon\right) dd^c\left(\pi_\Delta^*f\right)\wedge\pi_1^*S\wedge\pi_2^*\left(R_k\wedge\omega_{\rm euc}^{n-i}\right)\wedge (dd^cu)^{i-2}\\
		\label{eq:seq_2}&\quad\quad +\int u  \left(dd^c\left(\left(\pi_F^*\chi_\varepsilon\right)\left(\pi_\Delta^*f\right)\right)-\left(\pi_F^*\chi_\varepsilon\right)dd^c\left(\pi_\Delta^*f\right)\right)\wedge\pi_1^*S\wedge\pi_2^*\left(R_k\wedge\omega_{\rm euc}^{n-i}\right)\wedge (dd^cu)^{i-2}.			
	\end{align}
	In the same way, we see that
	\begin{align}
		\label{eq:limit_total}&\int \left(\pi_F^*\chi_\varepsilon\right)\left(\pi_\Delta^*f\right)\left\langle \pi_1^*S\wedge\pi_2^*\left(R\wedge\omega_{\rm euc}^{n-i}\right)\wedge (dd^cu)^{i-1}\right\rangle_\Kmc\\
		\notag&\quad=\lim_{\theta\to 0}\int u_\theta^\Kmc  \left(\pi_F^*\chi_\varepsilon\right) dd^c\left(\pi_\Delta^*f\right)\wedge\pi_1^*S\wedge\pi_2^*\left(R\wedge\omega_{\rm euc}^{n-i}\right)\wedge(dd^cu)^{i-2}\\
		\label{eq:limit_2}&\quad\quad +\int u\left(dd^c\left(\left(\pi_F^*\chi_\varepsilon\right)\left(\pi_\Delta^*f\right)\right)-\left(\pi_F^*\chi_\varepsilon\right)dd^c\left(\pi_\Delta^*f\right)\right) \wedge\pi_1^*S\wedge\pi_2^*(R\wedge\omega_{\rm euc}^{n-i})\wedge(dd^cu)^{i-2}.	
	\end{align}
	Since the region of integration is away from $\Delta$, \eqref{eq:seq_2} converges to \eqref{eq:limit_2} as $k\to\infty$.
	\medskip
	
	We claim that the limit currents of $\left(u\left(\pi_F^*\chi_\varepsilon\right) \pi_1^*S\wedge\pi_2^*\left(R_k\wedge\omega_{\rm euc}^{n-i+1}\right)\wedge (dd^cu)^{i-2}\right)_{k\in\N}$ have no mass on $\Delta$. (Their existence comes from $\Fc_{S, j}^{i-1}(R_k)\ge-M$ for $k\in\N$.)
	Since the current $R_k$ is smooth, it satisfies Condition $(\mathrm{K}-\max)$ along $\Delta$ and the current $\big\langle \pi_1^*S\wedge \pi_2^*\left(R_k\wedge\omega_{\rm euc}^{n-i+1}\right)\wedge (dd^cu)^{i-2}\big\rangle_\Kmc$ is well-defined and has no mass on $\Delta$. So, we have 
	$$u\left\langle \pi_1^*S\wedge \pi_2^*(R_k\wedge\omega_{\rm euc}^{n-i+1})\wedge(dd^cu)^{i-2}\right\rangle_\Kmc=\mathbf{1}_{\left(U_j^{i-1}\times U_j^{i-1}\right)\setminus \Delta}u\left( \pi_1^*S\wedge \pi_2^*(R_k\wedge\omega_2^{n-i+1})\right)\wedge(dd^cu)^{i-2}.$$
	By change of coordinates, which is the restriction $\pi\big|_{\Xfh\setminus \widehat{\Delta}}$ of the blow-up map, we can compare the currents $ (\pi^*u) \Pi_1^*S\wedge \Pi_2^*(R_k\wedge\omega_{\rm euc}^{n-i+1})\wedge \left(dd^c(\pi^*u)\right)^{i-2}$ and $u_{\widehat{\Delta}}\Pi_1^*S\wedge\Pi_2^*R_k\wedge \omega_{\Xfh}^{n-1}$ over the set ${\widehat{{U_j^{i-1}\times U_j^{i-1}}}\setminus \widehat{\Delta}}$. Writing out the integral in the blown-up space, the form $dd^c(\pi^*u)$ becomes a bounded smooth form. Since the singularity of $\pi^*u$ is the same as $\log \dist(\cdot, \widehat{\Delta})$, the current $\pi_*\left(u_{\widehat{\Delta}}\Pi_1^*S\wedge\Pi_2^*R_k\wedge \omega_{\Xfh}^{n-1}\right)$ dominates $u\left\langle \pi_1^*S\wedge \pi_2^*\left(R_k\wedge\omega_{\rm euc}^{n-i+1}\right)\wedge(dd^cu)^{i-2}\right\rangle_\Kmc$. The mass of $R_k$ is bounded independently of $k$ since $R$ and $R_k$ for $k\in\N$ are positive and $R_k\to R$ as $k\to\infty$. \cite[Proposition 2.7]{DNV} proves the claim. Then, together with the convergence of \eqref{eq:seq_2} to \eqref{eq:limit_2}, this proves that \eqref{eq:seq_total} converges to \eqref{eq:limit_total} since $dd^cu$ is smooth outside $\Delta$ and $\Fc_{S, j}^{i-1}(R)$ is finite.
\end{proof}

\begin{proof}[Proof of Proposition \ref{prop:usc}]
	By Lemma \ref{lem:usc_conti}, the function $R\to \Fc_{S, j, \theta}^i(R)$ 
	is continuous on the space $\{R, R_1, R_2, \ldots\}$. So, each function $\Fc_{S, j}^i$ is a decreasing limit of continuous functions $\Fc_{S, j, \theta}^i$. Hence, it is upper-semicontinuous on $\{R, R_1, R_2, \ldots\}$ and we obtain the inequality.
\end{proof}

\begin{proposition}
	Let $S\in\Cc_s(X)$ admit continuous superpotentials. Let $1\le r\le n-s$ be an integer. Then, for every $R\in\Cc_r(X)$, $S$ and $R$ satisfy Condition $(\mathrm{I})$. 
\end{proposition}

\begin{proof}
	It suffices to show that $\Dc_{S,j}^i=\Cc_{n-s+1}(X)$ for every $i\in\{1, \ldots, n\}$ and $j\in J$. So, we assume that $r=n-s+1$. Let $R\in \Cc_{n-s+1}(X)$. According to \cite{DS04}, $R$ can be written as $R=R^+-R^-$, where $R^\pm\in \Cc_{n-s+1}(X)$ are currents that can be approximated by smooth currents in $\Cc_{n-s+1}(X)$ in the same cohomology class as $R^\pm$, respectively. Let $(R_k^+)_{k\in\N}$ be a sequence of smooth currents in $\Cc_{n-s+1}(X)$ such that $\{R_k^+\}=\{R^+\}$ and $\lim_{k\to\infty}R_k^+=R^+$. By \cite[Proposition 2.5]{DNV}, there exists a constant $M_{R^+}>0$ such that $\int_{\Xfh}u_{\widehat{\Delta}}\Pi_1^*S\wedge\Pi_2^*(R_k^+\wedge \omega_X^{n-i})\wedge \omega_{\Xfh}^{i-1}>-M_{R^+}$ for $k\in\N$. When $i=1$, the same argument as in the proof of Proposition \ref{prop:usc}, the upper-semicontinuity in Proposition \ref{prop:usc} and the negativity of the current imply that $\Fc_{S, j}^1(R^+)$ is finite for every $j\in J$. From the negativity of $\Fc_{S, j}^1$ and the positivity of $R^-$, we see that $\Fc_{S, j}^1(R)$ is finite for every $j\in J$ and therefore, $\Dc_{S, j}^1=\Cc_{n-s+1}(X)$. We apply the argument inductively with respect to $i$. Then, we obtain $\Dc_{S, j}^n=\Dc_{S, j}^{n-1}=\ldots=\Dc_{S, j}^1=\Cc_{n-s+1}(X)$ for every $j\in J$ as desired.
\end{proof}

\subsection{Vertical and horizontal currents} This subsection is from a comment of T. -C. Dinh's. The theory of superpotentials in \cite{DS09} and \cite{DS10} was introduced for compact K\"ahler manifolds and the approach there is of global nature. As in Section \ref{sec:compact} and Subsection \ref{subsec:conti_sup}, our approach is applicable to compact K\"ahler manifolds.
However, due to the local nature of our approach, some more study needs to be carried out for the direct comparison of them in a general situation.\medskip

Nevertheless, the intersection of vertical and horizontal currents in \cite{DS06} contains a similar idea to the one in the theory of superpotentials in some sense and is of local nature. We study the intersection of vertical and horizontal currents, which may serve as an evidence that our approach is consistent with the theory of superpotentials. Note that any smooth test function can be written as a linear combination of smooth plurisubharmonic functions. 
\medskip

Let $M\subset \C^p$ and $N\subset \C^{n-p}$ be two bounded convex open sets. Consider the domain $D:=M\times N$ in $\C^n$. The set $\partial_vD:=\partial M\times N$ (resp., $\partial_hD:=M\times \partial N$) is called the vertical (resp., horizontal) boundary of $D$. A subset $E$ of $D$ is called vertical (resp., horizontal) if $\overline{E}$ does not intersect $\overline{\partial_vD}$ (resp., $\overline{\partial_hD}$). Let $\pi^D_1$ and $\pi^D_2$ denote the canonical projections of $D$ onto $M$ and $N$. Then, $E$ is vertical (resp., horizontal) if and only if $\overline{\pi_1^D(E)}$ is compact in $M$ (resp., $\overline{\pi_2^D(E)}$ is compact in $N$). A current on $D$ is vertical (resp., horizontal) if its support is vertical (resp., horizontal). Let $\Cc^v_p(D)$ (resp., $\Cc^h_{n-p}(D)$) denote the cone of positive closed vertical (resp., horizontal) currents of bidegree $(p, p)$ (resp., $(n-p, n-p)$) on $D$.
\begin{proposition}
	Let $S\in \Cc_p^v(D)$ and $R\in \Cc_{n-p}^h(D)$. Then, $S$ and $R$ satisfy Condition $(\mathrm{I})$  and the product $(S\wedge R)_K$ coincides with the product in \cite[Section 3]{DS06}, which we denote by $S\wedge_{VH}R$ here. More precisely, for a smooth plurisubharmonic function $\varphi$ on $D$, we have
	\begin{align}\label{eq:HV}
		\langle (S\wedge R)_K, \varphi\rangle=\langle S\wedge_{VH}R, \varphi\rangle:=\limsup_{\substack{S'\to S\\ R'\to R}}\langle S'\wedge R', \varphi\rangle,
	\end{align}
	where $S'\in\Cc_p^v(D)$ and $R'\in\Cc_{n-p}^h(D)$ are smooth with supports converging in the Hausdorff sense to those of $S$ and $R$, respectively.
\end{proposition}

\begin{proof}
	We slightly shrink $D$ so that $D$ has smooth boundary but still contains the compact subset $\supp\, S \cap \supp\, R$.
	We use the notations in Section \ref{sec:intersection} with $k=2$. Let $L(u):=\{(x, y)\in D^2: u \textrm{ is unbounded in any neighborhood of }(x, y) \textrm{ in }D^2\}$. Then, $L(u)=\Delta$ and $\supp\left(\pi_1^*S\wedge\pi_2^*R\right)\cap \Delta$ is a compact subset of $\Delta$. By Oka's inequality in \cite[Proposition 3.1]{FS}, the integrals in Definition \ref{defn:Condition(I)} are all finite, which means that $S$ and $R$ satisfy Condition $(\mathrm{I})$ (or, one might be able to use Stokes' theorem). So, the first assertion is proved.\medskip
	
	\medskip
	
	Since $\displaystyle(S\wedge R)_K= \lim_{\theta\to 0}\lim_{\theta'\to 0}\Lc^{\Kmc^n_\theta}(S)\wedge \Lc^{\Kmc^n_{\theta'}}(R)$ according to Theorem \ref{thm:conv_many}, we have
	\begin{align*}
		\langle (S\wedge R)_K, \varphi\rangle \le \limsup_{\substack{S'\to S\\ R'\to R}}\langle S'\wedge R', \varphi\rangle=\langle S\wedge_{VH}R, \varphi\rangle.
	\end{align*}
	On the other hand, according to \cite[Proposition 3.7]{DS06}, there exists a family $\left(R^{(\varepsilon)}\right)_{0<\varepsilon\ll 1}$ of smooth currents in $\Cc_{n-p}^h(D)$ such that $R^{(\varepsilon)}\to R$ and $S\wedge R^{(\varepsilon)}$ converges to $S\wedge_{VH}R$ as $\varepsilon\to 0$. Let $\chi_0:\C^n\to [0,1]$ be a smooth function with compact support such that $\chi_0\equiv 1$ in a neighborhood of the origin. Let $\chi_\Delta(x, y):=\chi_0(x-y)$ be a smooth function on $D^2$ with support in a neighborhood of $\Delta$ such that $\chi_\Delta\equiv 1$ in a neighborhood of $\Delta$. We have
	\begin{align}
		\notag&\langle S\wedge_{VH}R, \varphi\rangle=\lim_{\varepsilon\to 0}\left\langle S\wedge R^{(\varepsilon)}, \varphi\right\rangle
		=\lim_{\varepsilon\to 0}\lim_{\theta\to 0}\int\Kmc_\theta^n\wedge\pi_1^*S\wedge\pi_2^*R^{(\varepsilon)}\wedge \chi_\Delta\left(\pi_2^*\varphi\right)\\
		\label{eq:last_11}&=\lim_{\varepsilon\to 0}\lim_{\theta\to 0}\int u^\Kmc_\theta(dd^cu)^{n-1}\wedge\pi_1^*S\wedge\pi_2^*R^{(\varepsilon)}\wedge \left(\chi_\Delta\pi_2^*dd^c\varphi\right)\\
		\label{eq:last_12}&\quad\quad\quad+\lim_{\varepsilon\to 0}\lim_{\theta\to 0}\int u^\Kmc_\theta(dd^cu)^{n-1}\wedge\pi_1^*S\wedge\pi_2^*R^{(\varepsilon)}\wedge \left(dd^c\left(\chi_\Delta\pi_2^*\varphi\right)-\chi_\Delta\pi_2^*dd^c\varphi\right).
	\end{align}
	Since $\Delta\cap \supp\,\left(dd^c\left(\chi_\Delta\pi_2^*\varphi\right)-\chi_\Delta\pi_2^*dd^c\varphi\right)=\emptyset$, we have 
	\begin{align*}
		\eqref{eq:last_12}=\int u^\Kmc(dd^cu)^{n-1}\wedge\pi_1^*S\wedge\pi_2^*R\wedge \left(dd^c\left(\chi_\Delta\pi_2^*\varphi\right)-\chi_\Delta\pi_2^*dd^c\varphi\right).
	\end{align*}
	As noted previously, the integrals in Definition \ref{defn:Condition(I)} are all finite. In particular, $\int_{D^2}\chi_\Delta u\big\langle \pi_1^*S\wedge\pi_2^*(R\wedge dd^c\varphi)\wedge(dd^cu)^{n-1}\big\rangle_\Kmc$ is well-defined and finite. So, as in Proposition \ref{prop:usc_local}, the function $R\to\int_{D^2\setminus \Delta} u(dd^cu)^{n-1}\wedge\pi_1^*S\wedge\pi_2^*R\wedge \left(\chi_\Delta\pi_2^*dd^c\varphi\right)$ is upper-semicontinuous on $\{R\}\cup\left\{R^{(\epsilon)}:0<\epsilon\ll 1\right\}$. So, we have
	\begin{align*}
		\eqref{eq:last_11}&=\lim_{\varepsilon\to 0}\int_{D^2\setminus \Delta} u(dd^cu)^{n-1}\wedge\pi_1^*S\wedge\pi_2^*R^{(\varepsilon)}\wedge \left(\chi_\Delta\pi_2^*dd^c\varphi\right)\\
		&\le \int_{D^2\setminus \Delta} u(dd^cu)^{n-1}\wedge\pi_1^*S\wedge\pi_2^*R\wedge \left(\chi_\Delta\pi_2^*dd^c\varphi\right).
	\end{align*}
	Hence, as in Theorem \ref{thm:main_intersection_1}, we obtain that
		$\langle S\wedge_{VH}R, \varphi\rangle\le \langle (S\wedge R)_K, \varphi\rangle$,
	which means that the two definitions coincide.
\end{proof}

\section{Self-intersections of Analytic Subsets}\label{sec:self-intersecting}

In this section, we further study the intersection of positive closed currents beyond the proper intersection. In Sections \ref{sec:intersection} and \ref{sec:examples}, we have seen that under Condition $(\mathrm{I})$ (or equivalently, Condition $(\mathrm{K}-\max)$), the shadow of tangent currents is closely related to the intersection of positive closed currents. Now, we look into the shadow of tangent currents when the relative non-pluripolar product is used. We consider the self-intersection of analytic subsets, that is, $A\cap A=A$, where $A$ is an analytic subset. The shadow of tangent currents may be regarded as the largest piece of the intersection of positive closed currents. We use Theorem \ref{thm:shadow_general} to prove that the shadow of the corresponding tangent current in this case is exactly the current of integration on $A$ itself.
\medskip

\noindent{\bf Theorem \ref{thm:non-proper}. }{\it
Let $X$ be a compact K\"ahler manifold of dimension $n$. Let $A$ be an irreducible analytic subset of pure codimension $a$ in $X$. Let $\Delta$ be the diagonal submanifold of $X^2$. Let $\pi_i:X^2\to X$ denote the canonical projection onto the $i$-th factor for $i=1, 2$. Then, $\pi_1^*[A]\wedge\pi_2^*[A]$ has a unique tangent current along $\Delta$ and its $h$-dimension is $n-a$. Its shadow is exactly $[A]$.
}

\begin{proof}
	\cite[Proposition 2.6]{Vu} proves that $\pi_1^*[A]\wedge\pi_2^*[A]$ has a unique tangent current along $\Delta$. For the existence of tangent currents, see also \cite{DS18}. Let $T_{A, \infty}$ denote the unique tangent current.\medskip 
	
	We compute its $h$-dimension and shadow. According to \cite[Lemma 5.3]{DS18}, the $h$-dimension is at most $n-a$.	The problem of determining the $h$-dimension and the shadow is of a local nature. Hence, we consider a local situation. Let $D\subset X$ be an open subset and we identify it with a bounded simply connected domain with smooth boundary via coordinate chart. We use the notations in Section \ref{sec:intersection} with $k=2$.\medskip
	
	If its $h$-dimension is smaller than $n-a$, then $(\pi_\Delta)_*(T_{A, \infty}\wedge \Omega_{n-a})=0$. Hence, once we prove that $(\pi_\Delta)_*(T_{A, \infty}\wedge \Omega_{n-a})$ is not $0$, it means that the $h$-dimension of $T_{A, \infty}$ is $n-a$ and by Remark \ref{rmk:indep_shadow}, $(\pi_\Delta)_*(T_{A, \infty}\wedge \Omega_{n-a})$ equals the shadow of $T_{A, \infty}$.\medskip
	
	The unbounded locus $L(u)$ of $u$ is $\Delta$ and therefore $L(u)\cap \supp\, \left(\pi_1^*[A]\wedge\pi_2^*[A]\right)$ is essentially $A$, which is of complex dimension $n-a$. As in Proposition \ref{prop:analytic_subset}, by \cite[Theorem III.4.5]{Demailly} and \cite[Proposition III.4.9]{Demailly}, the integral 
	\begin{align}\label{eq:int_self_intersecting}
		\int_K u (dd^cu)^{n-a-1}\pi_1^*[A]\wedge\pi_2^*[A]\wedge \omega^{n-a+1}
	\end{align}	
	is finite on every compact $K\subset D^2$. By Proposition \ref{prop:Kmc}, $\left\langle \pi_1^*[A]\wedge\pi_2^*[A]\wedge(dd^c u)^{n-a}\right\rangle_\Kmc$ exists. By Proposition \ref{prop:main_general}, we have 
	\begin{displaymath}
		(\pi_\Delta)_*(T_{A, \infty}\wedge \Omega_{n-a})=(\pi_\Delta)_*\left(\mathbf{1}_\Delta\left\langle \pi_1^*[A]\wedge\pi_2^*[A]\wedge(dd^c u)^{n-a}\right\rangle_\Kmc\right).
	\end{displaymath}
	The support of $\pi_1^*[A]\wedge\pi_2^*[A]\wedge\Kmc_{\theta}^{n-a}$ sits inside $A\times A$. Note that based on the coordinates for $\overline{E}$ in our consideration, $D^2\cap \Delta$ and $D$ are identical. Hence, we may say that the support of $(\pi_\Delta)_*\left(\mathbf{1}_\Delta\left\langle \pi_1^*[A]\wedge\pi_2^*[A]\wedge(dd^c u)^{n-1}\right\rangle_\Kmc\right)$
	sits inside $A$. Since $(\pi_\Delta)_*(T_{A, \infty}\wedge \Omega_{n-a})$ is a positive closed $(n-a, n-a)$-current and its support lies in the analytic subset $A$ of dimension $n-a$, by the support theorem of Siu, it should be a linear combination of the currents of integration on irreducible components of $A\cap D$.\medskip
	
	We prove that the coefficients equal $1$. For this, by shrinking $D$ to a smaller open subset of $D$, we may assume that $A\cap D$ is irreducible and in the regular part of $A$. From the boundedness of the integrals in \eqref{eq:int_self_intersecting}, we can use Theorem \ref{thm:shadow_general} to compute the shadow of tangent currents.
	We choose a local coordinate chart $x=(x', x'')$ of $D$, so that $x'=(x_1, \ldots, x_{n-a})$, $x''=(x_{n-1+1}, \ldots, x_n)$ and $A=\{x''=0\}$ in $D$. We also choose a similar local coordinate chart $(x, y)=(x', x'', y', y'')$ of $D^2$. We denote by $z=(z', z'')$, $z'=x'-y'$ and $z''=x''-y''$. We use $u=\frac{1}{2}\log \sum \left(|z'|^2+|z''|^2\right)$ and $u'= \log |z'|$ on $D^2$.\medskip
	
	We claim that the support of $\left\langle \pi_1^*[A]\wedge\pi_2^*[A]\wedge(dd^c u)^{n-a}\right\rangle_\Kmc$ sits inside $\Delta$. Let $\Phi$ be a smooth test form in $D^2$, whose support do not intersect $\Delta$. Over $\supp\, \Phi$, $dd^c u$ is smooth and we have $u^\Kmc_\theta = u$ when $|\theta|\ll 1$. We have
	\begin{align*}
		&\int_{D^2} \left\langle \pi_1^*[A]\wedge\pi_2^*[A]\wedge(dd^c u)^{n-a}\right\rangle_\Kmc\wedge\Phi = \int_{D^2}  \pi_1^*[A]\wedge\pi_2^*[A]\wedge (dd^c u)^{n-a}\wedge\Phi\\
		&\quad\quad\quad=\frac{1}{2^{n-a}}\int_{D^2}  \pi_1^*[A]\wedge\pi_2^*[A]\wedge (dd^c \log \sum \left(|z'|^2+|z''|^2\right))^{n-a}\wedge\Phi\\
		&\quad\quad\quad=\frac{1}{2^{n-a}}\int_{A^2} (dd^c \log \sum \left(|z'|^2\right))^{n-a}\wedge\Phi|_{A^2}.
	\end{align*}
	The product $(dd^c \log \sum \left(|z'|^2\right))^{n-a}$ vanishes from direct computations. Indeed, King's residue formula says that $(dd^c \log \sum \left(|z'|^2\right))^{n-a}$ equals the Dirac mass at $z'=0$ on $\C^{n-a}$. As $\supp\, \Phi$ does not intersect $z'=0$, the integral vanishes as desired.\medskip
	
	Let $\varphi$ be a smooth test $(n-a, n-a)$-form on $D$ or equivalently on $\Delta$. In the integral below, we are integrating outside $\Delta$. Thus, we have
	\begin{align*}
		&\left\langle(\pi_\Delta)_*(T_{A, \infty}\wedge \Omega_{n-a}), \varphi\right\rangle=\left\langle (\pi_\Delta)_*\left\langle \pi_1^*[A]\wedge\pi_2^*[A]\wedge(dd^cu)^{n-a}\right\rangle_\Kmc, \varphi \right\rangle \\
		&\quad\quad\quad= \int_{D^2\setminus \Delta}u \pi_1^*[A]\wedge\pi_2^*[A]\wedge (dd^cu)^{n-a-1}\wedge dd^c(\pi_\Delta^*\varphi)\\
		&\quad\quad\quad=\int_{A^2\setminus \Delta} u' (dd^cu')^{n-a-1} dd^c\pi_\Delta^*\varphi|_{A^2\cap \Delta}=\int_{A^2} u' (dd^cu')^{n-a-1} dd^c\pi_\Delta^*\varphi|_{A^2\cap \Delta}=\langle [A], \varphi\rangle.
	\end{align*}
	The first equality is from the above support condition and Theorem \ref{thm:shadow_general}. The second equality comes from Theorem \ref{thm:shadow_general}. As previously, the second to last equality is from \cite[Theorem III.4.5]{Demailly} and \cite[Proposition III.4.9]{Demailly}. The last equality comes from King's residue formula. Hence, we see that the coefficient should be $1$ as desired.
\end{proof}

\begin{remark}
	Theorem \ref{thm:non-proper} holds for any complex manifolds with the same proof as long as tangent currents of $\pi_1^*[A]\wedge\pi_2^*[A]$ along $\Delta$ exist. For instance, the reader is referred to \cite{Vu}.
\end{remark}

\end{document}